\documentclass[graybox]{svmult}

\usepackage{mathptmx}       % selects Times Roman as basic font
\usepackage{helvet}         % selects Helvetica as sans-serif font
\usepackage{courier}        % selects Courier as typewriter font
\usepackage{type1cm}        % activate if the above 3 fonts are
\usepackage{makeidx}         % allows index generation
\usepackage{graphicx}        % standard LaTeX graphics tool
\usepackage{multicol}        % used for the two-column index
\usepackage[bottom]{footmisc}% places footnotes at page bottom

\usepackage{latexsym}
\usepackage{amsmath}
\usepackage{amssymb}
\usepackage[isolatin]{inputenc}
\usepackage{graphicx}
\usepackage[dvips]{epsfig}
\usepackage{mathrsfs} 
\usepackage{float}
\usepackage{enumerate}
\usepackage{enumitem}
\usepackage[titletoc]{appendix}
\newcommand{\R}{{\mathbb R}}
\newcommand{\N}{{\mathbb N}}
\newcommand{\eps}{\varepsilon}

\DeclareMathOperator*{\argmin}{arg\,min}

\newcommand{\funarrow}{\rightarrow}
\newcommand{\define}{\stackrel{\Delta}{=}}
\newcommand{\maxindex}{\hat{\iota}}
\newcommand{\set}[2]{\left\{ {#1} : {#2} \right\}}
\newcommand{\lebesgue}[2]{\mathcal{L}^{#1}\left({#2}\right)}

\newcommand{\vnorm}[1]{\left\| {#1} \right\|}
\renewcommand{\cal}[1]{\mathcal{#1}}

\renewcommand{\cal}[1]{\mathcal{#1}}

\makeindex             % used for the subject index
\begin{document}

\title*{Sparse Control of Multiagent Systems}
\author{Mattia Bongini and Massimo Fornasier}
\institute{Mattia Bongini \at Technische Universit\"at M\"unchen, Fakult\"at Mathematik, Boltzmannstra{\ss}e 3, D-85748 Garching, Germany, \email{mattia.bongini@ma.tum.de}
\and Massimo Fornasier \at Technische Universit\"at M\"unchen, Fakult\"at Mathematik, Boltzmannstra{\ss}e 3, D-85748 Garching, Germany \email{massimo.fornasier@ma.tum.de}}
\maketitle

%{\footnotesize \it 
%\noindent ``Psychohistory is the mathematical study of the reactions of human conglomerates in response to economic and social stimuli.'' }
%\begin{flushright} -- Isaac Asimov, Foundation\end{flushright}
%
%{\footnotesize \it \noindent ``One believes things because one has been conditioned to believe them.''}
%\begin{flushright} -- Aldous Huxley, Brave New World \end{flushright}
%
%{\footnotesize \it \noindent ``History merely repeats itself. It has all been done before. Nothing under the sun is truly new.''}
%\begin{flushright} -- Bible, Ecclesiastes 1:9 \end{flushright}

%\abstract*{Each chapter should be preceded by an abstract (10--15 lines long) that summarizes the content. The abstract will appear \textit{online} at \url{www.SpringerLink.com} and be available with unrestricted access. This allows unregistered users to read the abstract as a teaser for the complete chapter. As a general rule the abstracts will not appear in the printed version of your book unless it is the style of your particular book or that of the series to which your book belongs.
%Please use the 'starred' version of the new Springer \texttt{abstract} command for typesetting the text of the online abstracts (cf. source file of this chapter template \texttt{abstract}) and include them with the source files of your manuscript. Use the plain \texttt{abstract} command if the abstract is also to appear in the printed version of the book.}

\abstract*{In recent years, numerous studies have focused on the mathematical modeling of social dynamics, with \textit{self-organization}, i.e., the autonomous pattern formation, as the main driving concept. Usually, first or second order models are employed to reproduce, at least qualitatively, certain global patterns (such as bird flocking, milling schools of fish or queue formations in pedestrian flows, just to mention a few). It is, however, common experience that self-organization does not always spontaneously occur in a society. In this review chapter we aim to describe the limitations of decentralized controls in restoring certain desired configurations and to address the question of whether it is possible to \textit{externally} and \textit{parsimoniously} influence the dynamics to reach a given outcome. More specifically, we address the issue of finding the sparsest control strategy for finite agent-based models in order to lead the dynamics optimally towards a desired pattern.}

\abstract{In recent years, numerous studies have focused on the mathematical modeling of social dynamics, with \textit{self-organization}, i.e., the autonomous pattern formation, as the main driving concept. Usually, first or second order models are employed to reproduce, at least qualitatively, certain global patterns (such as bird flocking, milling schools of fish or queue formations in pedestrian flows, just to mention a few). It is, however, common experience that self-organization does not always spontaneously occur in a society. In this review chapter we aim to describe the limitations of decentralized controls in restoring certain desired configurations and to address the question of whether it is possible to \textit{externally} and \textit{parsimoniously} influence the dynamics to reach a given outcome. More specifically, we address the issue of finding the sparsest control strategy for finite agent-based models in order to lead the dynamics optimally towards a desired pattern.}

\section{Introduction}

The autonomous formation of patterns in multiagent dynamical systems is a fascinating phenomenon which has spawned an enormous wealth of interdisciplinary studies: from social and economic networks \cite{battiston,currarini2009economic}, passing through cell aggregation and motility \cite{camazineselforganization,kese70,KocWhi98,be07}, all the way to coordinated animal motion \cite{MR2507454,ChuDorMarBerCha07,cristiani2010modeling,couzinlaneformation,couzin2005N,CS,Niw94,PE99,ParVisGru02,Rom96,TonTu95,YEECBKMS09} and crowd dynamics \cite{albi2015invisible,cristiani2011multiscale,CucSmaZho04,MR2438215}. Beyond biology and sociology, the principles of self-organization in multiagent systems are employed in engineering and information science to produce cheap, resilient, and efficient squadrons of autonomous machines to perform predefined tasks \cite{arvin2014development} and to render swarms of animals \cite{reynolds1987flocks} and hair/fur textures in CGI animations \cite{pixarhair}. The scientific literature on the subject is vast and ever-growing: the interested reader may be addressed to \cite{bak2013nature,CCH13,cafotove10,viza12} and references therein for further insights on the topic.

A common feature of all those studies is that self-organization is the result of the superimposition of binary interactions between agents amplified by an accelerating feedback loop. This reinforcement process is necessary to give momentum to the multitude of feeble local interactions and to eventually let a global pattern appear. Typically, the strength of such interaction forces is a function of the ``social distance'' between agents: for instance, birds align with their closest neighbors \cite{parisi08} and people agree easier with those who already conform to their beliefs \cite{krause02}. Some of the forces of the system may be of cohesive type, i.e., they tend to reduce the distance between agents: whenever cohesive forces have a comparable strength at short and long range, we call these systems \textit{heterophilious}; if, instead, there is a long-range bias we speak of \textit{homophilious} societies \cite{motsch2014heterophilious}.
Heterophilious systems have a natural tendency to keep the trajectories of the agents inside a compact region, and therefore to exhibit stable asymptotic profiles, modeling the autonomous emergence of global patterns. On the other hand, self-organization in homophilious societies can be accomplished only conditionally to sufficiently high levels of initial coherence that allow the cohesive forces to keep the dynamics compact \cite{birdsofafeather}. Being such systems ubiquitous in real life (e.g., see \cite{kirman2007marginal}), it is legitimate to ask whether -- in case of lost cohesion -- %external controls with limited strength
additional forces acting on the agents of the system may restore stability and achieve pattern formation.

%Alongside the above centralized approach to control (in the sense that there is an external figure implementing the control on the system), also 

A first solution to facilitate self-organization is to consider \textit{decentralized control strategies}: these consist in assuming that each agent, besides being subjected to forces induced in a feedback manner by the rest of the population, follows an individual strategy to coordinate with the other agents. However, as it was clarified in \cite{bongini2015conditional}, even if we allow agents to self-steer towards consensus according to {\it additional} decentralized feedback rules computed with {\it local} information, their action results in general in a minor modification of the initial homophilious model, with no improvement in terms of promoting unconditional pattern formation. Hence, blindly insisting and believing on decentralized control is certainly fascinating, but rather wishful, as it does not secure self-organization.

Such additional forces may eventually be the result of an offline optimization among perfectly informed players: in this case we fall into the realm of Game Theory \cite{nash1950equilibrium,von2007theory}.
%Whenever such additional force is the result of an offline optimization among perfectly informed players, we fall into the realm of Game Theory \cite{von2007theory,nash1950equilibrium}. %mean-field games, introduced by Lasry and Lions \cite{lasry2007mean}, and the Nash Certainty Equivalence principle \cite{huang2003individual,NCM11}.
%The case of strategic interactions, where perfectly informed agents optimize their behavior to reach a given outcome, was developed by Lasry and Lions in mean-field games theory \cite{} and, independently, in the optimal control community under the name Nash Certainty Equivalence (NCE) within the work \cite{}.
Games without an external regulator model situations where it is assumed that an automatic tendency to reach ``correct'' equilibria exists, like the stock market. However, also in this case such an optimistic view of the dynamics is often frustrated by evidences of the convergence to suboptimal configurations \cite{hardin1968tragedy}, whence the need of an external figure controlling the evolution of the system.

For all these reasons, in the seminal papers \cite{caponigro2013sparsefake,caponigro2015sparse} \textit{external} controls with limited strength were considered to promote self-organization in multiagent systems.
%In the case of multiagent systems,
Notice that, in such situations, \textit{efficient} control strategies should %take advantage of the mutual dependencies between agents and
target only few individuals of the population, instead of squandering resources on the entire group at once: taking advantage of the mutual dependencies between the agents, they should trigger a ripple effect that would spread their influence to the whole system, thus indirectly controlling the rest of the agents.
The property of control strategies to target only a small fraction of the total population is known in the mathematical literature as \textit{sparsity} \cite{tao,donoho,fora10}. %The control not only needs to be sparse in the sense that it acts on few elements of the group of agents simultaneously, but it is also supposed to hold for a short time, in order to always intervene on the most problematic portion of the population and to lead to the emergence of patterns as fast as possible. 
The fundamental issue is the selection of the few agents to control: an effective criterion is to choose them as to maximize the decay rate of some Lyapunov functional associated to the stability of the desired pattern \cite{cohen1983absolute}.

As a paradigmatic case study, let us consider \textit{alignment models} \cite{CS,krause02}: these are dissipative systems where imitation is the dominant feedback mechanism and in which the emerging pattern is a state where agents are fully aligned, also called \textit{consensus}. For several of such models it has been proved that consensus emergence can be guaranteed regardless of the initial conditions of the system only if %a certain Lyapunov functional is below a given threshold
the alignment forces are sufficiently strong at far distance, see \cite{HaHaKim,haskovec2014note}; in case they are not, it is easy to provide counterexamples to the emergence of a consensus. If we were to use the criterion above to select a control strategy to steer the system to consensus, it would lead to a sparse control targeting at each instant only the agent farthest away from the mean consensus parameter. Surprisingly enough, for such systems not only this strategy works for every initial condition, but the control of the instantaneous leaders of the dynamics is more convenient than controlling simultaneously all agents. Therefore if, on the one side, the homophilious character of a society plays against its compactness, on the other side, it may plays at its advantage if we allow for sparse interventions to restore consensus.

The above results have more far-reaching potential as they can be extended to non-dissipative systems as well, like the Cucker-Dong model of attraction and repulsion \cite{cucker14}. 
In this model, agents autonomously organize themselves in a \textit{cohesive and collision-avoiding} configuration provided that the total energy is below a certain level. %a given Lyapunov functional is inside a certain region, which is a function of the initial conditions. 
The sparse control strategy is able to raise this level considerably and it is optimal in maximizing the convergence of the energy functional towards it. However in this case, due to the singular non-conservative forces in play, it may be seen that sparse controllability is in general conditional to the choice of the initial conditions, as opposed to the unconditional controllability of alignment models.

%Notice that, according to this result and differently from the sparsely controlled Cucker and Smale system, the initial conditions need to fulfil a second threshold (11) for (9) to be sparsely controlled by a sparse feedback control as in Definition 3.9. While we do not dispose presently of results, which state the sparse uncon- trollability of the system if E(0) > ??, one can easily provide counterexamples showing that in general no sparse control as in Definition 3.9 (for any M > 0!) is able to stabilize the system if E(0) > ??. This suggests that sparse controllability is in general conditional to the choice of the initial conditions

%\cite{polyak2014sparse} sparse controls for linear systems were studied.

The essential scope of this review chapter is to describe in more detail the aforementioned mechanisms relating sparse controllability and pattern formation. We do so by condensing the results of the papers \cite{bofo13,bongini2014sparse,bongini2015conditional,caponigro2015sparse}, addressing the limitations of decentralized control strategies, the sparse controllability of alignment models and the one of attraction repulsion models.

\section{Self-organization in dynamical communication networks}\label{sec:alignmentchap1}

We start from the analysis of general properties of \textit{alignment models}. %, i.e., systems where imitation is the dominant feedback mechanism.
Instances of these models are ubiquitous in nature since several species are able to interpret and instinctively reproduce certain manoeuvres that they perceive (e.g., fleeing from a danger, searching for food, performing defense tactics, etc.), see \cite{animalbehavior}.
%We first derive sufficient conditions for their self-organization, which shows that only a certain balance between the strength of the alignment forces and the initial coherence of the group is able to guarantee the formation of a \textit{consensus}, interpreted in this case as a state where all agents are aligned. We then introduce decentralized feedback controls to facilitate the emergence of such patterns. The results reported in this chapter will make clear how prone to lose their initial coherence these systems are, and how much an \textit{external} control to help them organize is needed.
%In many natural and social phenomena, a group of agents faces the problem of coordinating on the basis of mutual communication. The modeling of such scenarios has to comply with the nature of protocols and customs governing the interaction among agents, limited or unreliable information transmission, and changing interaction topologies. Despite the above complications, %the framework provided by Graph Theory provides both a powerful language and a highly robust machinery to analyze these problems. Indeed,
Such systems may be seen as networks of agents with oriented information flow under possible link failure or creation, and can be effectively represented by means of directed graphs with edges possibly switching in time. %Therefore, we commence with a brief recall of some basic notions and terminology of this discipline.

A \textit{directed graph} $G$ on a set of nodes $A_1, \ldots, A_N$ is any subset of %the product set
$\{A_1, \ldots, A_N\}^2$. Each pair $(A, B) \in G$ is called an \textit{edge from $A$ to $B$}, and a \textit{directed path from $A$ to $B$} in $G$ is a sequence of edges $(A, A_{i_1}),(A_{i_1},A_{i_2}), \ldots, (A_{i_k},B) \in G$. The graph $G$ is said to be \textit{strongly connected} if for any pair $A,B$ of distinct nodes there is a directed path from $A$ to $B$ and a directed path from $B$ to $A$.

When studying under which conditions networks of agents are able to self-organize, it is usually not enough to know if two nodes are connected: the strength of the interaction between them also matters. Hence, given a system of $N \in \N$ agents, for each pair of agents $i,j = 1, \ldots, N$ we denote by $g_{ij}(t) \in \R_+$ the weight of the link connecting $i$ with $j$: clearly, if $g_{ij}(t) = 0$, $i$ is not connected to $j$ at time $t$. The value $g_{ij}(t)$ can be seen as the relative intensity of the information exchange flowing from agent $i$ to agent $j$ at time $t \geq 0$. We shall assume for the moment that each weight function $g_{ij}: \R_+ \funarrow \R_+$ is piecewise continuous.

The weights $g_{ij}(\cdot)$ naturally induce a directed graph structure on the set of agents: we define, for any $\eps \geq 0$ and $t \geq 0$, the graph $G_{\eps}(t)$ as
\begin{align*}
G_{\eps}(t) \define \set{(i,j) \in \{1,\ldots, N\}^2}{g_{ij}(t) > \eps}.
\end{align*}
The \textit{adjacency matrix} $G_0(t)$ is the set of pairs $(i,j)$ for which the communication channel from $i$ to $j$ is active at time $t$.

%Come prototipo di sistema multiagente e per illustrare in modo quantitativo il concetto di self-organization, introduciamo un modello di allineamento
%We now move to describe the specific multiagent systems we are going to analyze in order to study self-organization.
As a prototypical example of a multiagent system and to quantitatively illustrate the concept of self-organization, we introduce \textit{alignment models}: if we denote by $\{v_1, \ldots, v_N\} \subset \R^d$ the states of the $N$ agents of our systems, then the instantaneous evolution of the state $v_i(t)$ of agent $i$ at time $t$ is given by %the \textit{alignment model}
\begin{align}\label{eq:graphdyn}
\dot{v}_i(t) = \sum^N_{j = 1} g_{ij}(t)(v_j(t) - v_i(t)), \quad i = 1,\ldots,N.
\end{align}
The meaning of the above system of differential equations is the following: at each instant $t \geq 0$, the state $v_i(t)$ of agent $i$ tends to the state $v_j(t)$ of agent $j$ with a speed that depends on the strength of the information exchange $g_{ij}(t)$.
%As already mentioned, our networks of interest are made up of agents which try to coordinate their state (denoted as $\{v_1, \ldots, v_N\} \subset \R^d$) with that of the others. 
Since \eqref{eq:graphdyn} is a system of ODEs with possibly discontinuous coefficients, we need for it a proper notion of solution.

\begin{definition}
Let $\{I_k\}_{k \in \N}$ denote a countable family of open intervals such that all the functions $g_{ij}$ are continuous on every $I_k$ and $\cup_{k \in \N} \overline{I}_k = \R_+$.
Given $v^0 = (v^0_1, \ldots, v^0_N) \in \R^{dN}$, we say that the curve $v = (v_1,\ldots,v_N): \R_+ \funarrow \R^{dN}$ is a \textit{solution} of \eqref{eq:graphdyn} with initial datum $v^0$ if
\begin{enumerate}[label=$(\roman*)$]
\item $v(0) = v^0$;
\item for every $i = 1, \ldots, N$ and $k \in \N$, $v_i$ satisfies \eqref{eq:graphdyn} on $I_k$.
\end{enumerate}
\end{definition}

The notion of self-organization that we are considering for system \eqref{eq:graphdyn} is that of \textit{consensus} or \textit{flocking}, which is the situation where the state variables of the agents asymptotically coincide.

\begin{definition}[Consensus for system \eqref{eq:graphdyn}]\label{def:consensusV}
Let $v: \R_+ \funarrow \R^{dN}$ denote a solution of \eqref{eq:graphdyn} with initial datum $v^0$. We say that $v(\cdot)$ \textit{converges to consensus} if there exists a $v^{\infty} \in \R^d$ such that, for every $i = 1, \ldots, N$, it holds
\begin{align*}
\lim_{t \rightarrow +\infty} \|v_i(t) - v^{\infty}\|_{\ell^d_2} = 0.
\end{align*}
The value $v^{\infty}$ is called the \textit{consensus state}.
\end{definition}

In the definition above, $\|\cdot\|_{\ell^d_2}$ stands for the Euclidean norm on $\R^d$. The subscript $\ell^d_2$ shall often be omitted whenever clear from context.

Roughly speaking, a system of agents satisfying \eqref{eq:graphdyn} converges to consensus regardless of the initial condition $v^0$ provided that the underlying communication graph is ``sufficiently connected''. With this we mean that each node must possess, over some dense collection of time intervals, a strong enough communication path to every other node in the network. This intuitive idea is made precise in the following result, whose proof can be found in \cite{haskovec2014note}. A similar answer for discrete-time systems was also provided in \cite{moreau2005stability}

\begin{theorem}\label{th:convexhull}
Let $v: \R_+ \funarrow \R^{dN}$ be a solution of \eqref{eq:graphdyn} with initial datum $v^0$. Suppose that there exists an $\eps > 0$ and a strongly connected directed graph $G$ on the set of agents on which the system spends an infinite amount of time, i.e.,
\begin{align*}
\lebesgue{1}{\set{t \geq 0}{G_{\eps}(t) = G}} = +\infty.
\end{align*}
Then $v(\cdot)$ converges to consensus with consensus state $v^{\infty}$ belonging to the convex hull of $\{v^0_1, \ldots, v^0_N\}$.
\end{theorem}

The above result is closely related, for instance, to \cite[Theorem 2.3]{motsch2014heterophilious}, which requires a stronger connectivity of the network of agents (the quantity $\eta_A$ in \cite[Equation (2.5)]{motsch2014heterophilious}) but also gives an explicit rate for the convergence towards $v^{\infty}$ (see \cite[Equation (2.6b)]{motsch2014heterophilious}).

%As pointed out also in \cite{haskovec2014note},
Theorem \ref{th:convexhull} also says that, without further hypotheses on the interaction weights $g_{ij}$, the value of $v^{\infty}$ is rather an emergent property of the global dynamics of system \eqref{eq:graphdyn} than a mere function of the initial datum $v^0$. Nonetheless, it is relatively simple to identify assumptions on $g_{ij}$ for which the latter is true. For example, from a trivial computation follows
\begin{align*}
\frac{1}{N}\sum^N_{i = 1} \dot{v}_i(t) = \frac{1}{N}\sum^N_{j = 1} \left(\sum^N_{i = 1} g_{ij}(t) - \sum^N_{i = 1} g_{ji}(t)\right) v_j(t).
\end{align*}
Hence, if for every $t \geq 0$ the weight matrix $(g_{ij}(t))_{i,j = 1}^N$ has the property that $\sum^N_{i = 1} g_{ij}(t) = \sum^N_{i = 1} g_{ji}(t)$ for every $j = 1, \ldots, N$, then the average
\begin{align} \label{eq:meanvel}
\overline{v}(t) \define \frac{1}{N}\sum^N_{i = 1} v_i(t)
\end{align}
is an invariant of the dynamics. This implies that $v^{\infty} = \overline{v}(0)$ holds, i.e., the consensus state is only a function of the initial datum $v^0$.

\section{Consensus emergence in alignment models}\label{sec:alignmentexample}

In this section we shall see that the assumptions of Theorem \ref{th:convexhull} can actually be very restrictive and seldom met when dealing with specific instances of alignment models.

\subsection{Some classic examples of alignment models}

A general principle in opinion formation is the \textit{conformity bias}, i.e., agents weight more opinions that already conform to their beliefs. This can, actually, be extended to coordination in general, since intuitively it is easier to coordinate with ``near'' agents than ``far away'' ones. Formally, this is equivalent to asking that the weights $g_{ij}$ are a nonincreasing function of the distance between the states of the agents, i.e.,
\begin{align}\label{eq:weight}
g_{ij}(t) = a(\|v_i(t) - v_j(t)\|),
\end{align}
where $a:\R_+ \funarrow \R_+$ is a nonincreasing \textit{interaction kernel}. Notice that \eqref{eq:weight} trivially implies the invariance of the mean $\overline{v}$ (given by \eqref{eq:meanvel}), and that $v^{\infty} = \overline{v}(0)$, if it exists.

Several classic opinion formation models combine conformity bias with alignment. In the DW model, see \cite{weisbuch}, two random agents $i$ and $j$ update their opinions $v_i$ and $v_j$ to $1/2(v_i + v_j)$, provided they originally satisfy $\|v_i - v_j\| \leq R$, where $R > 0$ is fixed \textit{a priori}. Instead, in the popular \textit{bounded confidence} model of Hegselmann and Krause \cite{krause02}, opinions evolves according to the dynamics \eqref{eq:graphdyn} where the function $a$ has the form %the weights have the form \eqref{eq:weight} with
\begin{align*}
a(r) = \chi_{[0,R]}(r) \define \begin{cases}
1 &\text{ if } r \in [0,R],\\
0 &\text{ otherwise,}\\
\end{cases}
\end{align*}
for some fixed \textit{confidence radius} $R > 0$. The dynamics is thus given by the system of ODEs
\begin{align} \label{eq:krauseflock}
\dot{v}_i(t) = \frac{1}{|\Lambda_R(t,i)|}\sum^N_{j = 1} \chi_{[0,R]}(\|v_i(t) - v_j(t)\|)(v_j(t) - v_i(t)), \quad i = 1,\ldots, N,
\end{align}
where we have set
\begin{align}\label{eq:neighborset}
\Lambda_R(t,i) \define \set{j \in \{1,\ldots,N\}}{\|x_i(t) - x_j(t)\|\leq R},
\end{align}
and $|\Lambda_R(t,i)|$ stands for its cardinality. It is straightforward to design an instance of this model not fulfilling the hypothesis of Theorem \ref{th:convexhull}. Indeed, consider a group of $N = 2$ agents in dimension $d = 1$ with initial conditions $v_1(0) = -R$ and $v_2(0) = R$. Since $g_{12}(0) = g_{21}(0) = 0$, it follows that $G_{\eps}(t) = \emptyset$ for all $t \geq 0$ and for all $\eps \geq 0$.

%When the motion of agents is due to a force field, rather than a velocity field, a second-order model is necessary in order to capture the inertia-based evolution of the system under investigation.
Second-order models are necessary whenever we want to describe the dynamics of physical agents, like flocks of birds, herds of quadrupeds, schools of fish, and colonies of bacteria, where individuals are considered aligned whenever they move in the same direction, regardless of their position. Since in such cases it is necessary to perceive the velocities of the others in order to align, %The need to perceive the velocities of the others prescribes that, in such cases,
to describe the motion of the agents we need the pair position-velocity $(x,v)$, but this time only the velocity variable $v$ is the \textit{consensus parameter}. %just to name a few, since they capture the intrinsic inertial evolution of such system.

One of the first of such models, named \textit{Vicsek's model} in honor of one of its fathers, was introduced in \cite{vicsek1995novel}. Very much in the spirit of \eqref{eq:krauseflock}, it postulates that the evolution of the spatial coordinate $x_i$ and of the orientation $\theta_i \in [0,2\pi]$ in the plane $\R^2$ of the $i$-th agent follows the law of motion given by
\begin{align}
\left\{
\begin{aligned}
\begin{split}\label{eq:vicsek1}
\dot{x}_{i}(t) & = v_i(t) = \hat v \left(\begin{matrix}\cos(\theta_i(t))\\\sin(\theta_i(t))\end{matrix}\right), \\
\dot{\theta}_{i}(t) & = \frac{1}{|\Lambda_R(t,i)|} \sum_{j = 1}^N \chi_{[0,R]}\left(\|x_i(t) - x_j(t)\|\right)\left(\theta_{j}(t)-\theta_{i}(t)\right),
\end{split}\quad i = 1, \ldots, N,
\end{aligned}
\right.
\end{align}
where $\hat v > 0$ denotes the constant modulus of $v_i(t)$.%, we have set
%\begin{align}\label{eq:neighborset}
%\Lambda_R(t,i) \define \set{j \in \{1,\ldots,N\}}{\|x_i(t) - x_j(t)\|\leq R},
%\end{align}
%and $|\Lambda_R(t,i)|$ stands for its cardinality.
In this model, the orientation of the consensus parameter $v_i$ %(which, in a physical interpretation, may stand for the velocity of the $i$-th agent)
is adjusted with respect to the other agents according to a weighted average of the differences $\theta_j - \theta_i$. The influence of the $j$-th agent on the dynamics of the $i$-th one is a function of the (physical or social) distance between the two agents: if this distance is less than $R$, the agents interact by appearing in the computation of the respective future orientation.
\begin{figure}[!htb]
\centering
\includegraphics[width = 0.49\textwidth]{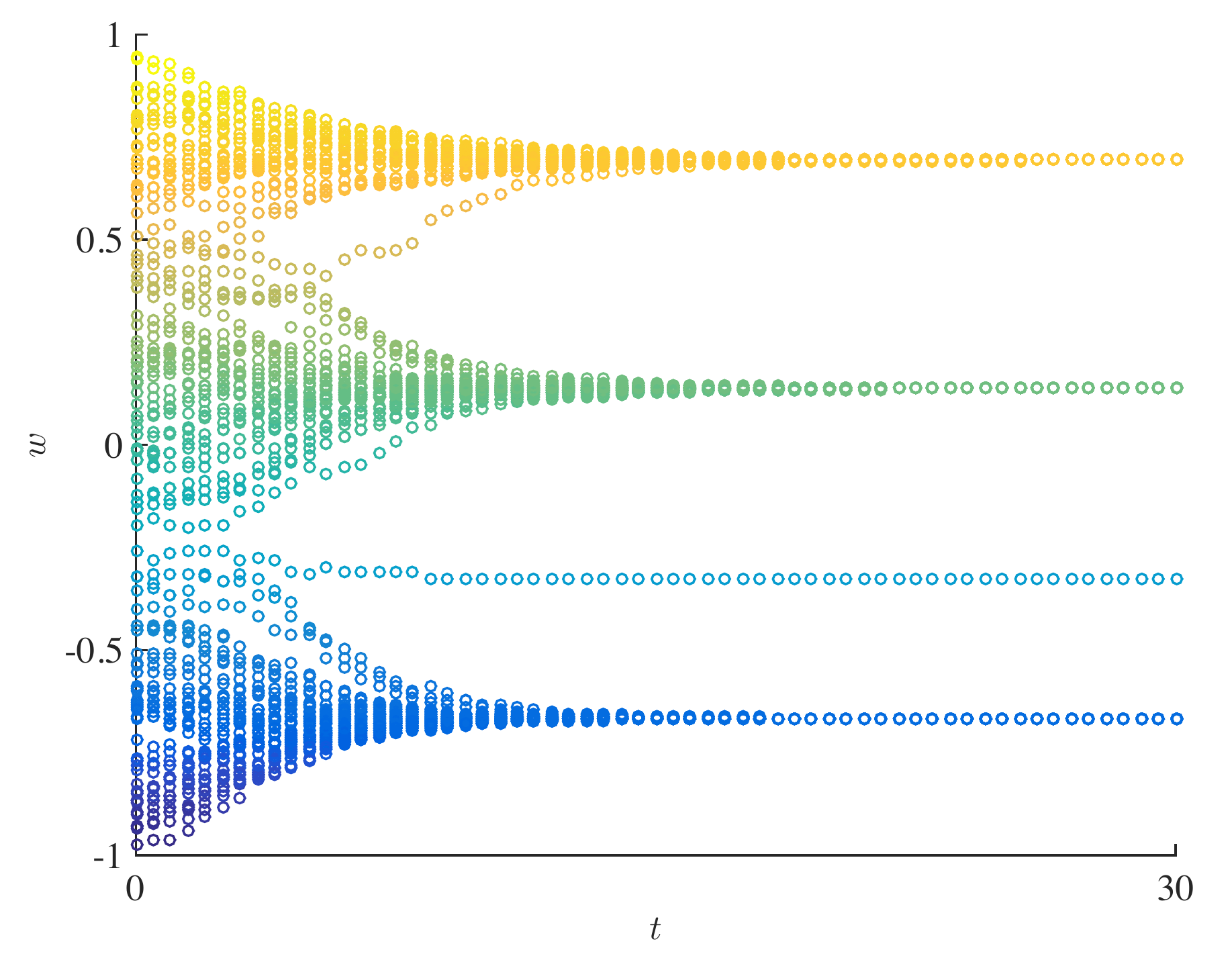}
\includegraphics[width = 0.49\textwidth]{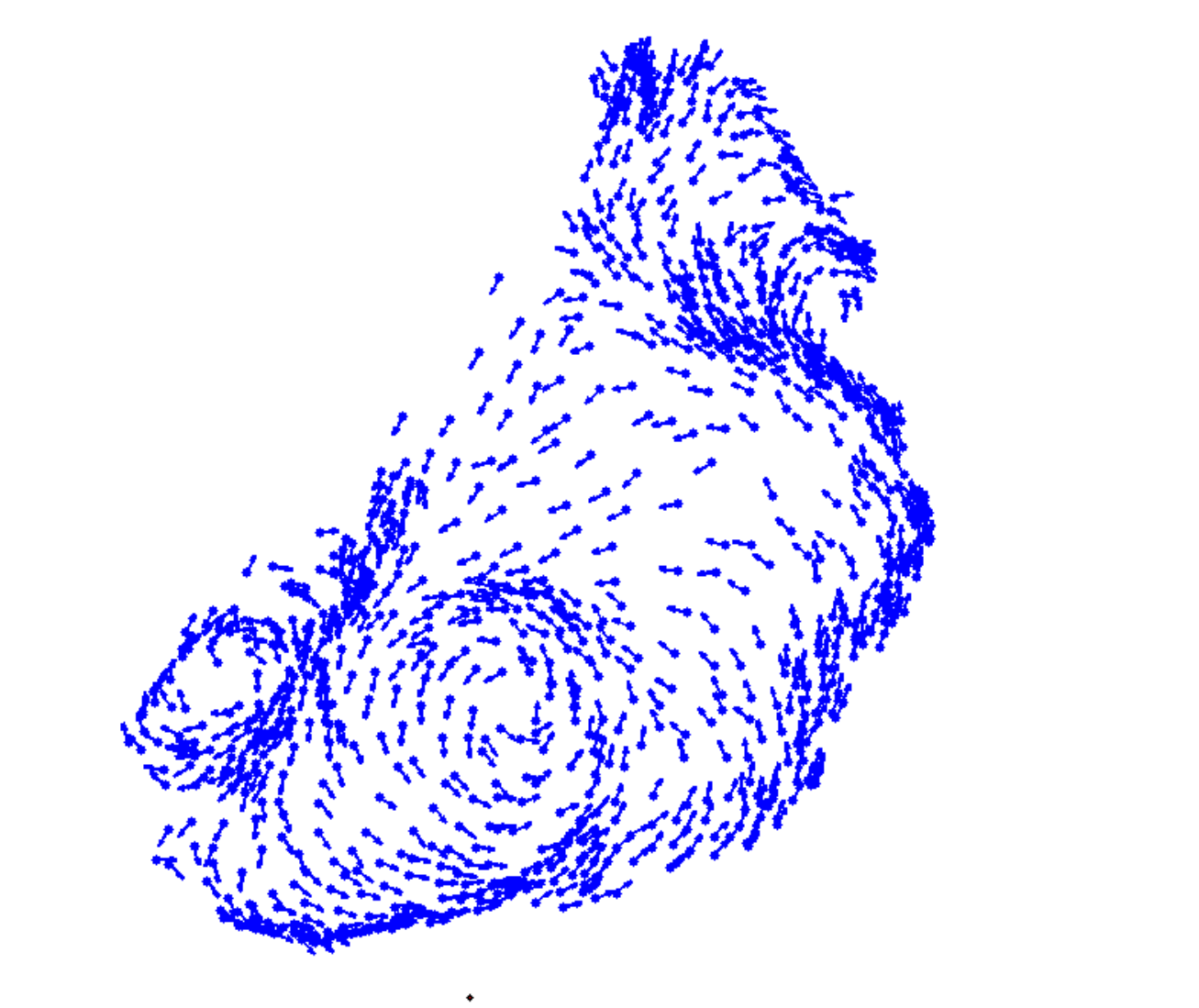}
\caption{On the left: a typical evolution of the Hegselmann-Krause model. On the right: mill patterns in the Vicsek model. (Kind courtesy of G. Albi)}
\end{figure}

In \cite{CS}, the authors proposed a possible extension of system \eqref{eq:vicsek1}  to dimensions $d > 2$ as follows
\begin{align*}
\left\{
\begin{aligned}
\begin{split}%\label{eq:vicsek}
\dot{x}_{i}(t) & = v_{i}(t), \\
\dot{v}_{i}(t) & = \frac{1}{|\Lambda_R(t,i)|} \sum_{j = 1}^N \chi_{[0,R]}\left(\|x_i(t) - x_j(t)\|\right)\left(v_{j}(t)-v_{i}(t)\right),
\end{split}\qquad i = 1, \ldots, N.
\end{aligned}
\right.
\end{align*}
The substitution of the function $\chi_{[0,R]}$ with a strictly positive \textit{kernel} $a:\R_+\funarrow\R_+$ let us drop the highly irregular and nonsymmetric normalizing factor $|\Lambda_R(t,i)|$ in favor of a simple $N$, and leads to the system
%A great leap forward in the study of these systems consisted in replacing the function $\chi_{[0,R]}$ by means of a \textit{strictly positive} function $a$, which let us drop the highly irregular normalizing factor $|\Lambda_R(t,i)|$ in favor of a simple $N$ factor. This modification, which leads to systems of the form
\begin{align}
\left\{
\begin{aligned}
\begin{split} \label{eq:cuckersmale}
\dot{x}_{i}(t) & = v_{i}(t), \\
\dot{v}_{i}(t) & = \frac{1}{N} \sum_{j = 1}^N a\left(\|x_i(t) - x_j(t)\|\right)\left(v_{j}(t)-v_{i}(t)\right),
\end{split}\qquad i = 1, \ldots, N.
\end{aligned}
\right.
\end{align}
Notice that the equation governing the evolution of $v_i$ has the same form as \eqref{eq:graphdyn}, and since now the weights $g_{ij}$ are symmetric (i.e., $g_{ij} = g_{ji}$ for all $i,j = 1, \ldots, N$) then $\overline{v}$ is a conserved quantity.
%does not alter significantly the nature of system \eqref{eq:vicsek}, and has the great advantage of being analytically easier to study, as will be further analyzed in Sections \ref{sec:cuckersmale_localmean} and \ref{sec:localmeanR}.

An example of a system of the form \eqref{eq:cuckersmale} is the influential model of Cucker and Smale, introduced in \cite{CS}, in which the function $a$ is
\begin{align}\label{eq:cuckerkernel}
a(r) \define \frac{H}{(\sigma^2 + r^2)^{\beta}},
\end{align}
where $H > 0$, $\sigma > 0$, and $\beta \geq 0$ are constants accounting for the social properties of the group. Systems like \eqref{eq:cuckersmale} are usually referred to as \textit{Cucker-Smale systems} due to the influence of their work, as can be witnessed by the wealth of literature focusing on their model, see for instance \cite{ahn2010stochastic,carrillo2010asymptotic,dalmao2011cucker,ha2009simple,perea2009extension,shen2007cucker}.

%The above list of models is clearly incomplete for the sake of space. We refer the interested reader to the papers \cite{} and the reviews \cite{} for other classes of models.

\subsection{Pattern formation for the Cucker-Smale model}

We now focus on consensus emergence for system \eqref{eq:cuckersmale}. In the following, we shall consider a kernel $a:\R_+\funarrow\R_+$ which is decreasing, strictly positive, bounded and Lipschitz continuous.

As already noticed, in second-order models alignment means that all agents move with the same velocity, but \textit{not necessarily} are in the same position. Therefore, Definition \ref{def:consensusV} of consensus applies here on the $v_i$ variables only. %We restate it for the convenience of the reader.
\begin{definition}[Consensus for system \eqref{eq:cuckersmale}] \label{def:consensus}
We say that a solution
\begin{align*}
(x,v) = (x_1,\ldots,x_N, v_1, \ldots, v_N):\R_+ \rightarrow \R^{2dN}
\end{align*}
of system \eqref{eq:cuckersmale} \textit{tends to consensus} if the consensus parameter vectors $v_i$ tend to the mean $\overline{v}$, i.e.,
\begin{align*}
\lim_{t \rightarrow + \infty} \vnorm{v_i(t) - \overline{v}(t)} = 0 \quad \text{for every } i = 1,\ldots,N.
\end{align*}
\end{definition}

\begin{figure}[!htb]
\centering
\includegraphics[width = 0.49\textwidth]{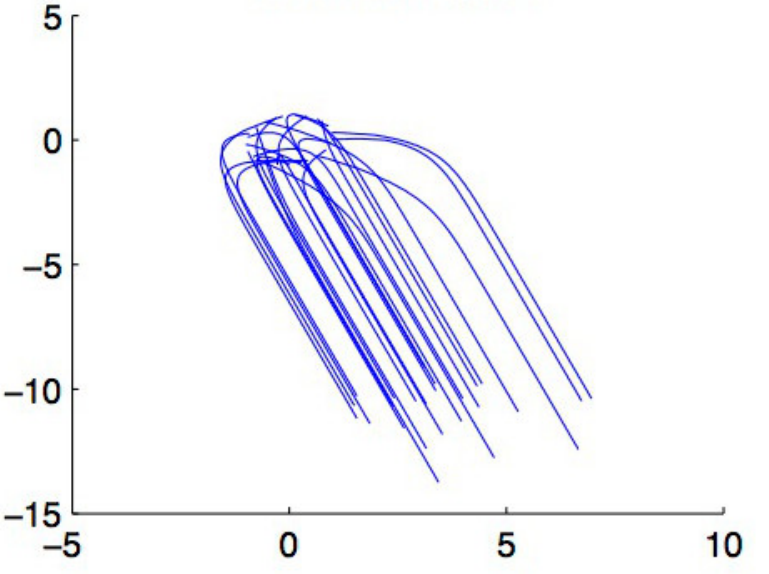}
\includegraphics[width = 0.49\textwidth]{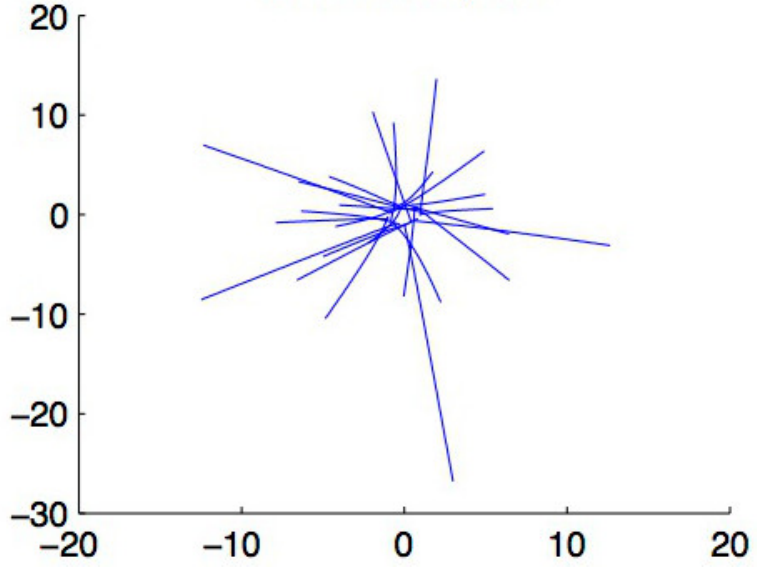}
\caption{Consensus behavior of a Cucker-Smale system. On the left: agents align with the mean velocity. On the right: agents fail to reach consensus.}
\end{figure}

The following result is an easy corollary of Theorem \ref{th:convexhull}.

\begin{corollary}\label{cor:unifbound}
Let $(x(\cdot),v(\cdot))$ be a solution of system \eqref{eq:cuckersmale}, where the interaction kernel $a$ is decreasing and strictly positive. Suppose that there exists $R > 0$ for which it holds
\begin{align*}
\lebesgue{1}{\set{t \geq 0}{\vnorm{x_i(t) - x_j(t)} \leq R \; \text{ for all } \; i,j = 1, \ldots, N}} = +\infty.
\end{align*}
Then $(x(\cdot),v(\cdot))$ converges to consensus.
\end{corollary}
\begin{proof}
%First of all, since $a$ is positive, then $g_{ij}(t) = a(\vnorm{x_i(t) - x_j(t)}) > 0$ for all $t \geq 0$, hence $G_{\eps}(t) = \{1, \ldots, N\}^2$ for every $\eps \geq 0$ and $t \geq 0$. Moreover, 
Since $a$ is decreasing and strictly positive, from the initial assumptions follows
\begin{align*}
g_{ij}(t) = \frac{1}{N}a(\vnorm{x_i(t) - x_j(t)}) \geq \frac{a(R)}{N} > 0,
\end{align*}
for every $t \geq 0$ for which $\vnorm{x_i(t) - x_j(t)} \leq R$ holds for every $i,j = 1, \ldots, N$. Therefore, the condition $\vnorm{x_i(t) - x_j(t)} \leq R$ for every $i,j = 1, \ldots, N$ implies
$G_{a(R)/N}(t) =  \{1, \ldots, N\}^2,$
which yields
\begin{align*}
\lebesgue{1}{\set{t \geq 0}{G_{a(R)/N}(t) =  \{1, \ldots, N\}^2}} = +\infty,
\end{align*}
The statement then follows from Theorem \ref{th:convexhull} for the choice $\eps = a(R)/N$.
\end{proof}

Unfortunately, the result above has the serious flaw that it cannot be invoked directly to infer convergence to consensus, since establishing a uniform bound in time for the distances of the agents is very difficult, even for smooth kernels like \eqref{eq:cuckerkernel}. Intuitively, consider the case where the interaction strength is too weak and the agents too dispersed in space to let the velocities $v_i$ align. In this case, nothing prevents the distances $\|x_i - x_j\|$ to grow indefinitely, violating the hypothesis of Corollary \ref{cor:unifbound}. Hence, in order to obtain more satisfactory consensus results, we need to follow approaches that take into account the extra information at our disposal, which are the strength of the interaction and the initial configuration of the system.

Originally, this problem was studied in \cite{CS,cusm07} borrowing several tools from Spectral Graph Theory, see as a reference \cite{chung1997spectral}. Indeed, system \eqref{eq:cuckersmale} can be rewritten in the following compact form
\begin{align}\label{systemmatrix}
\left\{
\begin{aligned}
\dot{x}(t) & = v(t), \\
\dot{v}(t) & = L(x(t))v(t),
\end{aligned}\right.
\end{align}
where $L(x(t))$ is the Laplacian\footnote{Given a real $N\times N$ matrix $A = (a_{ij})_{i,j = 1}^N$ and $v\in \R^{dN}$ we denote by $Av$ the action of $A$ on $\R^{dN}$  by mapping $v$ to $(a_{i1}v_{1} + \cdots + a_{iN}v_{N})_{i=1}^N$. Given a nonnegative symmetric $N \times N$ matrix $A = (a_{ij})_{i,j = 1}^N$, the \textit{Laplacian} $L$ of $A$ is defined by $L = D - A$, with $D = \mathrm{diag} (d_{1}, \ldots, d_{N})$ and $d_{k} = \sum_{j=1}^{N} a_{kj}$.} of the matrix $(a\left(\vnorm{x_{i}(t)-x_{j}(t)}\right)/N)_{i,j=1}^N$, which is a function of $x(t)$. Being the Laplacian of a positive definite, symmetric matrix, $L(x(t))$ encodes plenty of information regarding the adjacency matrix $G_0(t)$ of the system, see \cite{mohar1991laplacian}. In particular, the second smallest eigenvalue $\lambda_2(t)$ of $L(x(t))$, called the \textit{Fiedler's number} of $G_0(t)$ is deeply linked with consensus emergence: provided that a sufficiently strong bound from below of $\lambda_2(t)$ is available, the system converges to consensus.

To establish under which conditions we have convergence to consensus, we shall follow a different approach. The advantage of it is that it can be employed also to study the issue of the controllability of several multiagent systems (see Section \ref{sec:cuckersparse}).%, as well as their dimensionality reduction (see Chapter \ref{ch:dimred}).

\subsection{The consensus region}

A natural strategy to improve Corollary \ref{cor:unifbound} would be to look for quantities which are invariant with respect to $\overline{v}$, since it is conserved in systems like \eqref{eq:cuckersmale}.

\begin{definition}
The symmetric bilinear form $B: \R^{dN} \times \R^{dN} \funarrow \R$ is defined, for any $v,w \in \R^{dN}$, as
\begin{align*}
B(v,w) \define \frac{1}{2N^2} \sum^N_{i = 1} \sum^N_{j = 1} (v_i - v_j)\cdot(w_i - w_j),
\end{align*}
where $\cdot$ denotes the usual scalar product on $\R^d$.
\end{definition}

\begin{remark}
It is trivial to prove that
\begin{align}\label{eq:otherB}
B(v,w) = \frac{1}{N} \sum^N_{i = 1} (v_i\cdot w_i) - \overline{v} \cdot \overline{w},
\end{align}
where $\overline{v}$ stands for the average of the elements of the vector $v = (v_1, \ldots, v_N)$ given by \eqref{eq:meanvel}. From this representation of $B$ follows easily that the two spaces
\begin{align*}
\cal{V}_f &\define \set{v \in \R^{dN}}{v_1 = \ldots = v_N} \quad \text{and} \quad \cal{V}_{\perp} \define \set{v \in \R^{dN}}{\sum^N_{i = 1} v_i = 0},
\end{align*}
are perpendicular with respect to the scalar product $B$, i.e., $\R^{dN} = \cal{V}_f \oplus \cal{V}_{\perp}$. This means that every $v \in \R^{dN}$ can be written uniquely as $v = v^f + v^{\perp}$, where $v^f \in \cal{V}_f$ and $v^{\perp} \in \cal{V}_{\perp}$. A closer inspection reveals that it holds
$v^f_i = \overline{v}$ and $v^{\perp}_i =  v_i - \overline{v}$ for every $i = 1, \ldots, N$.
Notice that, since $v^{\perp} \in \cal{V}_{\perp}$, for any vector $w \in \R^d$ it holds
\begin{align} \label{eq:vertequalzero}
\sum^N_{i = 1} (v^{\perp}_i \cdot w) = \left(\sum^N_{i = 1}v^{\perp}_i \right) \cdot w = 0.
\end{align}
\end{remark}

Since for every $v,w \in \R^{dN}$ we have $B(v^f,w) = 0 = B(v,w^f)$, it holds
\begin{align*}
B(v,w) = B(v^{\perp},w) = B(v,w^{\perp}) = B(v^{\perp},w^{\perp}).
\end{align*}
This means that $B$ distinguishes two vectors modulo their projection on $\cal{V}_f$. Moreover, from \eqref{eq:otherB} immediately follows that $B$ restricted to $\cal{V}_{\perp} \times \cal{V}_{\perp}$ coincides, up to a factor $1/N$, with the usual scalar product on $\R^{dN}$.

\begin{remark}[Consensus manifold]
Notice that whenever the initial datum $(x^0,v^0)$ belongs to the set $\R^{dN}\times\cal{V}_f$, the right-hand size of $\dot{v}_i$ in \eqref{eq:cuckersmale} is 0, hence the equality $v_1(t) = \ldots = v_N(t)$ is satisfied for all $t\geq 0$ and the system is already in consensus. For this reason, the set $\R^{dN}\times\cal{V}_f$ is called the \textit{consensus manifold}.
\end{remark}

The bilinear form $B$ can be used to characterize consensus emergence for solutions $(x(\cdot),v(\cdot))$ of system \eqref{eq:cuckersmale} by setting
\begin{align*}
X(t) \define B(x(t), x(t)) \quad \text{ and } \quad V(t) \define B(v(t),v(t)).
\end{align*}
The functionals $X$ and $V$ provide a description of consensus by measuring the spread, both in positions and velocities, of the trajectories of the solution $(x(\cdot),v(\cdot))$, as the following trivial result shows. %In Corollary \ref{cor:trivialVdecay} we shall see that actually $V$ is a Lyapunov functional for system \eqref{eq:cuckersmale}

\begin{proposition}\label{p:equivconsensus}
The following statements are equivalent:
\begin{enumerate}
\item $\lim_{t \rightarrow + \infty} \vnorm{v_i(t) - \overline{v}(t)} = 0$ for every $i = 1,\ldots,N$;
\item $\lim_{t \rightarrow + \infty} v_i^{\perp}(t) = 0$ for every $i = 1,\ldots,N$;
\item $\lim_{t \rightarrow + \infty} V(t) = 0$.
\end{enumerate}
\end{proposition}

The following Lemma shows that $V$ is a Lyapunov functional for system \eqref{eq:cuckersmale}.

\begin{lemma}[{\cite[Lemma 1]{caponigro2015sparse}}]\label{cor:trivialVdecay}
Let $(x(\cdot), v(\cdot))$ be a solution of system \eqref{eq:cuckersmale}. Then for every $t \geq 0$ it holds
\begin{align}\label{eq:Vdecayvanilla}
\frac{d}{dt}V(t) \leq -2 a(\sqrt{2NX(t)})V(t).
\end{align}
Therefore, $V$ is decreasing. %Moreover, if there exists $\overline{X}>0$ such that $X(t)\leq\overline{X}$ for every $t \geq 0$, then $(x(\cdot), v(\cdot))$ converges to consensus.
\end{lemma}
%\begin{proof}
%The only implication which is not immediate is ($ii$)$\Rightarrow$($iii$), for which it suffices to notice that it holds
%\begin{align*}
%V(t) = \frac{1}{N} \sum^N_{i = 1} \vnorm{v^{\perp}_i(t)}^2.
%\end{align*}
%This concludes the proof.
%\end{proof}
%
%\begin{remark}\label{rem:decay}
%From Proposition \ref{p:equivconsensus} follows that, in order to establish that a solution of \eqref{eq:cuckersmale} tends to consensus, a possible strategy would be to prove that the Lyapunov functional $V$ has a sufficiently strong decay. For this purpose, the following computation shall be often exploited: given a matrix $\omega \in \R^{N\times N}$ which is symmetric and with positive entries, i.e., for every $i, j = 1, \ldots, N$ it holds
%\begin{align*}
%\omega_{ij} = \omega_{ji} \quad \text{ and } \quad \omega_{ij} > 0,
%\end{align*}
%for any $v \in \R^{dN}$ we have
%\begin{align} \label{eq:maintrick}
%\frac{1}{N^2} \sum^N_{i = 1} \sum^N_{j = 1} \omega_{ij} (v_j - v_i)\cdot v_i & = \frac{1}{2 N^2} \Bigg( \sum^N_{i = 1} \sum^N_{j = 1} \omega_{ij} (v_j - v_i)\cdot v_i + \sum^N_{j = 1} \sum^N_{i = 1} \omega_{ji} (v_i - v_j)\cdot v_j \Bigg) \nonumber\\
%& = -\frac{1}{2 N^2}\sum^N_{j = 1} \sum^N_{i = 1} \omega_{ij} \vnorm{v_i - v_j}^2 \nonumber\\
%& \leq - \min_{1 \leq i,j \leq N} \omega_{ij} \frac{1}{N} \sum^N_{i = 1}\vnorm{v_i^{\perp}}^2.
%\end{align}
%\end{remark}

By means of the quantities $X$ and $V$ we can provide a sufficient condition for consensus emergence for solutions of system \eqref{eq:cuckersmale}.

\begin{theorem}[{\cite[Theorem 3.1]{HaHaKim}}]\label{thm:hhk}
Let $(x^0,v^0) \in \R^{dN} \times \R^{dN}$ and set $X_0 \define B(x^0, x^0)$ and $V_0 \define B(v^0, v^0)$. If the following inequality is satisfied
\begin{align}
\label{eq:HaHaKim}
\int^{+\infty}_{\sqrt{X_0}} a(\sqrt{2N} r) \ dr \geq \sqrt{V_0},
\end{align}
then the solution of \eqref{eq:cuckersmale} with initial datum $(x^0,v^0)$ tends to consensus.
\end{theorem}

The inequality \eqref{eq:HaHaKim} defines a region in the space $(X_0,V_0)$ of initial conditions for which the balance between $X_0$, $V_0$ and the kernel $a$ is such that the system tends to consensus autonomously.

%, which will lead to consensus. We call this set consensus region. If the rate of communication function a is integrable, i.e., far distant agents are influencing very weakly the dynamics, i.e., they system has a homophilious regime, then such a region is essentially bounded, and actually not all initial conditions will realize self-organization, as the following example shows.

%We shall prove this result in Section \ref{sec:localmeanR}, although stated in a slightly different fashion (see Theorem \ref{th:HaHaKimExtended}).
%The meaning of \eqref{eq:HaHaKim} is that, as soon as there is a good balance between $X_0$, $V_0$ and the kernel $a$ (i.e., the agents are neither too dispersed nor too unaligned with respect to the strength of the mutual interaction), then the system tends to consensus autonomously.

\begin{definition}[Consensus region] We call \textit{consensus region} the set of points $(X_0,V_0)\in\R^{dN} \times \R^{dN}$ satisfying \eqref{eq:HaHaKim}.
\end{definition}

The size of the consensus region gives an estimate of how large the basin of attraction of the consensus manifold $\R^{dN} \times \cal{V}_f$ is. If the rate of communication function $a$ is integrable, i.e., far distant agents are only weakly influencing the dynamics, then such a region is essentially bounded, and actually not all initial conditions will realize self-organization, as the following example shows.%This estimate is sharp in some cases, as the following example shows.

\begin{example}[{\cite[Proposition 5]{CS}}]\label{ex:notsharp}
Consider $N = 2$ agents in dimension $d = 1$ subject to system \eqref{eq:cuckersmale} with interaction kernel given by \eqref{eq:cuckerkernel} with $H = 1/2$, $\sigma = 1$, and $\beta = 1$. If we denote by $(x_1(\cdot), v_1(\cdot))$ and $(x_2(\cdot), v_2(\cdot))$ the trajectories of the two agents, it is easy to show that the evolution of the relative main state $x(t) \define x_1(t) - x_2(t)$ and of the relative consensus state $v(t) \define v_1(t) - v_2(t)$ is given for every $t \geq 0$ by
\begin{align}\label{eq:cuckersharp}
\left\{
\begin{aligned}
\begin{split}
\dot{x}(t) & = v(t),  \\
\dot{v}(t)  & = -\frac{v(t) }{1+x(t) ^2},
\end{split}
\end{aligned}
\right.
\end{align}
with initial condition $x(0) = x^0$ and $v(0) = v^0$ (without loss of generality, we may assume that $x^0, v^0 > 0$). An explicit solution of the above system can be easily derived by means of direct integration:
\begin{align*}
v(t) - v^0 = -\arctan x(t) + \arctan x^0.
\end{align*}
Condition \eqref{eq:HaHaKim} in this case reads $\pi/2 - \arctan x^0 \geq v^0.$ Hence, suppose \eqref{eq:HaHaKim} is violated, i.e., $\arctan x^0 + v^0 > \pi/2$. This means $\arctan x^0 + v^0 \geq \pi/2 + \eps$ for some $\eps > 0$, which implies
\begin{align*}
|v(t)| = |-\arctan x(t) + \arctan x^0 + v^0| \geq \left|-\arctan x(t) + \frac{\pi}{2} + \eps\right| > \eps
\end{align*}
for every $t \geq 0$. Therefore, the solution of system \eqref{eq:cuckersharp} with initial datum $(x^0,v^0)$ satisfying $\arctan x^0 + v^0 > \pi/2$ does not converge to consensus, since otherwise we would have $v(t) \rightarrow 0$ for $t \rightarrow +\infty$.
\end{example}

\begin{remark}\label{rem:ainL1}
Notice that, if  $\int_{\delta}^{+\infty} a(r) dr$ diverges for every $\delta \geq 0$, then the consensus region coincides with the entire space $\R^{dN} \times \R^{dN}$. In other words, in this case the interaction force between the agents is so strong that the system will reach consensus no matter what the initial conditions are.
\end{remark}

%The example below shows that Cucker-Smale systems can converge to consensus even if the hypothesis of Theorem \ref{thm:hhk} is not satisfied.
As the following example shows, there may be initial configurations from which the system can reach consensus automatically even if condition \eqref{eq:HaHaKim} is not satisfied.

\begin{example}\label{ex:autonomous}
Consider an instance of the Cucker-Smale system \eqref{eq:cuckersmale} without control in dimension $d=1$ with $N=2$ agents, where the interaction function $a:\R_+\funarrow\R_+$ is of the form
\begin{align*}
a(r) = \begin{cases}
M & \text{ if } r\leq R,\\
f(r) & \text{ if } r \geq R,
\end{cases}
\end{align*}
for some given $R > 0$ and $f:\R_+\funarrow\R_+$ positive continuous function satisfying
$$f(R)=M \quad \text{ and } \quad \int^{+\infty}_R f(r) \, dr = \eps < + \infty.$$
The constant $M>0$ is to be properly chosen later on. Assume that the initial state and consensus parameters of the two agents are $(x^0_1,v^0_1) = (-R/2,v^0)$ and $(x^0_2,v^0_2) = (R/2,-v^0)$ respectively, for some $v^0 > \eps/2$.

Due to the nature of the situation, is fairly easy to check if condition \eqref{eq:HaHaKim} of Theorem \ref{thm:hhk} is satisfied or not. Indeed we have $X(0) = R^2/4$ and $V(0) = (v^0)^2$, and, by the particular form of $a$, after a change of variables the computation below follows
\begin{align*}
\int^{+\infty}_{\frac{R}{2}} a(2r) \,dr = \frac{1}{2}\int^{+\infty}_{R} a(r) \,dr = \frac{1}{2}\int^{+\infty}_{R} f(r) \,dr = \frac{\eps}{2}.
\end{align*}
Therefore at time $t = 0$ we are not in the consensus region given by \eqref{eq:HaHaKim}, since
\begin{align*}
\int^{+\infty}_{\sqrt{X(0)}} a(\sqrt{4}r) \,dr  = \frac{\eps}{2} < v^0 = \sqrt{V(0)}.
\end{align*}
We now show that there exists a time $T >0$ such that
\begin{equation}\label{eq:entracazzo}
\int^{+\infty}_{\sqrt{X(T)}} a(\sqrt{4}r) \,dr  \geq \sqrt{V(T)},
\end{equation}
i.e., the system enters the consensus region autonomously at time $T$.

To do so, we first compute a lower bound for the integral. Notice that, since we are considering a Cucker-Smale system with mean consensus parameter $\overline{v} = 0$, the speeds $|v_1(t)|$ and $|v_2(t)|$ are decreasing by Lemma \ref{cor:trivialVdecay}. Therefore, we can estimate from above the time until $|x_1(t)-x_2(t)| \leq R$ holds by $T^* \define R/2v^0$ (since the agents are moving on the real line in opposite directions). Hence $X(t)\leq X(0) = R^2/4$ for every $t\in[0,T^*]$, which yields the following lower bound
\begin{equation*}%\label{eq:boundfacilefacile}
\int^{+\infty}_{\sqrt{X(t)}} a(\sqrt{4}r) \,dr \geq \int^{+\infty}_{\sqrt{X(0)}} a(\sqrt{4}r) \,dr = \frac{\eps}{2}
\end{equation*}
valid for any $t \leq T^*$.

We now compute an upper bound for the functional $\sqrt{V(t)}$ for $t\in [0,T^*]$. Notice that
$$a(\sqrt{4X(t)}) \geq a(\sqrt{4X(0)}) = a(R) = M,$$
hence by \eqref{eq:Vdecayvanilla} we have
$$\frac{d}{dt}V(t) \leq - 2MV(t)$$
which, by integration, implies that $\sqrt{V(t)}\leq v^0e^{-Mt}$ for every $t\in [0,T^*]$.

We now plug together the two bounds. In order for \eqref{eq:entracazzo} to hold at some time $T < T^*$, simply choose
$$M = M_T \define \frac{1}{T}\log\left(\frac{2v^0}{\eps}\right).$$
For this choice of $M$, it follows
$$\int^{+\infty}_{\sqrt{X(T)}} a(\sqrt{4}r) \, dr \geq \frac{\eps}{2} = v^0e^{-M_T T} \geq \sqrt{V(T)}.$$
From Theorem \ref{thm:hhk} we can then conclude that any solution of the above system tends autonomously to consensus.
\end{example}

\section{The effect of perturbations on consensus emergence} \label{sec:first_results}

An immediate way to enhance the alignment capabilities of systems like \eqref{eq:cuckersmale} consists in adding a feedback term penalizing the distance of each agent's velocity from the average one, i.e.,
\begin{align}
\left\{
\begin{aligned}
\begin{split} \label{eq:cuckersmale_uniform}
\dot{x}_{i}(t) & = v_{i}(t), \\
\dot{v}_{i}(t) & = \frac{1}{N} \sum_{j = 1}^N a\left(\vnorm{x_i(t) - x_j(t)}\right)\left(v_{j}(t)-v_{i}(t)\right) + \gamma(\overline{v}(t) - v_i(t)),
\end{split}
\end{aligned}
\right.
\end{align}
where $\gamma > 0$ is a prescribed constant, modeling the strength of the additional alignment term.

This approach to the enforcement of consensus is a particular instance of what in the literature is known as \textit{decentralized control strategy}, which 
%well-known in literature as \textit{decentralized control strategy} and
has been thoroughly studied especially for its application in the self-organization of \textit{unmanned aerial vehicles} (UAVs) \cite{fax2004information}, congestion control in communication networks \cite{paganini2001scalable}, and distributed sensor newtorks \cite{cortes2005coordination}. We also refer to \cite{tanner2007flocking} for the stability analysis of a decentralized coordination method for dynamical systems with switching underlying communication network.

As system \eqref{eq:cuckersmale_uniform} can be rewritten as \eqref{eq:cuckersmale} with the interaction kernel $a(\cdot) + \gamma$ replacing $a(\cdot)$, by Theorem \ref{thm:hhk} and Remark \ref{rem:ainL1} each solution of \eqref{eq:cuckersmale_uniform} tends to consensus.

However, the apparently innocent fix of adding the extra term above has actually a huge impact on the interpretation of the model: as pointed out in \cite{caponigro2015sparse}, this approach requires that each agent must possess at every instant a \textit{perfect information} of the whole system, since it has to correctly compute the mean velocity of the group $\overline{v}$ in order to compute its trajectory. This condition is seldom met in real-life situations, where it is usually only possible to ask that each agent computes an approximated mean velocity vector $\overline{v}_i$, instead of the true $\overline{v}$. These considerations lead us to the model
\begin{align}
\left\{
\begin{aligned}
\begin{split} \label{eq:cuckersmale_local}
\dot{x}_{i}(t) & = v_{i}(t), \\
\dot{v}_{i}(t) & = \frac{1}{N} \sum_{j = 1}^N a\left(\vnorm{x_i(t) - x_j(t)}\right)\left(v_{j}(t)-v_{i}(t)\right) + \gamma(\overline{v}_i(t) - v_i(t)).
\end{split}
\end{aligned}
\right.
\end{align}

In studying under which conditions the solutions of system \eqref{eq:cuckersmale_local} tend to consensus, it is often desirable to express the approximated feedback as a combination of a term consisting on a \textit{true information feedback}, i.e., a feedback based on the real average $\overline{v}$, and a perturbation term, which models the deviation of $\overline{v}_i$ from $\overline{v}$. To this end, we rewrite system \eqref{eq:cuckersmale_local} in the following form:
\begin{align}
\left\{
\begin{aligned}
\begin{split} \label{eq:cuckersmale_perturbed}
\dot{x}_{i}(t) & = v_{i}(t), \\
\dot{v}_{i}(t) & = \frac{1}{N} \sum_{j = 1}^N a\left(\vnorm{x_i(t) - x_j(t)}\right)\left(v_{j}(t)-v_{i}(t)\right) + \alpha(t)(\overline{v}(t) - v_i(t)) + \beta(t) \Delta_i(t),
\end{split}
\end{aligned}
\right.
\end{align}
where $\alpha(\cdot)$ and $\beta(\cdot)$ are two nonnegative, piecewise continuous functions, and $\Delta_i(\cdot)$ is the deviation acting on the estimate of $\overline{v}$ by agent $i$ (which can, of course, depend on $(x_1(t),\ldots,x_N(t), v_1(t),\ldots,v_N(t))$). Therefore, solutions in this context have to be understood in terms of weak solutions in the Carath\'eodory sense, see \cite{filipov}. %For a quick reference, see Appendix \ref{ch:cara}.

\begin{remark}\label{rm:exuniqchap1}
In what follows, we will not be interested in the well-posedness of system \eqref{eq:cuckersmale_perturbed}, but rather in finding assumptions on the functions $a$, $\alpha$, $\beta$, and $\Delta_i$ for which we can guarantee its asymptotic convergence to consensus.
%To identify the additional conditions on $a$, $\alpha$, $\beta$ and $\Delta_i$ for which existence and/or uniqueness of solutions on a finite time horizon $[0,T]$ can be established for system \eqref{eq:cuckersmale_perturbed}, we need to check under which hypotheses we have that, for every $i = 1, \ldots, N$, the function
%\begin{align*}
%g_i(t, x, v) = \frac{1}{N} \sum_{j = 1}^N a\left(\vnorm{x_i - x_j}\right)\left(v_{j}-v_{i}\right) + \alpha(t)\left(\frac{1}{N}\sum^N_{j = 1}v_j - v_i\right) + \beta(t) \Delta_i(t)
%\end{align*}
%satisfies the hypotheses of Theorems \ref{cara-global} and/or \ref{le:uniquecara}. In case they are both satisfied for every $T \geq 0$, Remark \ref{caraext} tells us how we can obtain a unique solution of the system defined on the entire half line $\R_+$.
\end{remark}

System \eqref{eq:cuckersmale_perturbed} provides the advantage of encompassing all the previously introduced models, as can be readily seen:
\begin{itemize}
\item if $\alpha = \beta \equiv \gamma$ and $\Delta_i =  v_i - \overline{v}$, or $\alpha = \beta \equiv 0$, then we recover system \eqref{eq:cuckersmale},
\item the choices $\alpha \equiv \gamma$, $\Delta_i \equiv 0$ (or equivalently $\beta \equiv 0$) yield system \eqref{eq:cuckersmale_uniform},
\item if $\alpha = \beta \equiv \gamma$ and $\Delta_i = \overline{v}_i - \overline{v}$ we obtain system \eqref{eq:cuckersmale_local}.
\end{itemize}

The introduction of the perturbation term in system \eqref{eq:cuckersmale_perturbed} may deeply modify  the nature of the original model: for instance, an immediate consequence is that the mean velocity of the system is, in general, no longer a conserved quantity.

\begin{proposition} \label{prop:derivative_mean}
For system \eqref{eq:cuckersmale_perturbed}, with perturbations given by the vector-valued function $\Delta(\cdot) = (\Delta_1(\cdot), \ldots, \Delta_N(\cdot))$, for every $t \geq 0$ it holds
\begin{align*}
	\frac{d}{dt}\overline{v}(t) = \beta(t) \overline{\Delta}(t).
\end{align*}
\end{proposition}
%\begin{proof}
%A trivial computation shows that for every $t\geq0$ it holds
%\begin{align*}
%	\frac{d}{dt} \overline{v}(t) & =  \frac{1}{N} \sum^N_{i = 1} \dot{v}_i(t) \\
%	& =  \frac{1}{N} \sum^N_{i = 1} \left(\frac{1}{N} \sum_{j = 1}^N a\left(\vnorm{x_i(t) - x_j(t)}\right)\left(v_{j}(t)-v_{i}(t)\right) + \alpha(t)(\overline{v}(t) - v_i(t)) + \beta(t) \Delta_i(t) \right) \\
%	& = \underbrace{\frac{1}{N^2} \sum^N_{i = 1} \sum_{j = 1}^N a\left(\vnorm{x_i(t) - x_j(t)}\right)\left(v_{j}(t)-v_{i}(t)\right)}_{= 0 \text{, by simmetry.}} + \underbrace{\frac{\alpha(t)}{N} \sum^N_{i = 1} v^{\perp}_i(t)}_{= 0} + \frac{\beta(t)}{N} \sum^N_{i = 1} \Delta_i(t) \\
%	& = \beta(t) \overline{\Delta}(t).
%\end{align*}
%This proves the statement.
%\end{proof}

\begin{remark} \label{rem:derivative_mean}
As we have already pointed out, it is possible to recover system \eqref{eq:cuckersmale} by setting $\Delta_i =  v_i - \overline{v}$, whereas we can recover system \eqref{eq:cuckersmale_uniform} for the choice $\Delta_i \equiv  0$. Note that in both cases we have $\overline{\Delta}(t) =  0$ for every $t\geq0$, therefore the mean velocity is conserved both in systems \eqref{eq:cuckersmale} and \eqref{eq:cuckersmale_uniform}.

We also highlight the fact that $\overline{v}$ is not conserved even in the case that for every $t \geq 0,$ and for every $i = 1, \ldots, N$ we have $\Delta_i(t) = w$, where $w \in \R^d\backslash\{0\}$, i.e., the case in which all agents make the same mistake in evaluating the mean velocity.
\end{remark}

\subsection{General results for consensus stabilization under perturbations}
%
%As already noticed in Remark \ref{rem:decay}, a possible strategy for studying under which assumptions the solutions of system \eqref{eq:cuckersmale} tend to consensus is to obtain an estimate of the decay of the Lyapunov functional $V$. We shall follow a similar approach in order to study consensus emergence for system  \eqref{eq:cuckersmale_perturbed}. We begin by proving the following lemma.
%
The following is a generalization of Lemma \ref{cor:trivialVdecay} to systems like \eqref{eq:cuckersmale_perturbed}.

\begin{lemma}[{\cite[Lemma 3.1]{bongini2015conditional}}] \label{lem:bigv_growth}
Let $(x(\cdot), v(\cdot))$ be a solution of system \eqref{eq:cuckersmale_perturbed}. For every $t \geq 0$ it holds
\begin{align} \label{eq:maintool}
\frac{d}{dt} V(t) \leq - 2 a\left(\sqrt{2NX(t)}\right) V(t) -2 \alpha(t) V(t) + \frac{2 \beta(t)}{N} \sum^N_{i = 1} \Delta_i(t) \cdot v^{\perp}_i(t).
\end{align}
\end{lemma}
\begin{proof}
Differentiating $V$ for every $t \geq 0,$ we have
\begin{align*}
	\frac{d}{dt} V(t) = \frac{2}{N} \sum^N_{i = 1}  \frac{d}{dt}v^{\perp}_i(t)\cdot v^{\perp}_i(t) = \frac{2}{N} \sum^N_{i = 1} \frac{d}{dt}v_i(t)\cdot v^{\perp}_i(t)- \frac{2}{N} \sum^N_{i = 1} \frac{d}{dt}\overline{v}(t)\cdot v^{\perp}_i(t).
\end{align*}
Hence, inserting the expression for $\dot{v}_i(t)$, using the fact that $a$ is nonincreasing, and invoking Proposition \ref{prop:derivative_mean}, we get \eqref{eq:maintool}.
\end{proof}

%The following result is a trivial consequence of Lemma \ref{lem:bigv_growth}.
%
%\begin{corollary}\label{cor:trivialVdecay}
%Let $(x(\cdot), v(\cdot))$ be a solution of system \eqref{eq:cuckersmale}. Then $V$ is decreasing. %Moreover, if there exists $\overline{X}>0$ such that $X(t)\leq\overline{X}$ for every $t \geq 0$, then $(x(\cdot), v(\cdot))$ converges to consensus.
%\end{corollary}
%\begin{proof}
%The choice $\alpha = \beta \equiv 0$ in Lemma \ref{lem:bigv_growth} yields the decay estimate
%\begin{align}\label{eq:easydecayV}
%\frac{d}{dt} V(t) \leq - 2 a\left(\sqrt{2NX(t)}\right) V(t).
%\end{align}
%Assume $V(0) > 0$, otherwise we are already in consensus. By the continuity of the dynamics, there exists a $T > 0$ for which $V(t) > 0$ holds for every $t\in[0,T]$. Directly integrating \eqref{eq:easydecayV} gives
%\begin{align*}
%V(t) \leq V(0) e^{- 2\int^t_0a\left(\sqrt{2NX(s)}\right)ds}
%\end{align*}
%for any $t\in[0,T]$. Actually this estimate holds for any $t\geq 0$ since it is trivially satisfied whenever $V(t) = 0$. Since $a$ is strictly positive, it then follows that $V$ is decreasing. %If, in addition, the bound $X(t)\leq\overline{X}$ holds for every $t \geq 0$, then we have
%%\begin{align*}
%%V(t) \leq V(0) e^{- 2a\left(\sqrt{2N\overline{X}}\right)t},
%%\end{align*}
%%which implies that $V$ converges to 0 exponentially fast as $t\rightarrow+\infty$.
%\end{proof}

Since we are interested in the case where %the deviation
$\Delta_i$ plays an active role in the dynamics, in what follows we %shall
assume $\beta(t) > 0$ for all $t \geq 0$. %From Lemma \ref{lem:bigv_growth} we obtain the following
As a direct consequence of Lemma \ref{lem:bigv_growth} we get that, by controlling the magnitude of the deviations $\Delta_i$, we can establish the unconditional convergence to consensus.

\begin{theorem} \label{th:perpbound_convergence}
Let $(x(\cdot), v(\cdot))$ be a solution of system \eqref{eq:cuckersmale_perturbed}, and suppose that there exists a $T \geq 0$ such that for every $t \geq T$,
\begin{align} \label{eq:smallerror}
\sum^N_{i = 1} \Delta_i(t)\cdot v^{\perp}_i(t) \leq \phi(t) \sum^N_{i = 1} \vnorm{v^{\perp}_i(t)}^2
\end{align}
for some function $\phi:[T,+\infty) \funarrow [0,\ell]$, where
\begin{align}\label{goodell}
\ell<\frac{\min_{t \geq T}\alpha(t)}{\max_{t \geq T}\beta(t)}.
\end{align}
Then $(x(\cdot), v(\cdot))$ tends to consensus.
\end{theorem}
\begin{proof}
Under the assumption \eqref{eq:smallerror}, for every $t \geq T$ the upper bound in \eqref{eq:maintool} can be simplified to
\begin{align*}
%\frac{d}{dt}V(t) & \leq -2 \alpha(t) V(t) + \frac{2 \beta(t)}{N} \sum^N_{i = 1}\Delta_i(t) \cdot v^{\perp}_i(t) \\
%& \leq -2 \alpha(t) V(t) + \frac{2 \beta(t)}{N} \phi(t) \sum^N_{i = 1} \vnorm{v^{\perp}_i(t)}^2 \\
%& = -2 \alpha(t) V(t) + 2 \beta(t) \phi(t) V(t) \\
\frac{d}{dt}V(t) \leq 2 \beta(t) \left(\ell - \frac{\alpha(t)}{\beta(t)}\right) V(t).
\end{align*}
Integrating between $T$ and $t$ (where $t \geq T$) we get $V(t) \leq V(T) e^{2 \int^t_T \beta(s) \left(\ell - \frac{\alpha(s)}{\beta(s)}\right)ds},$ and as the factor $\ell - \alpha(s)/\beta(s)$ is negative while $\beta$ is nonnegative, $V$ approaches $0$ exponentially fast.
\end{proof}

We then immediately get the following

\begin{corollary} \label{cor:noperp_convergence}
If there exists $T \geq 0$ such that $\Delta^{\perp}_i(t) = 0$ for every $t \geq T$ and for every $1 \leq i \leq N$, then any solution of system \eqref{eq:cuckersmale_perturbed} tends to consensus.
\end{corollary}
\begin{proof}
Noting that $\Delta^{\perp}_i = 0$ implies $\Delta_i = \overline{\Delta}$, by \eqref{eq:vertequalzero} we have $\sum^N_{i = 1} \Delta_i(t)\cdot v^{\perp}_i(t) = \sum^N_{i = 1} \overline{\Delta} \cdot v^{\perp}_i(t) = 0.$ Hence, we can apply Theorem \ref{th:perpbound_convergence} with $\phi(t) = 0$ for every $t \geq T$ to obtain the result.
\end{proof}

\begin{remark}
A trivial implication of Corollary \ref{cor:noperp_convergence} is that any solution of system \eqref{eq:cuckersmale_uniform} tends to consensus (this was already a consequence of Theorem \ref{thm:hhk}), but has moreover a rather nontrivial implication: also any solution of systems subjected to \textit{deviated} uniform control, i.e., systems like \eqref{eq:cuckersmale_perturbed} where $\Delta_i(t) = \Delta(t)$ for every $i = 1,\ldots, N$ and for every $t \geq 0$, tends to consensus, because it holds
\begin{align*}
\begin{split}
\Delta^{\perp}_i(t) = \Delta_i(t) - \frac{1}{N}\sum^N_{j = 1} \Delta_j(t) = \Delta(t) - \Delta(t) = 0
\end{split}
\end{align*}
for every $i = 1,\ldots, N$ and for every $t \geq 0$, therefore Corollary \ref{cor:noperp_convergence} applies. This means that systems of this kind converge to consensus even if the agents have an incorrect knowledge of the mean velocity, provided they all make the same mistake. 
\end{remark}

Another consequence of the previous results is the following corollary, which provides an upper bound for tolerable perturbations under which consensus emergence can be unconditionally guaranteed.

\begin{corollary}
For every $i = 1, \ldots, N$, let $\varepsilon_i: \R_+ \funarrow [0,\ell]$ for $\ell > 0$ as in \eqref{goodell}. If there exists $T \geq 0$ such that $\vnorm{\Delta_i(t)} \leq \varepsilon_i(t) \vnorm{v_i^{\perp}(t)}$ for every $t \geq T$ and for every $i = 1,\ldots, N$, then any solution of system \eqref{eq:cuckersmale_perturbed} tends to consensus.
\end{corollary}
%\begin{proof}
%By using the Cauchy-Schwarz inequality we have
%\begin{align*}
%\sum^N_{i = 1} \Delta_i(t) \cdot v^{\perp}_i(t) & \leq \sum^N_{i = 1} \varepsilon_i(t) \vnorm{v^{\perp}_i(t)}^2 \\
%& \leq \ell \sum^N_{i = 1} \vnorm{v^{\perp}_i(t)}^2.
%\end{align*}
%The conclusion follows by taking $\phi(t) = \ell$ for every $t \geq T$ in Theorem \ref{th:perpbound_convergence}.
%\end{proof}
%
%The result above shows that, provided that the magnitude of the perturbation is smaller than the one of the deviation of the agent's velocity from the mean, then convergence to consensus is obtained unconditionally with respect to the initial condition. This is the case of local estimates of the average, as the largest error that an agent can make when estimating the group average upon a subset of agents is precisely its own deviation from the mean, $v^{\perp}_i$.

%\section{Perturbations as linear combinations of velocity deviations} \label{sec:cuckersmale_localmean}

\subsection{Perturbations as leader-based feedback} 

We now consider the problem of consensus stabilization based on a leader-following feedback.

\begin{example}\label{ex:followingtheleader}
Let us use Lemma \ref{lem:bigv_growth} to study the convergence to consensus of a system like \eqref{eq:cuckersmale_local}, where each agent computes its local mean velocity $\overline{v}_i$ by taking into account itself plus a single common agent $(x_1, v_1)$, which in turn takes into account only itself by computing $\overline{v}_1 = v_1$. Formally, given two finite conjugate exponents $p, q$ (i.e., two positive real numbers satisfying $1/p + 1/q = 1$), we assume that for any $i = 1, \ldots, N$ it holds
\begin{align*}
\overline{v}_i(t) = \frac{1}{p}v_i(t) + \frac{1}{q}v_1(t) \quad \text{ for every } t\geq 0.
\end{align*}
We shall prove that any solution of this system tends to consensus, no matter how small the positive weight $1/q$ of $v_1$ in $\overline{v}_i$ is. We start by writing the system under the form \eqref{eq:cuckersmale_perturbed}, with $\alpha(t) = \beta(t) = \gamma > 0$ and
\begin{align*}
\Delta_i(t) = \frac{1}{p}v^{\perp}_i(t) + \frac{1}{q}v^{\perp}_1(t)\quad \text{ for every } t\geq 0.
\end{align*}
Hence, the perturbation term in the estimate \eqref{eq:maintool} on the decay of $V$ becomes
\begin{align*}
	\frac{2 \gamma}{N}\sum^N_{i = 1} \Delta_i(t) \cdot v^{\perp}_i(t) & =  \frac{2 \gamma}{N}\sum^N_{i = 1} \frac{1}{p}(v^{\perp}_i(t) + \frac{1}{q}v^{\perp}_1(t))\cdot v^{\perp}_i(t)\\
	& = \frac{1}{p} \frac{2 \gamma}{N} \sum^N_{i = 1} \vnorm{v_i^{\perp}(t)}^2 + \frac{1}{q} \frac{2 \gamma}{N} v^{\perp}_1(t) \cdot \underbrace{\sum^N_{i = 1} v^{\perp}_i(t)}_{= 0} = \frac{2 \gamma}{p} V(t).
\end{align*}
Lemma \ref{lem:bigv_growth} let us bound the growth of $V$ as
\begin{align*}
	\frac{d}{dt}V(t) \leq 2 \gamma \left(-1 + \frac{1}{p}\right) V(t) = -\frac{2 \gamma}{q} V(t),
\end{align*}
which ensures the exponential decay of the functional $V$ for any $q > 0$.
\end{example}

%The first case that we address is the one presented in Example \ref{ex:followingtheleader}, where we consider a system like \eqref{eq:cuckersmale_local} where the local mean is computed upon local information and a single leader, i.e.,
%\begin{align*}
%\overline{v}_i(t)=(1-q)v_i(t)+q v_1(t) \quad \text{ for every } i=1,\ldots,N \text{ and } t\geq 0,
%\end{align*}
%where for convenience we have selected the first agent as the leader of the group.
Figure \ref{fig:1} shows the behavior of the group of agents considered in Example \ref{ex:followingtheleader} depending on the parameter $ q\in (0, 1]$, which represents the influence of the leader in the local average. The result above asserts that for every such $q$, the system will converge to consensus independently of the initial configuration, %which is
as illustrated %by our numerical experiments, as shown
in Figures \ref{fig:1} and \ref{fig:12}. It can be observed that, the weaker the influence of the leader, the longer the group of agents takes to align.%reach consensus.

\begin{figure}[!ht]
\centering
%\resizebox{\textwidth}{!}{
%\begin{tabular}{cc}
%\epsfig{file=q0.eps,width=0.49\linewidth,clip=}
%\epsfig{file=q1.eps,width=0.49\linewidth,clip=}\\
%\epsfig{file=q5.eps,width=0.49\linewidth,clip=}
%\epsfig{file=q9.eps,width=0.49\linewidth,clip=}
\includegraphics[width = 0.49\textwidth]{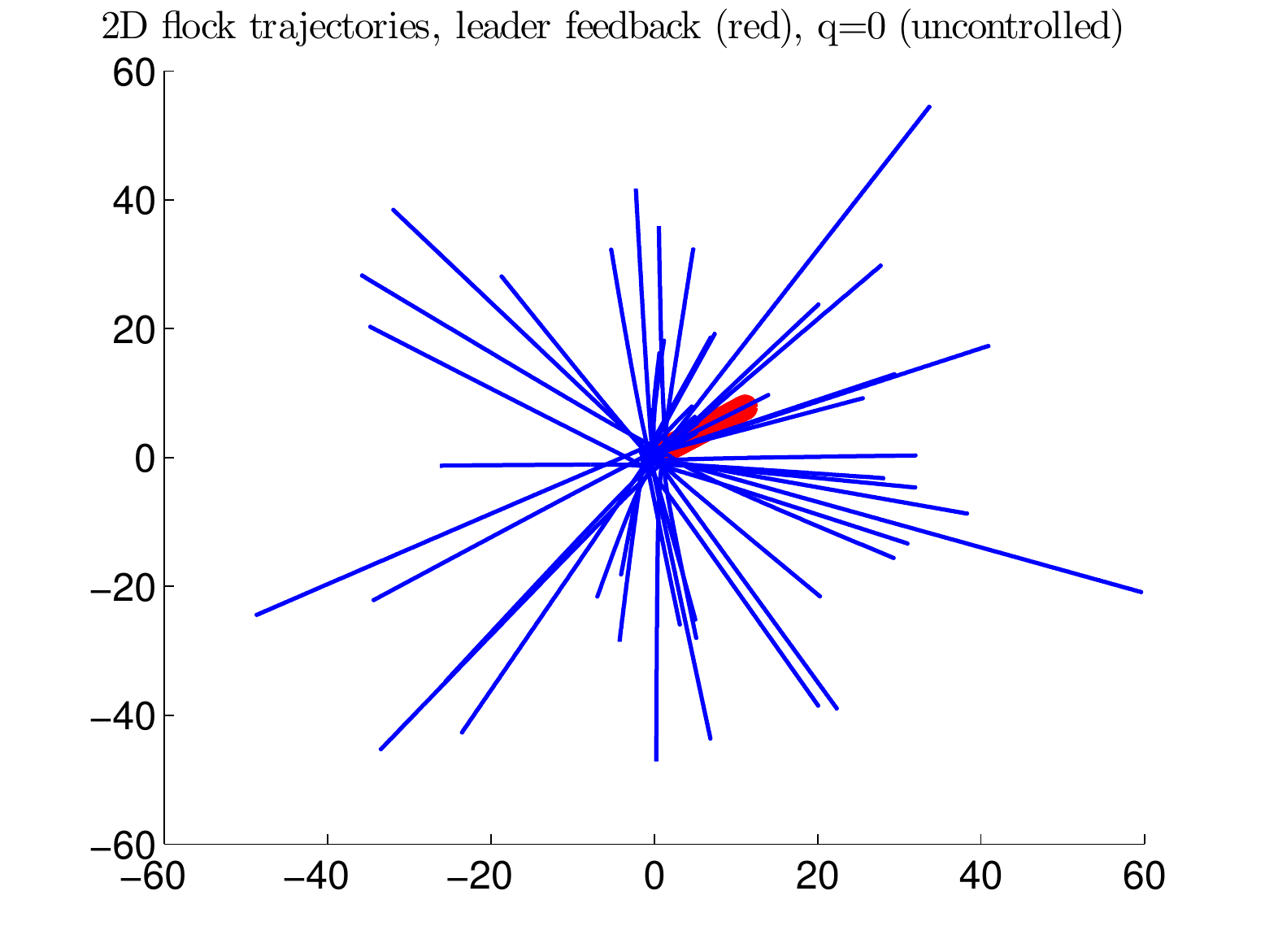}
\includegraphics[width = 0.49\textwidth]{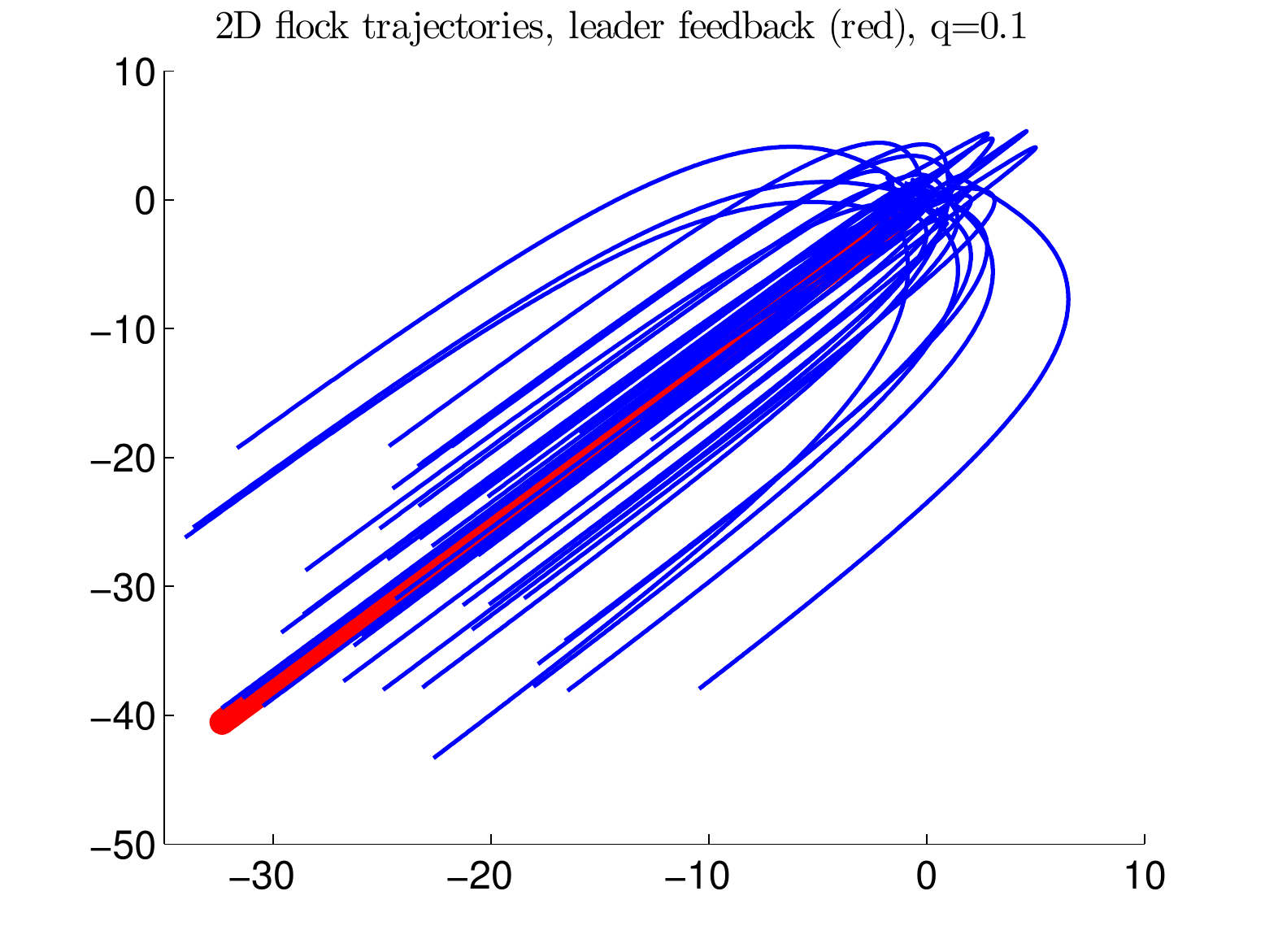}\\
\includegraphics[width = 0.49\textwidth]{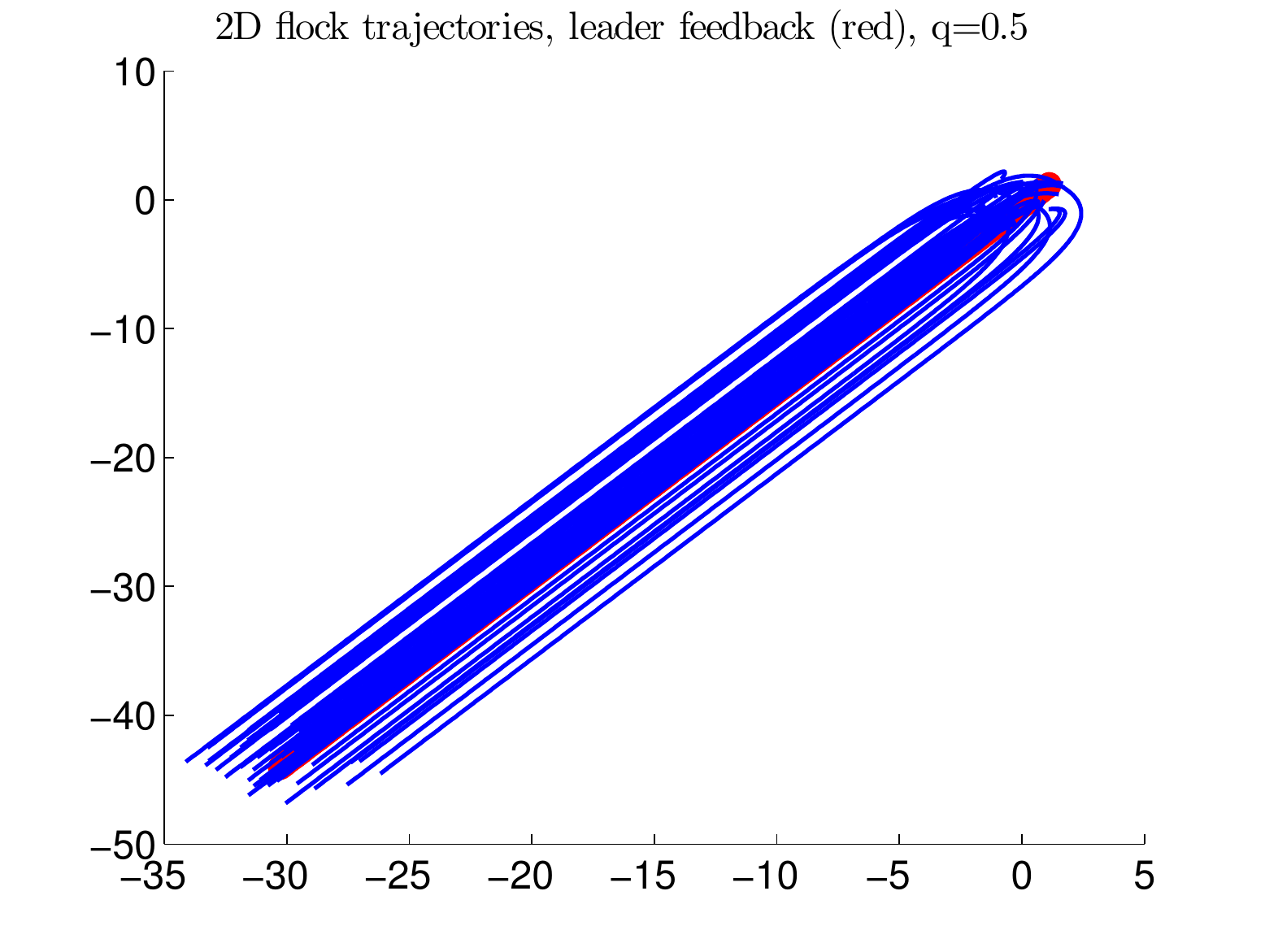}
\includegraphics[width = 0.49\textwidth]{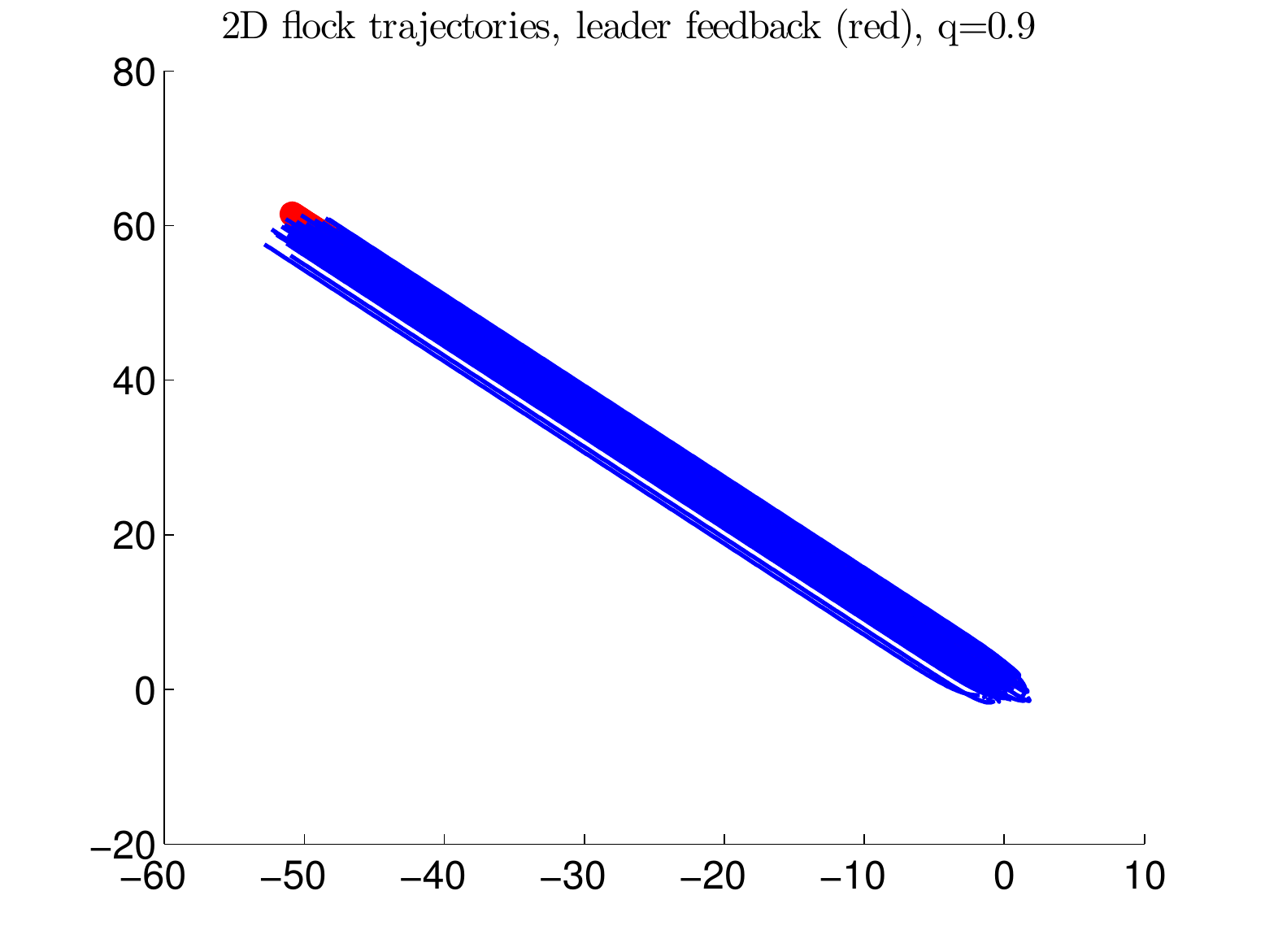}
%\end{tabular}}
\caption{Leader-based feedback control. Simulations with 100 agents, the value $q$ indicates the strength of the leader in the partial average. It can be observed how, as the strength of the leader is increased, convergent behavior is improved.}
\label{fig:1}
\vspace{0.2cm}
%\resizebox{\textwidth}{!}{
%\begin{tabular}{cc}
%\epsfig{file=x.eps,width=0.49\linewidth,clip=}
%\epsfig{file=v.eps,width=0.49\linewidth,clip=}
\includegraphics[width = 0.49\textwidth]{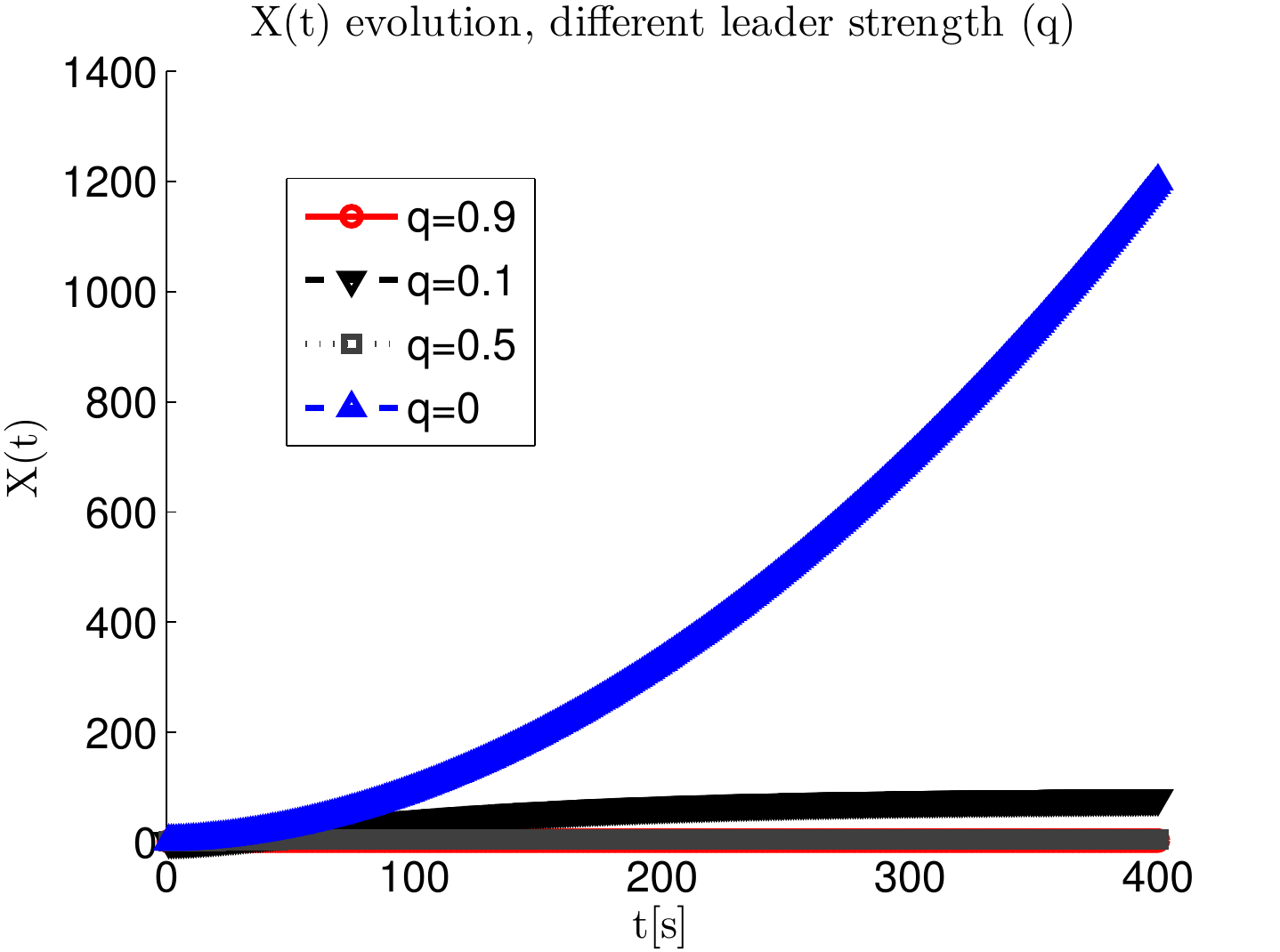}
\includegraphics[width = 0.49\textwidth]{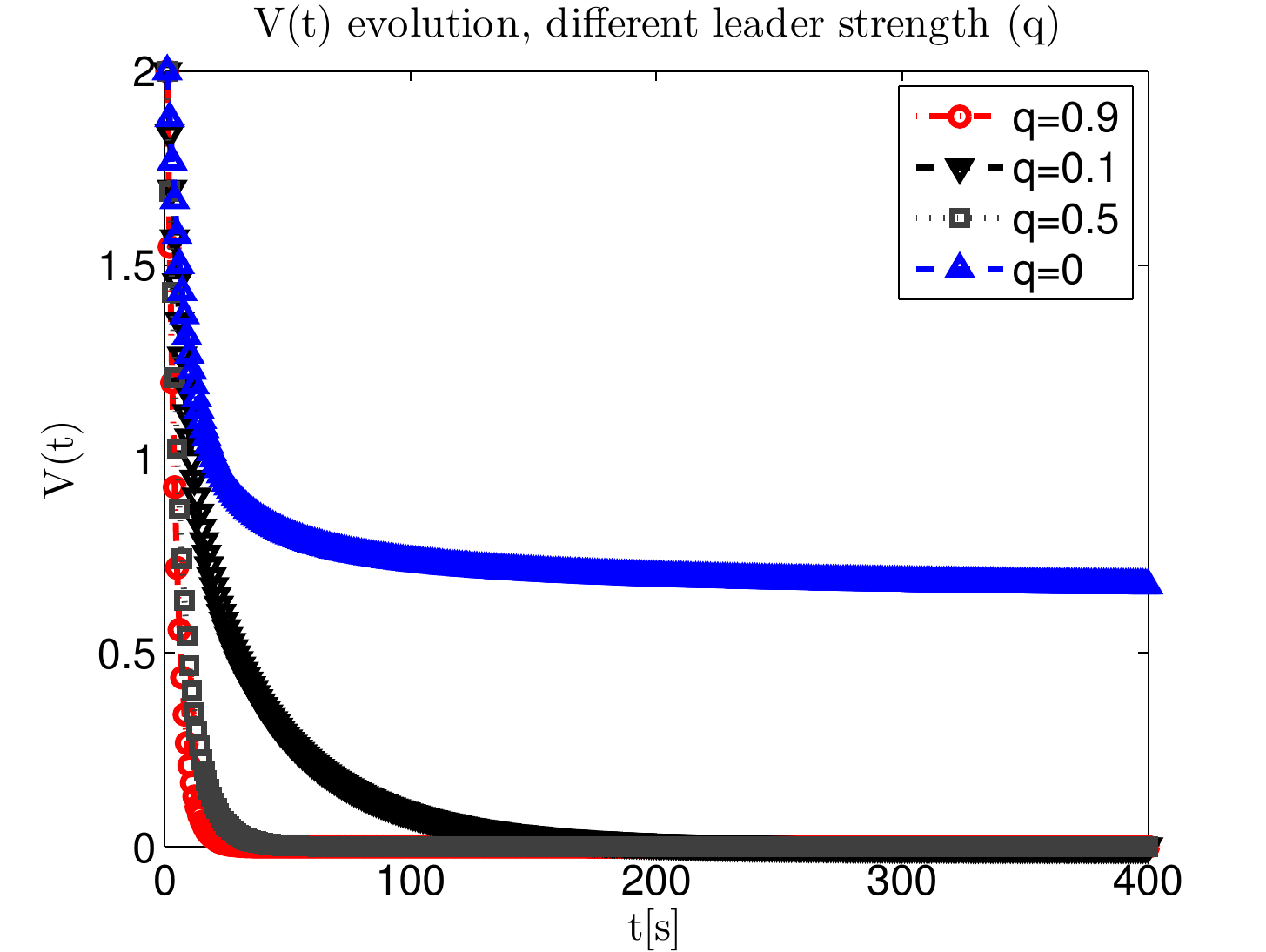}
%\end{tabular}}
\caption{Leader-based feedback control. Simulations with 100 agents, where $q$ indicates the strength of the leader in the partial average. Evolution of $X$ and $V$ for the simulations in Figure \ref{fig:1}.}
\label{fig:12}
\end{figure}

Besides the Cucker-Smale model, the leader-following control problem was also studied in \cite{wongkaew2015control} for the Hegselmann-Krause model, and in \cite{borzi2015modeling} for the D'Orsogna et al. model (see \cite{d2006self} as a reference): in these papers, the leader's optimal strategy to induce pattern formation was discussed.

\subsection{Feedback under perturbed information}

Motivated by the example of the last section, we turn our attention to the study of systems like \eqref{eq:cuckersmale_perturbed} where the perturbation of the mean of the $i$-th agent has the specific form
\begin{align} \label{eq:pert_cucker}
\Delta_i(t) = \sum^N_{j = 1} \omega_{ij}(t) v^{\perp}_j(t) \quad \text{ for every } t\geq 0,
\end{align}
for some positive measurable mapping $\omega: \R_+\rightarrow\R^{N\times N}$, i.e., for every $t \geq 0$ the function $\omega$ has the property $\omega_{ij}(t) > 0$ for all $i,j = 1, \ldots, N$.% as well as %Notice that with the form of $\Delta_i$ given by \eqref{eq:pert_cucker}, system \eqref{eq:cuckersmale_perturbed} satisfies the hypotheses of Theorems \ref{cara-global} and \ref{le:uniquecara}, hence we have existence and uniqueness of Carath\'{e}odory solutions on any finite time interval $[0,T]$.

An example of the above framework is provided by a weight matrix of the form
%Next, we deal with the setting presented again in Section \ref{sec:cuckersmale_localmean}, by considering a system of the form \eqref{eq:cuckersmale_perturbed} where $\alpha$ and $\beta$ are constant functions and the feedback $\Delta_i$ is a structured perturbation written as
%\begin{align*}
%\Delta_i(t)=\sum_{j=1}^N\frac{\phi(\|x_i(t)-x_j(t)\|)}{\eta_i(t)}(v_j(t)-\overline{v}(t)) \quad \text{ for every } t\geq 0.
%\end{align*}
\begin{align*}
\omega_{ij}(t) \define \frac{\phi(\|x_i(t)-x_j(t)\|)}{\eta_i(t)} \quad \text{ for every } t\geq 0.
\end{align*}
where the weighting function $\phi$ corresponds to the Cucker-Smale kernel \eqref{eq:cuckerkernel} with $H = 1, \sigma = 1$ and $\beta = \epsilon$, i.e.,
\begin{align*}
\phi(r)\define\frac{1}{(1+r^2)^{\epsilon}}.
\end{align*}
and the normalizing terms $\eta_i$ are defined as
\begin{align} \label{eq:usual_choice}
\eta_i(t)\define\sum^N_{j = 1} \phi(\|x_i(t)-x_j(t)\|).
\end{align}
Let us consider the case $\alpha(t) = \alpha > 0$ and $\beta(t) = \beta > 0$ for every $t \geq 0$. Figures \ref{fig:2} and \ref{fig:22} show the behavior of the system when changing the balance between the constants $\alpha$ and $\beta$. In this test, we fix a large value of $\beta=10$, representing a strong perturbation of the feedback, and a small value of $\epsilon=1e-5$, related to a disturbance which is distributed among all the agents: increasing the value of $\alpha$ in system \eqref{eq:cuckersmale_perturbed} (which represents the energy of the \textsl{correct information feedback}) induces faster consensus emergence.

\begin{figure}[!ht]
\centering
%\resizebox{\textwidth}{!}{
%\begin{tabular}{cc}
%\epsfig{file=uncontrolled.eps,width=0.49\linewidth,clip=}
%\epsfig{file=001.eps,width=0.49\linewidth,clip=}\\
%\epsfig{file=1.eps,width=0.49\linewidth,clip=}
%\epsfig{file=01.eps,width=0.49\linewidth,clip=}
\includegraphics[width = 0.49\textwidth]{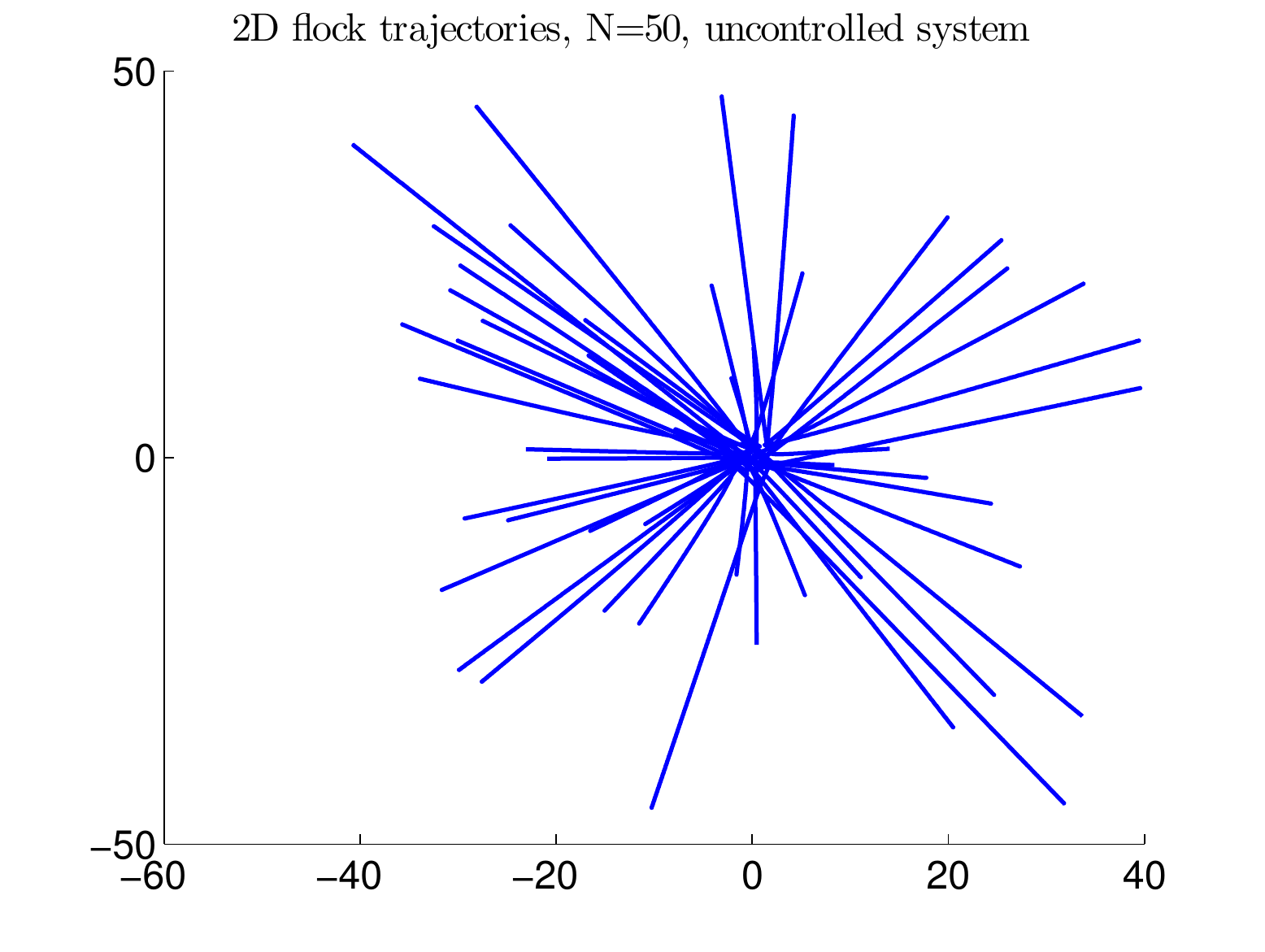}
\includegraphics[width = 0.49\textwidth]{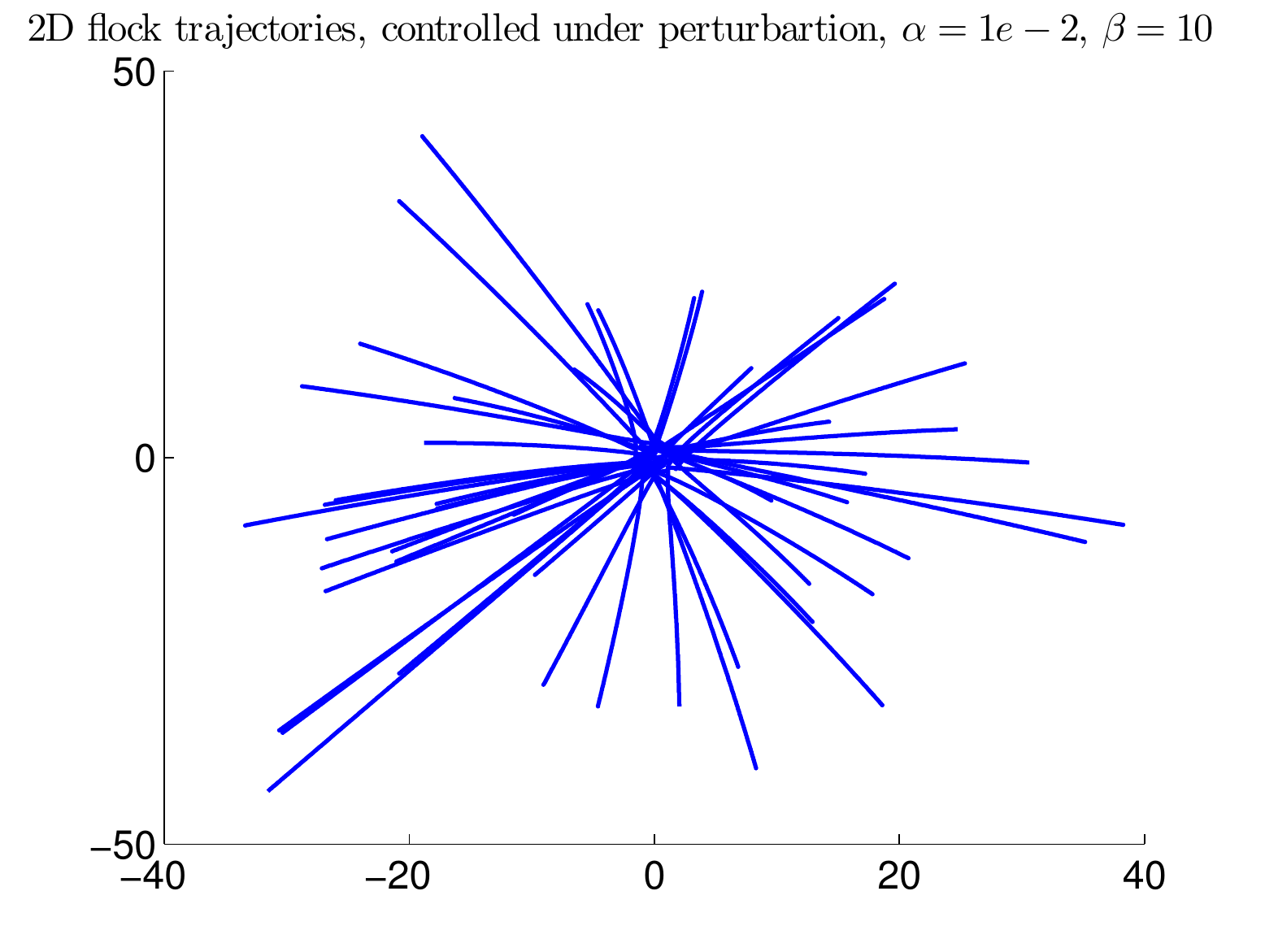}\\
\includegraphics[width = 0.49\textwidth]{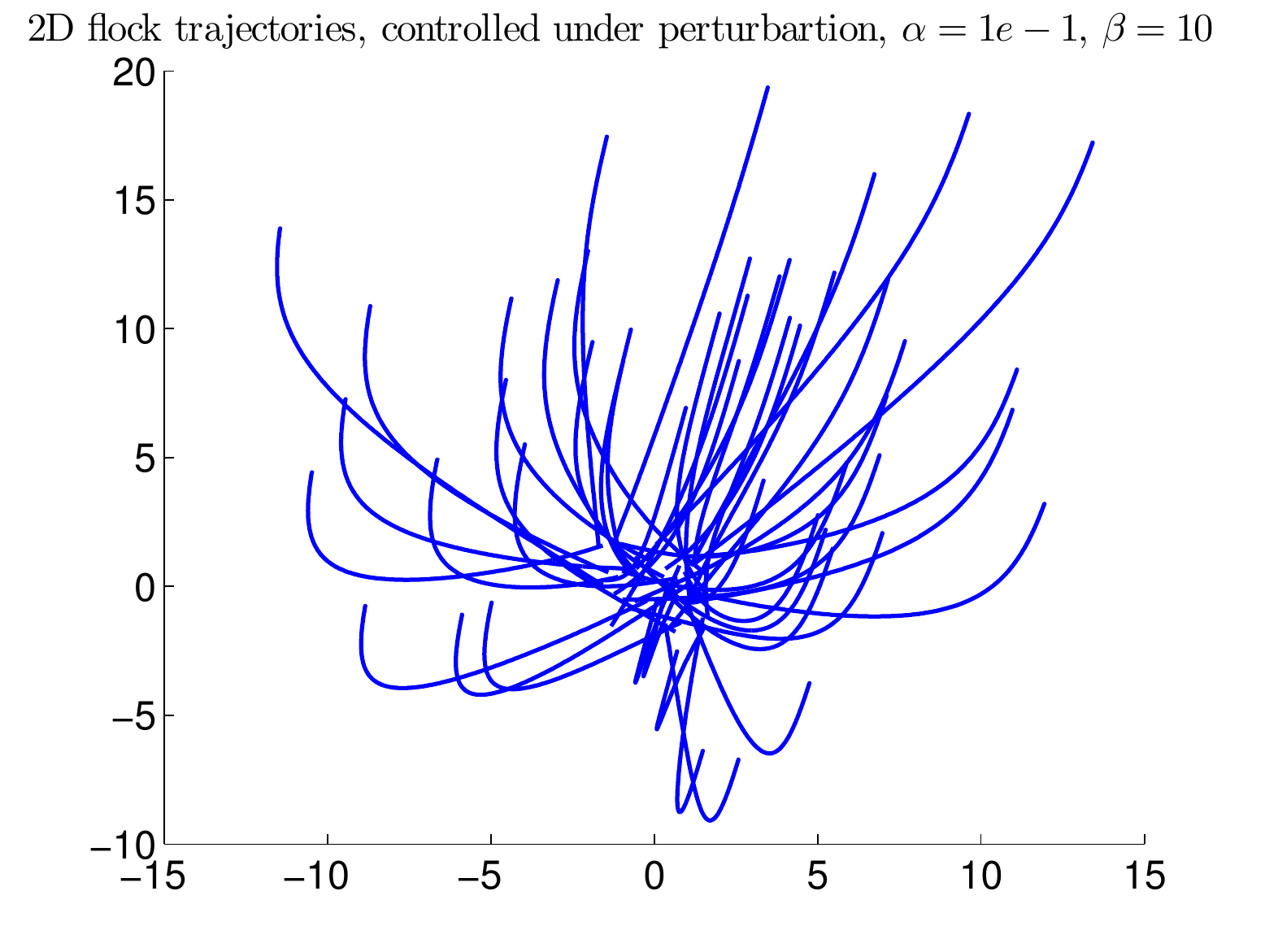}
\includegraphics[width = 0.49\textwidth]{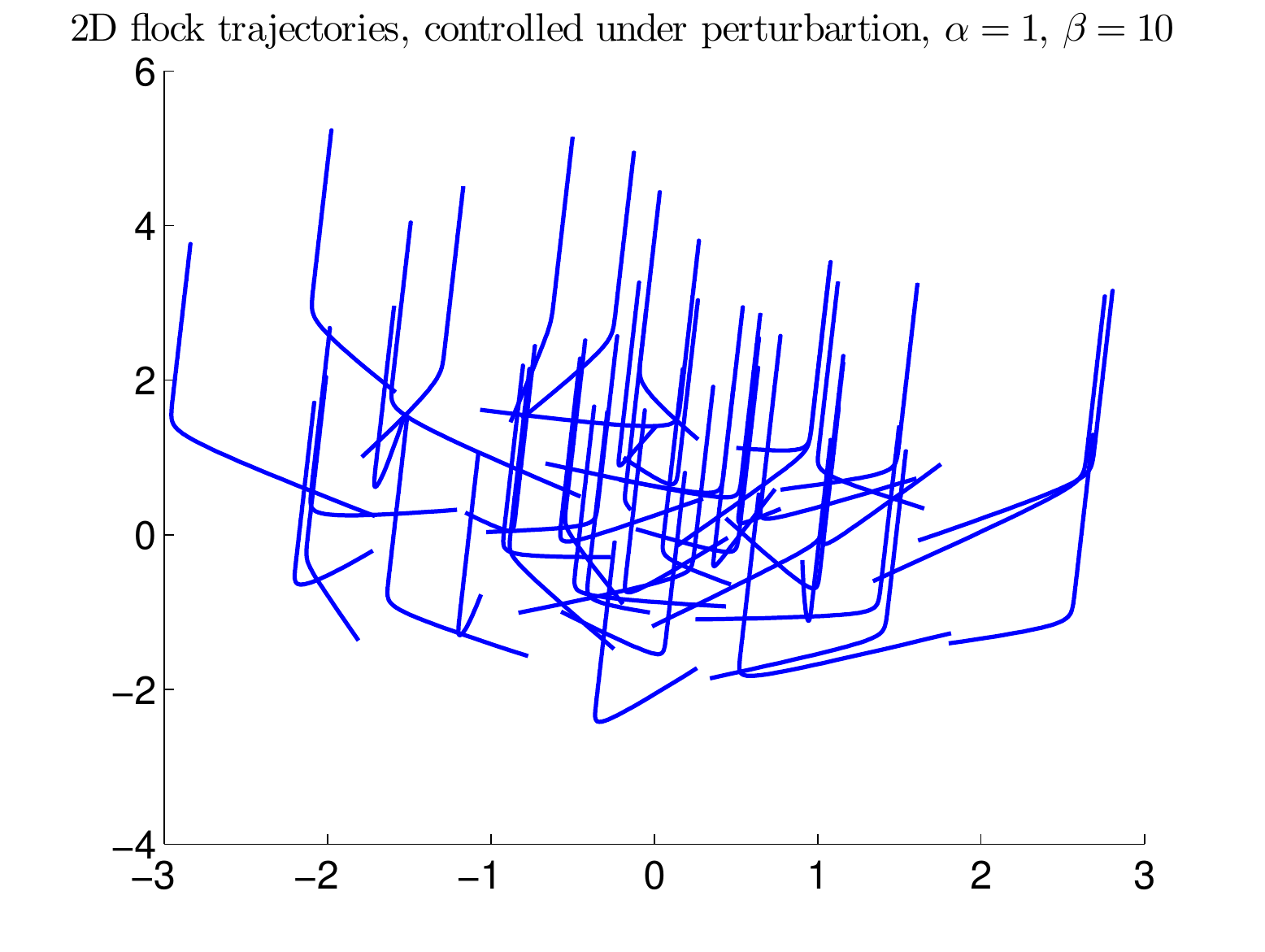}
%\end{tabular}}
\caption{Total feedback control under structured perturbations. For a fixed strong structured perturbation term ($\beta=10$), different energies for the unperturbed control term $\alpha$ generate different consensus behavior; the stronger the correct information term is, the faster consensus is achieved.}
\label{fig:2}
\vspace{0.15cm}
%\resizebox{\textwidth}{!}{
%\begin{tabular}{cc}
\includegraphics[width = 0.49\textwidth]{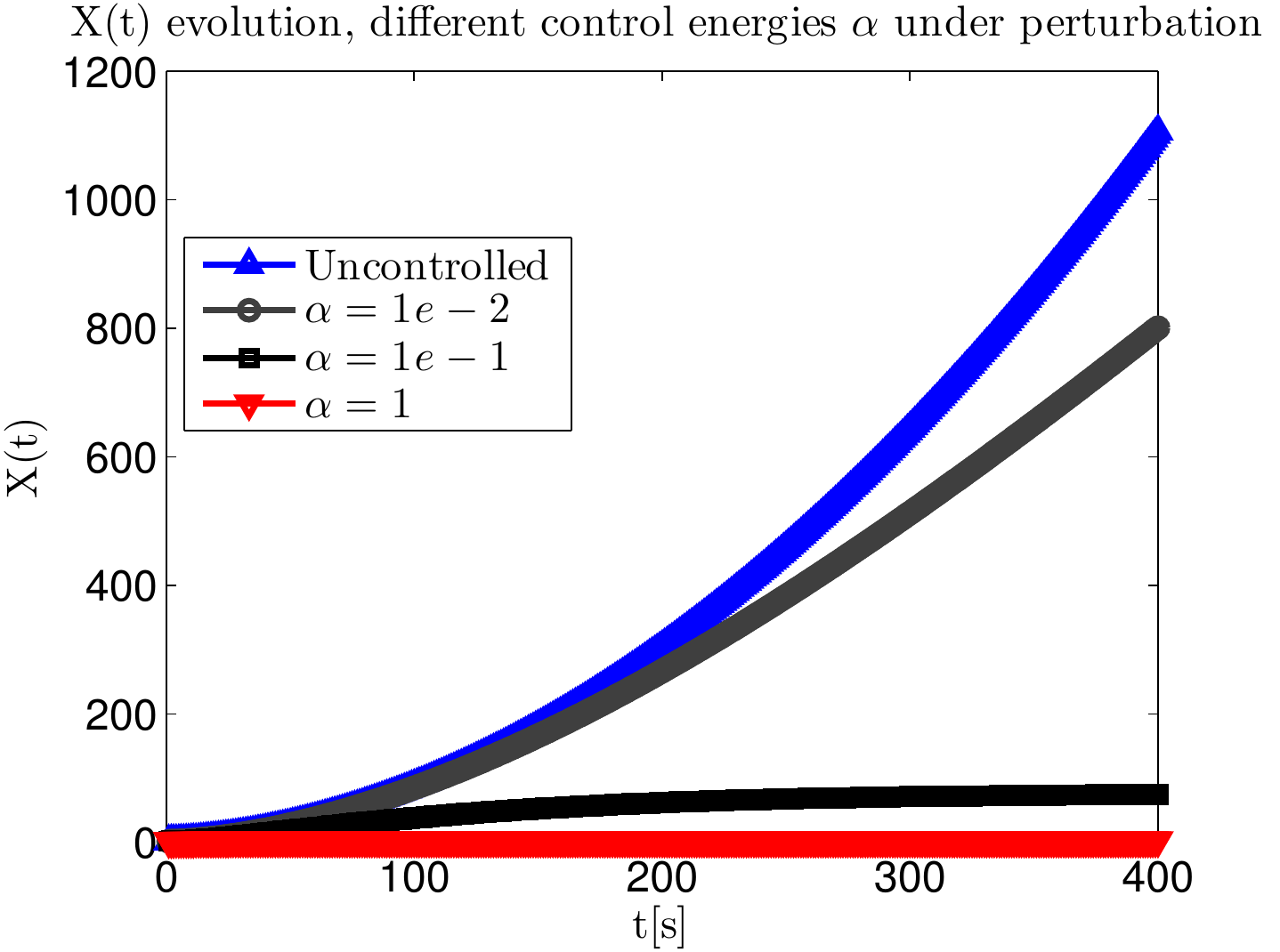}
\includegraphics[width = 0.49\textwidth]{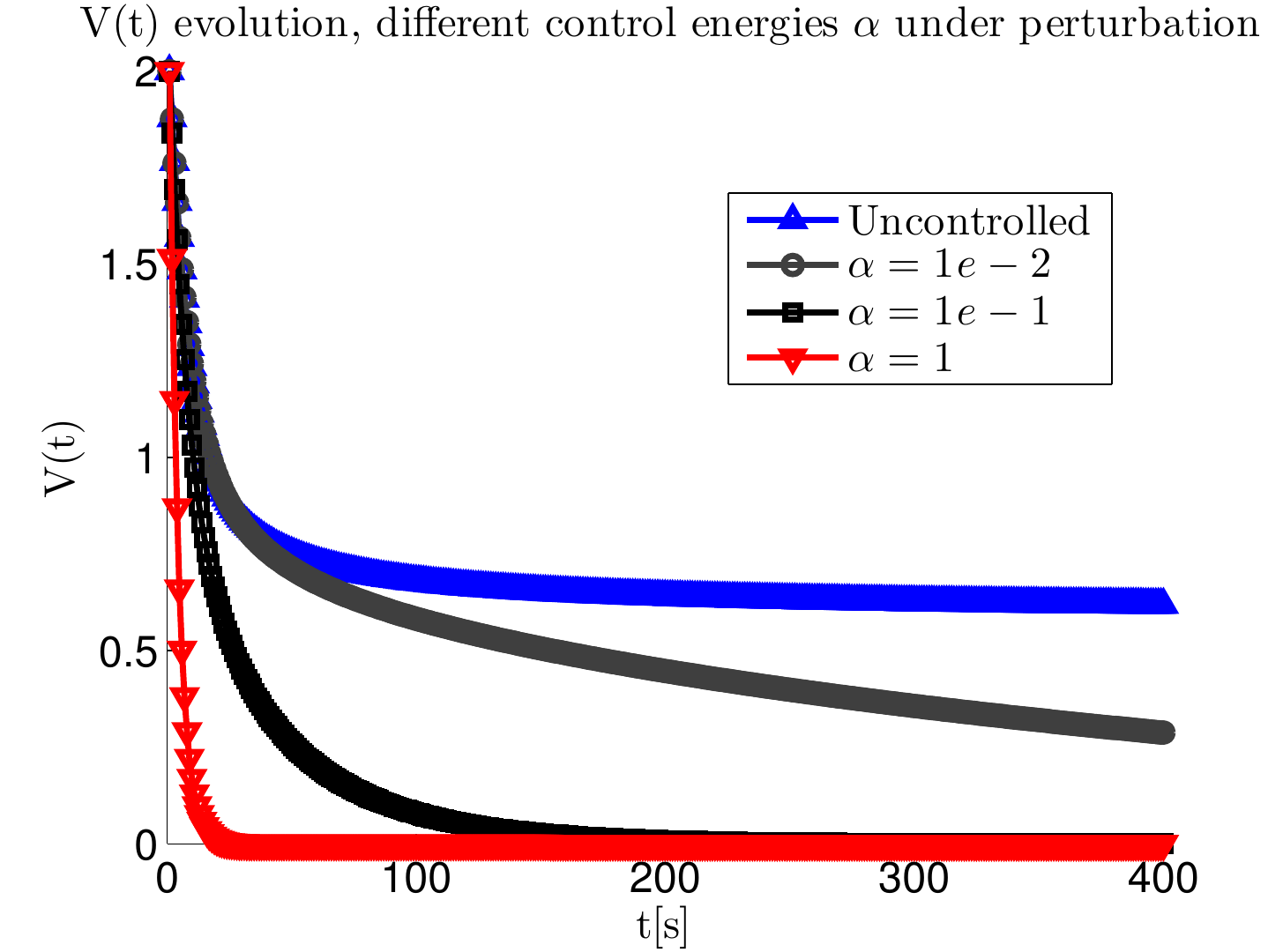}
%\epsfig{file=x1.eps,width=0.49\linewidth,clip=} 
%\epsfig{file=v1.eps,width=0.49\linewidth,clip=}
%\end{tabular}}
\caption{Total feedback control under structured perturbations. Evolution of $X$ and $V$ for the simulations in Figure \ref{fig:2}.}
\label{fig:22}
\end{figure}

As already mentioned in Section \ref{sec:alignmentexample}, the use of a common normalizing factor $\eta$ in place of different terms $\eta_i$ greatly helps in the study of consensus emergence. For this particular case we get the following result.

%The next result focuses on a specific form of the weight matrix $\omega$ which shall be studied further in the upcoming section.

\begin{corollary}[{\cite[Corollary 3]{bongini2015conditional}}]\label{cor:phi_consensus}
Suppose that for all $i,j = 1,\ldots, N$ the function $\omega_{ij}:\R_+\funarrow\R_+$ satisfies
\begin{align*}
\omega_{ij}(t) = \frac{\phi(\|x_i(t)-x_j(t)\|)}{\eta(t)} \quad \text{ for every } t\geq 0,
\end{align*}
where $\phi:\R_+ \funarrow \left(0,1\right]$ is a non increasing, positive, bounded function, and $\eta:\R_+ \funarrow \R_+$ is a nonnegative bounded function satisfying
\begin{align*}% \label{eq:betaalpha}
1 \leq N \frac{\beta(t)}{\alpha(t)} \leq \eta(t) \quad \text{for every } t \geq 0.
\end{align*}
Then, any solution of system \eqref{eq:cuckersmale_perturbed} with $\Delta_i$ as in \eqref{eq:pert_cucker} tends to consensus.
\end{corollary}

\begin{remark} \label{re:corollary_phi}
A concrete example of a system for which we can apply Corollary \ref{cor:phi_consensus} is obtained by considering the common normalizing term
\begin{align*}
\eta(t) \define \max_{1\leq i \leq N}\left\{\sum^N_{j = 1} \phi(\|x_i(t)-x_j(t)\|)\right\}.
\end{align*}
%In this case, $\epsilon \in \R_+$ can be thought of as a parameter tuning the ability of each agent to gather information about the speed of the other agents. Indeed, consider the case $\|x_i(t)-x_j(t)\| > 0$, if $i \not = j$: then, if $\epsilon = 0$, it holds $\overline{v}_i = \overline{v}$ for every $i = 1,\ldots, N$, and each agent communicates at the same rate with both near and far away agents; if instead $\epsilon \rightarrow +\infty$ then $\overline{v}_i$ approaches $v_i$, hence each agent is unable to gain sufficient knowledge about the speed of the others in order to compute the correct mean. The function $\eta$ serves to the purpose of being a common normalizing factor: naturally one would choose for every agent $i$ the normalizing factor given by
%\begin{align} \label{eq:usual_choice}
%\sum^N_{j = 1} \phi(\|x_i(t)-x_j(t)\|),
%\end{align}
%but that would produce a non symmetric matrix $\omega$, for which the above results are not valid. In this context, the function $\eta$ is a suitable replacement, being 
which is also coherent with the asymptotic behavior of \eqref{eq:usual_choice} for $\epsilon \rightarrow 0$ and $\epsilon \rightarrow +\infty$.
\end{remark}

\begin{remark} \label{rem:functionR}
The request of positivity of the function $\phi$ cannot be removed from Corollary \ref{cor:phi_consensus}, see \cite[Remark 4]{bongini2015conditional}%as the function $\phi = \chi_{[0,R]}$ shows. %Here $\chi_{[0,R]}$ is the characteristic function of the ball of radius $R$ centered at the origin.
%This is the case when each agent is only able to gather information to calculate the mean velocity of the system in a ball with radius $R \geq 0$ with center in the agent. This model can be visualized by the following situation: imagine a group of people are inside a pitch black environment, where they cannot see anything. They can communicate with each other using only their voice, so the nearer the agent which speak is the better the voice can be recognized, and thus the more effective the exchange of information is. Now, suppose they are given a lighter and this lighter casts a light aura of radius $R > 0$ (the case $R = 0$ can be visualized with the absence of a lighter). The consequence is that they receive further information from other agents as soon as they stay inside the cone of light the lighter casts. The strongest is the will of each agent to remain near to those he can see thanks to the lighter (preference modeled by the parameter $\alpha$ in \eqref{eq:cuckersmale_perturbed}) the more this situation will give rise to subgroups of agents moving coherently, i.e. as a single agent placed in the center of mass of the subgroup will move.
%Indeed, what fails in the argument of the proof is the case in which we suppose that the functional $X$ is bounded from above by $\overline{X}$: if the quantity $\sqrt{2N\overline{X}}$ is not less or equal to $R$, then $\phi(\sqrt{2N\overline{X}}) = 0$ in the inequality \eqref{eq:invoketheorem}, and we cannot invoke Theorem \ref{th:suff_cond_consensus} in order to infer consensus.
\end{remark}

\subsection{Perturbations due to local averaging} \label{sec:localmeanR}

An interesting case of a system like \eqref{eq:cuckersmale_local} is the one where the local mean is given by
\begin{align}\label{eq:approxmeanR}
\overline{v}_i(t) = \frac{1}{|\Lambda_R(t,i)|} \sum_{j \in \Lambda_R(t,i)} v_j(t) \quad \text{ for every } t\geq 0,
\end{align}
where $\Lambda_R(i)$ is defined as in \eqref{eq:neighborset}. In this case, we model the situation in which each agent estimates the average velocity of the group in the extra feedback term by only counting those agents inside a ball of radius $R$ centered on him.

Simulations in Figure \ref{fig:3} illustrate the behavior of %system \eqref{eq:cuckersmale_local} with local average given by \eqref{eq:approxmeanR}.
such configuration. From an uncontrolled system, represented by a local feedback radius $R=0$, by increasing this quantity, partial alignment is consistently achieved, until  full consensus is observed for large radii mimicking a total information feedback control.

\begin{figure}[!ht]
\centering
%\resizebox{1.05\textwidth}{!}{
%\begin{tabular}{ccc}
%\epsfig{file=r0.eps,width=0.49\linewidth,clip=}
%\epsfig{file=r1.eps,width=0.49\linewidth,clip=}\\
%\epsfig{file=r2.eps,width=0.49\linewidth,clip=}
%\epsfig{file=r3.eps,width=0.49\linewidth,clip=}\\
%\epsfig{file=gamma.eps,width=0.49\linewidth,clip=}
%\epsfig{file=lambda.eps,width=0.49\linewidth,clip=}
\includegraphics[width = 0.49\textwidth]{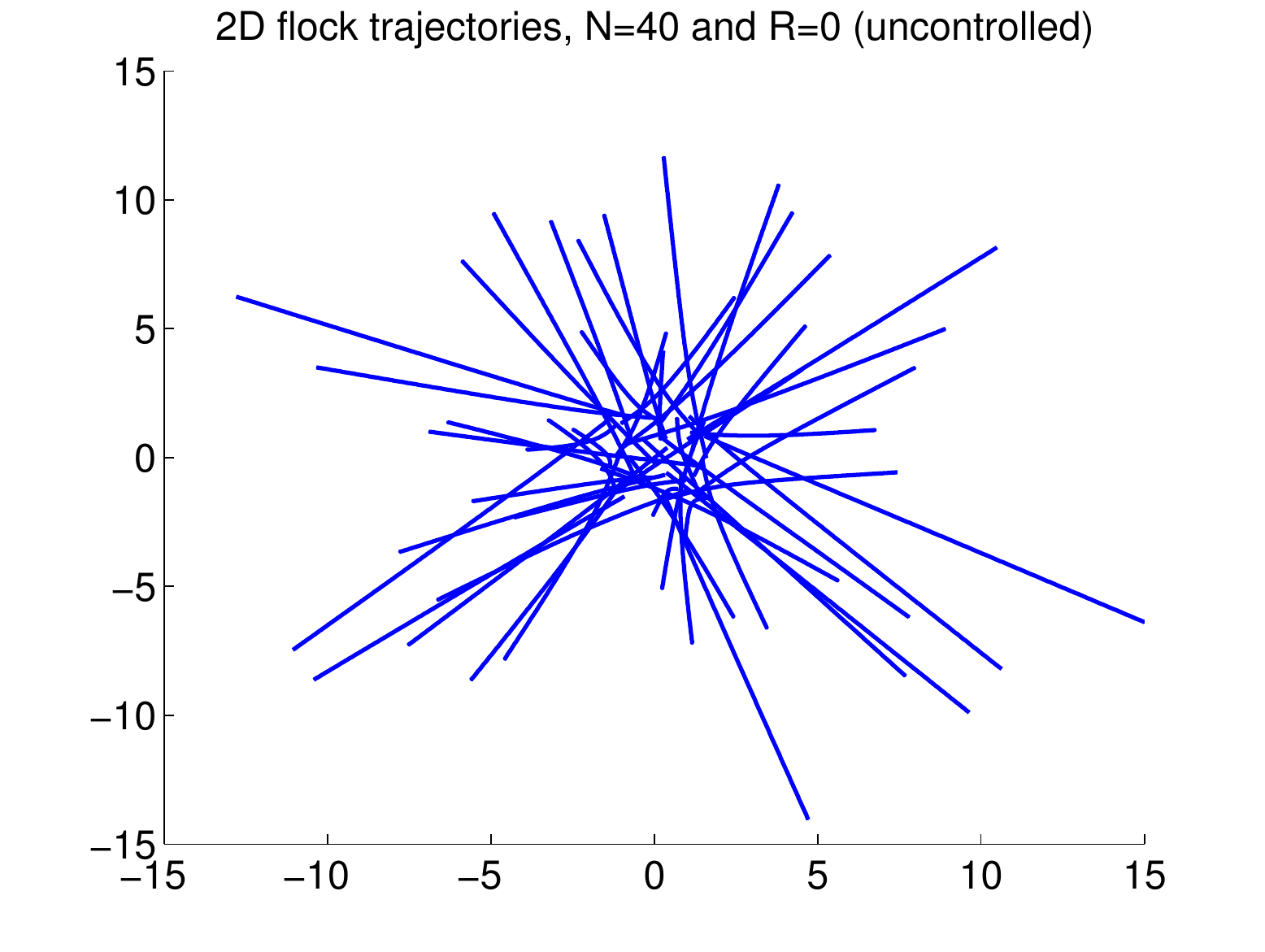}
\includegraphics[width = 0.49\textwidth]{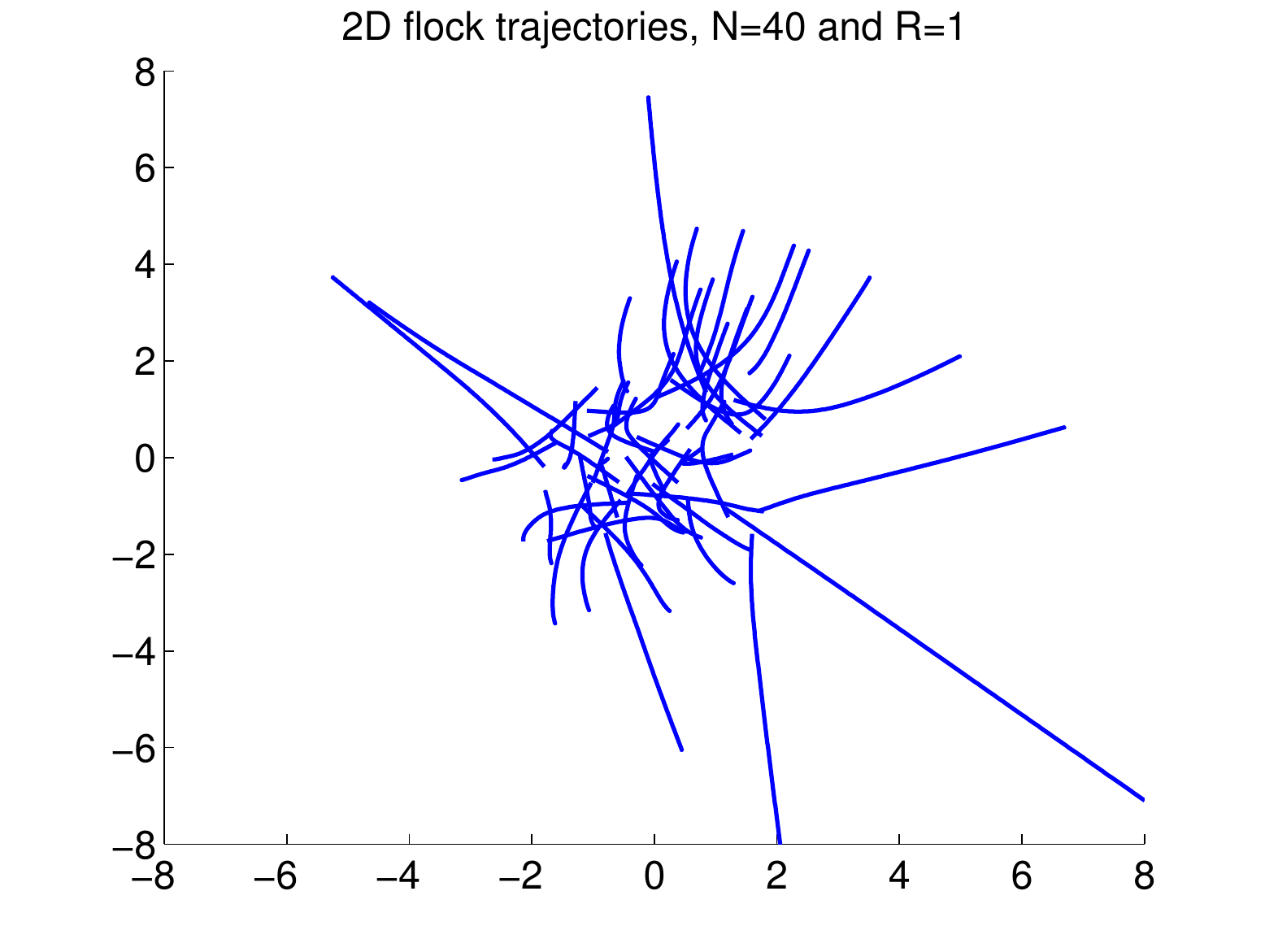}\\
\includegraphics[width = 0.49\textwidth]{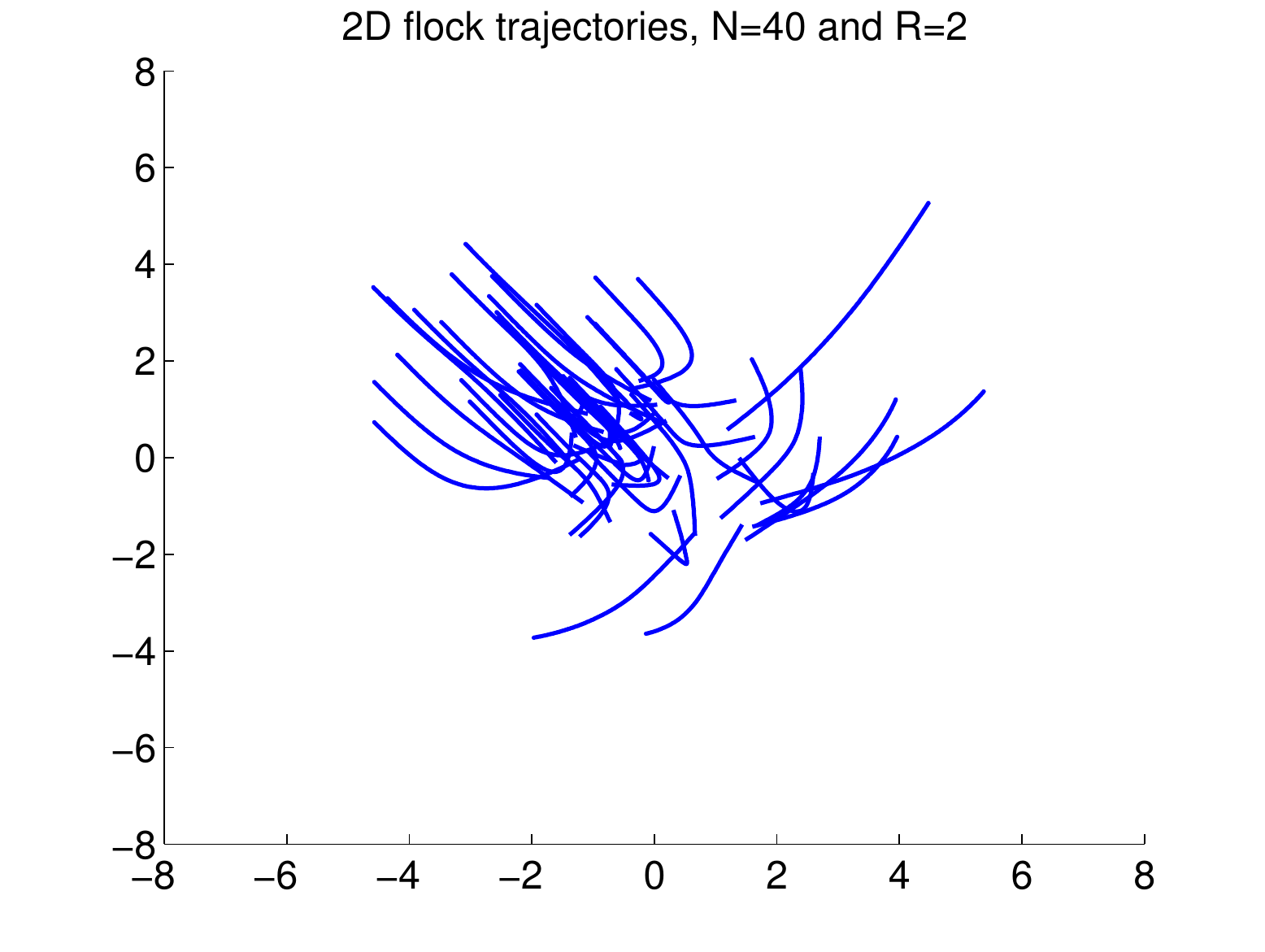}
\includegraphics[width = 0.49\textwidth]{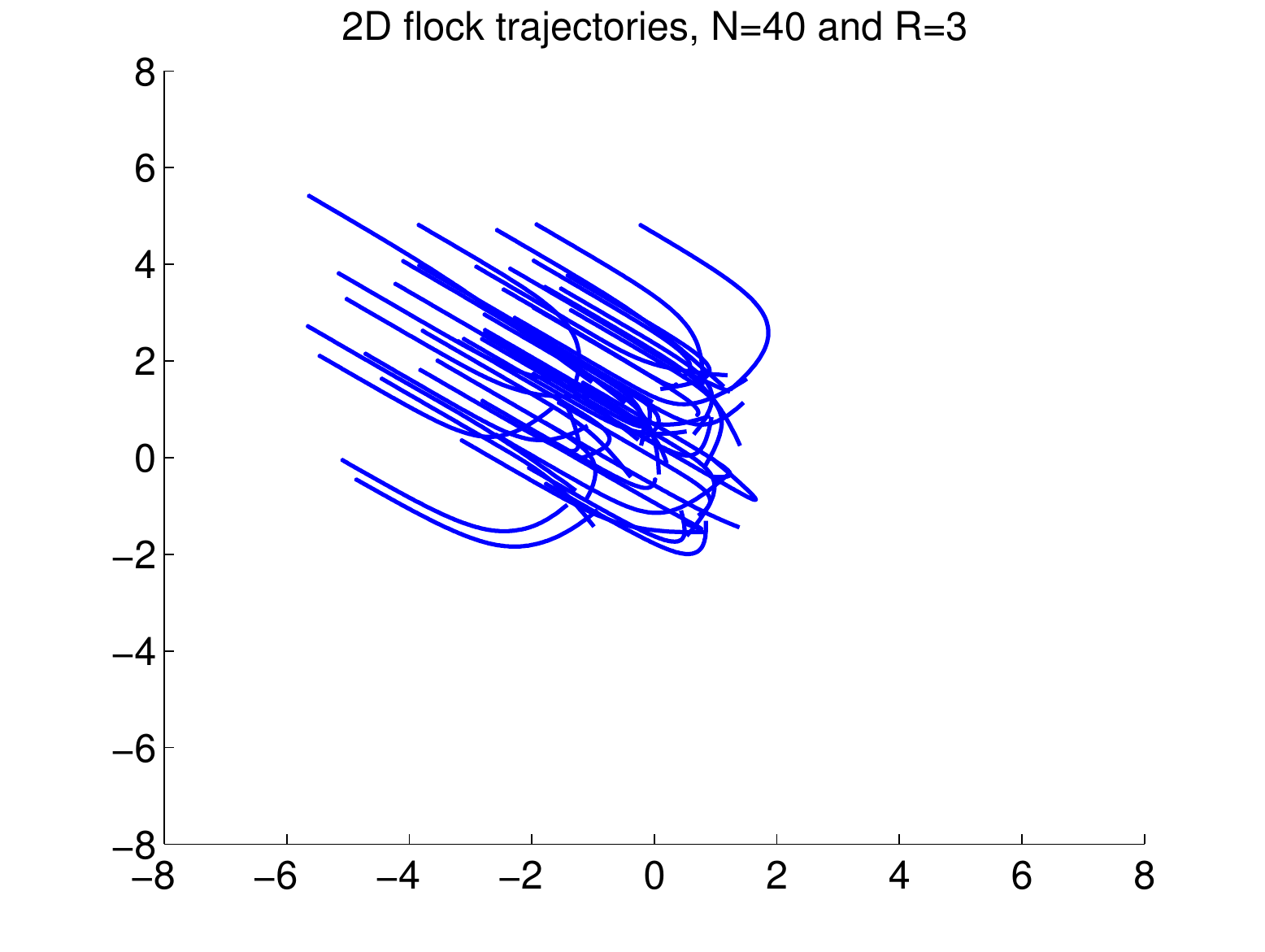}\\
\includegraphics[width = 0.49\textwidth]{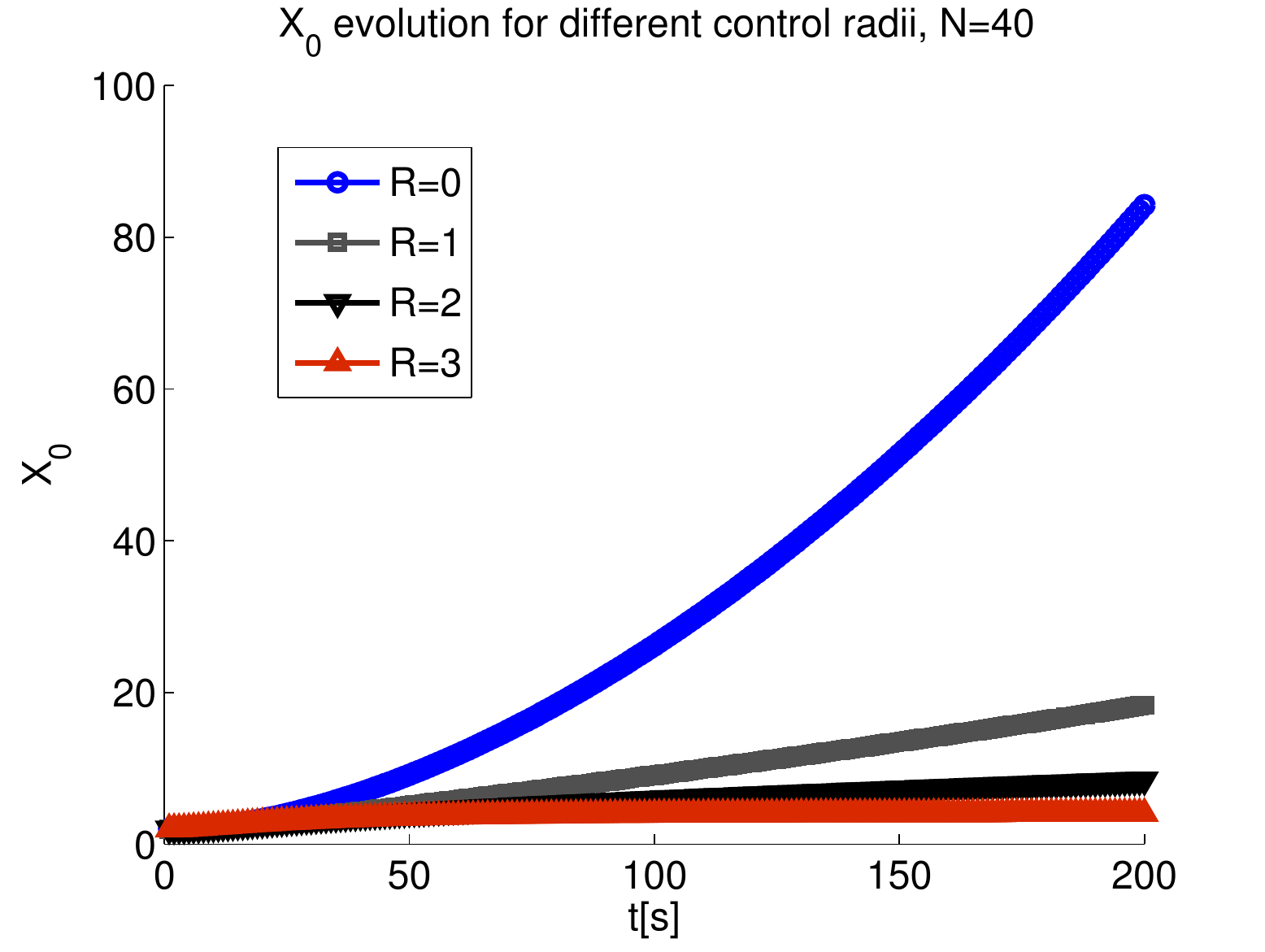}
\includegraphics[width = 0.49\textwidth]{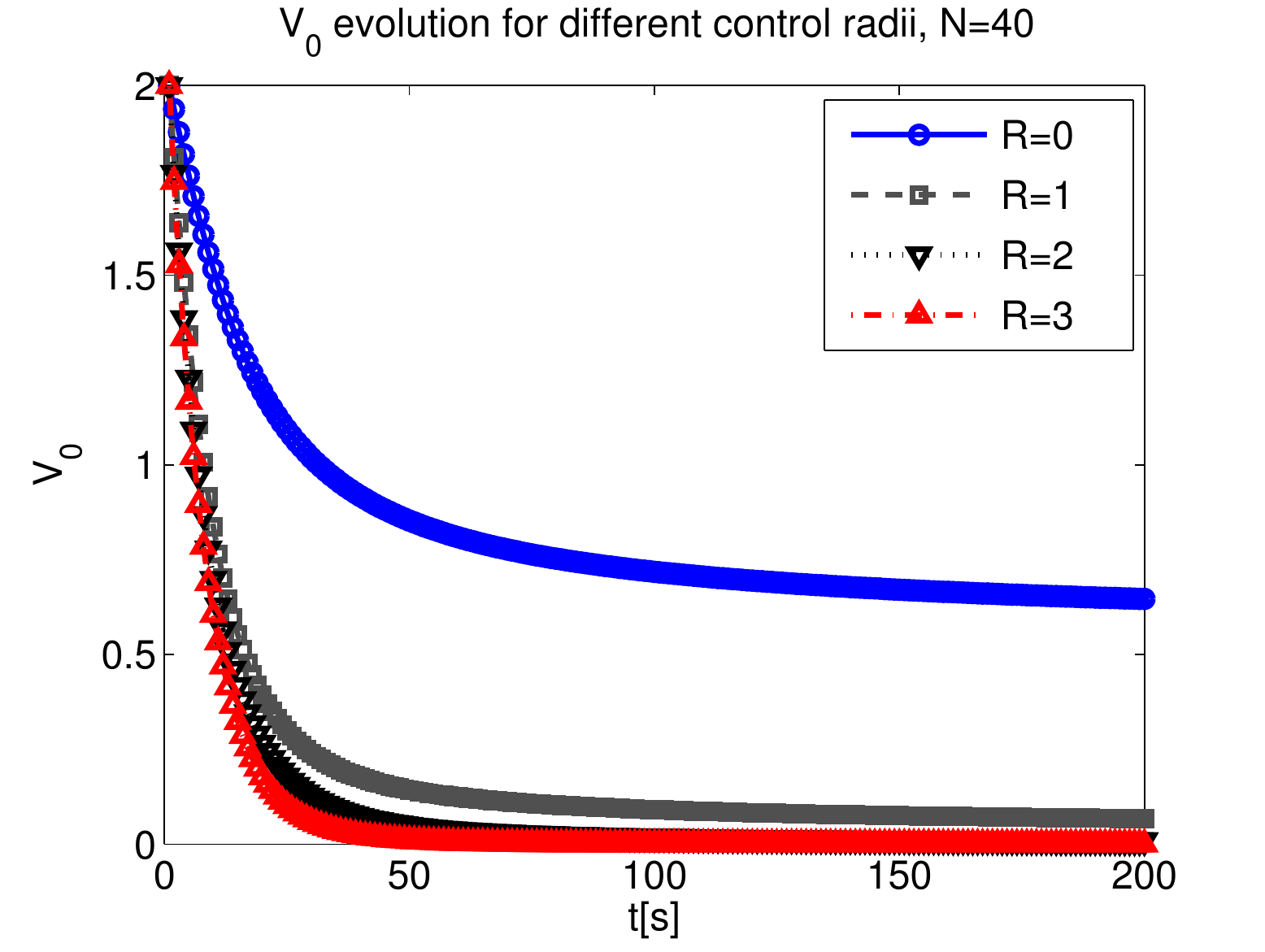}
%\end{tabular}}
\caption{Local feedback control. Simulations with $N=40$ agents, and  different control radii $R$. By increasing the value of $R$ the systems transits from uncontrolled behavior, to partial alignment, up to total, fast alignment.}
\label{fig:3}
\end{figure}

We want to address the issue of characterizing the behavior of system \eqref{eq:cuckersmale_local} with the above choice for $\overline{v}_i$ when the radius $R$ of each ball is either reduced to $0$ or set to grow to $+\infty$. We shall see that we can reformulate  this decentralized system again as a Cucker-Smale model for a different interaction function for which we can apply Theorem \ref{thm:hhk}. We shall show how tuning the radius $R$ affects the convergence to consensus, from the case $R \geq 0$ where only conditional convergence is ensured, to the unconditional convergence result given for $R = +\infty$.

\subsubsection{Preserving the asymptotics}

First of all, by means of $\chi_{[0,R]}$ we can rewrite $\overline{v}_i(t)$ as
\begin{align} \label{eq:truncated_perturbation}
\overline{v}_i(t) = \frac{1}{\sum^N_{k = 1} \chi_{[0,R]}(\|x_i(t) - x_k(t)\|)} \sum_{j = 1}^N \chi_{[0,R]}(\|x_i(t) - x_j(t)\|) v_j(t).
\end{align}
%A good approximation of the step function $\chi_{[0,R]}$ is given by the function
%\begin{align} \label{eq:truncated_bell}
%\psi_{R,\theta}(r) = \left\{
%\begin{array}{ll}
%1 & \text{ if } r \leq R, \\
%\frac{1}{(r - R + 1)^{\theta}} & \text{ if } r > R. \\
%\end{array}
%\right.
%\end{align}
%which is parametrized by the couple $(R, \theta)$.
%As already noted in Remark \ref{re:corollary_phi}
Unfortunately, the normalizing terms $\sum^N_{k = 1} \chi_{[0,R]}(\|x_i(t) - x_k(t)\|)$ give rise to a matrix of weights which is not symmetric, which greatly complicates the analysis of the convergence to consensus. However, since %this will be an issue also in the present section, 
we are mainly interested in the limit behavior of the system for $R\rightarrow 0$ and $R\rightarrow+\infty$, following Remark \ref{re:corollary_phi} we take $\eta_R$ to be a function approximating the above normalizing terms and which also preserves its asymptotics for $R \rightarrow 0$ and $R \rightarrow +\infty$, as for instance,
\begin{align} \label{eq:normalizingR}
\eta_R(t) = \max_{1\leq i\leq N} \left\{\sum^N_{k = 1} \chi_{[0,R]}(\|x_i(t) - x_k(t)\|) \right\}.
\end{align}
Therefore, we replace the vector $\overline{v}_i(t)$ by
\begin{align*}
\frac{1}{\eta_{R}(t)} \sum^N_{j  = 1} \chi_{[0,R]}(\|x_i(t) - x_j(t)\|) v_j(t)\,.
\end{align*}
On top of this, notice that the vector
\begin{align*}
\left(\frac{1}{\eta_{R}(t)} \sum^N_{j  = 1} \chi_{[0,R]}(\|x_i(t) - x_j(t)\|) \right)v_i(t)
\end{align*}
is an approximation of $v_i(t)$ for $R \rightarrow 0$ and $R \rightarrow +\infty$. This motivates the replacement of the term $\overline{v}_i - v_i$ where $\overline{v}_i$ is as in \eqref{eq:truncated_perturbation} with
\begin{align}\label{eq:repv}
\overline{v}_i - v_i \approx \frac{1}{\eta_{R}} \sum^N_{j  = 1} \chi_{[0,R]}(\|x_i - x_j\|)v_j - \left(\frac{1}{\eta_{R}}  \sum^N_{j  = 1} \chi_{[0,R]}(\|x_i - x_j\|) \right)v_i.
\end{align}
The term \eqref{eq:repv} can be rewritten as $1/\eta_{R}  \sum^N_{j  = 1} \chi_{[0,R]}(\|x_i - x_j\|) (v_j - v_i)$, which can be further simplified as follows
\begin{align}
\begin{split} \label{eq:derivationR}
\frac{1}{\eta_{R}}  \sum^N_{j  = 1} \chi_{[0,R]}(r_{ij}) (v_j - v_i) &= \frac{1}{\eta_{R}}  \sum^N_{i = 1} (v_j - v_i) + \frac{1}{\eta_{R}}  \sum^N_{j  = 1} (1 - \chi_{[0,R]}(r_{ij})) (v_i - v_j) \\
& = \frac{N}{\eta_{R}} \left(\overline{v} - v_i\right)  + \frac{1}{\eta_{R}}  \sum^N_{j  = 1} (1 - \chi_{[0,R]}(r_{ij})) (v_i - v_j),
\end{split}
\end{align}
where we have written $r_{ij}$ in place of $\|x_i - x_j\|$ and removed the time dependencies for the sake of compactness.

It is clear that the choice of the function $\chi_{[0,R]}$ is arbitrary and other alternatives can be selected, provided they give a coherent approximation of the local average \eqref{eq:approxmeanR}. For instance, instead of $\chi_{[0,R]}$ and $\eta_R$, we can consider two generic functions $\psi_{\varepsilon}$ and $\eta_{\varepsilon}$, where $\varepsilon$ is a parameter ranging in a nonempty set $\Omega$, satisfying the following properties:
\begin{enumerate}[label=$(\roman*)$]
\item \label{item:i} $\psi_{\varepsilon}:\R_+ \funarrow [0,1]$ is a nonincreasing measurable function for every $\varepsilon \in \Omega$;
\item \label{item:ii} $\eta_{\varepsilon}\in L^{\infty}(\R_+)$ %and $1/\eta_{\varepsilon}\in L^1(\R_+)$
for every $\varepsilon \in \Omega$;
\item \label{item:iii} there are two disjoint subsets $\Omega_{C\!S}$ and $\Omega_{U}$ of $\Omega$ such that
\begin{itemize}
\item if $\varepsilon \in \Omega_{C\!S}$ then $\psi_{\varepsilon} = \chi_{\{0\}}$ and $\eta_{\varepsilon} \equiv 1$;
\item if $\varepsilon \in \Omega_{U}$ then $\psi_{\varepsilon} = \chi_{\R_+}$ and $\eta_{\varepsilon} \equiv N$.
\end{itemize}
\end{enumerate}
Under the above hypotheses, we consider the perturbation given for every $t\geq 0$ by
\begin{align} \label{eq:pert_repulsion}
\Delta^{\varepsilon}_i(t) \define \frac{1}{\eta_{\varepsilon}(t)} \sum^N_{j = 1} (1 - \psi_{\varepsilon}(\|x_i(t) - x_j(t)\|)) (v_i(t) - v_j(t)).
\end{align}
With requirement \ref{item:iii} we impose that whenever $\varepsilon \in \Omega_{C\!S}$ then it holds $$\Delta_i^{\varepsilon}(t) = - \frac{N}{\eta_{\varepsilon}(t)} (\overline{v}(t) - v_i(t)),$$ therefore recovering the Cucker-Smale system \eqref{eq:cuckersmale} from \eqref{eq:cuckersmale_R}, while whenever $\varepsilon \in \Omega_{U}$ then $\Delta_i^{\varepsilon}(t) = 0$ holds, and we obtain a particular instance of system \eqref{eq:cuckersmale_uniform}.

\subsubsection{The enlarged consensus region}

By means of \eqref{eq:derivationR} and \eqref{eq:pert_repulsion}, we can rewrite our system with the local average \eqref{eq:approxmeanR} in the form of system \eqref{eq:cuckersmale_perturbed} as follows
\begin{align}
\left\{
\begin{aligned}
\begin{split} \label{eq:cuckersmale_R}
\dot{x}_{i}(t) & = v_{i}(t), \\
\dot{v}_{i}(t) & = \frac{1}{N} \sum_{j = 1}^N a\left(\|x_i(t) - x_j(t)\|\right)\left(v_{j}(t)-v_{i}(t)\right) + \gamma \frac{N}{\eta_{\varepsilon}(t)} (\overline{v}(t) - v_i(t)) + \gamma \Delta^{\varepsilon}_i(t).
\end{split}
\end{aligned}
\right.
\end{align}
Using \eqref{eq:derivationR} and collecting the term $(v_j - v_i)$, it is easy to see that Theorem \ref{thm:hhk} yields the following description of the consensus region as a function of the parameter $\eps$.

\begin{theorem} \label{th:HaHaKimExtended}
Fix $\gamma \geq 0$, consider system \eqref{eq:cuckersmale_R} where $\Delta^{\varepsilon}_i$ is as in \eqref{eq:pert_repulsion} and let $(x^0,v^0) \in \R^{dN} \times \R^{dN}$. If $X_0 \define B(x^0, x^0)$ and $V_0 \define B(v^0,v^0)$ satisfy
\begin{align} \label{eq:HaHaKimLarge}
\int^{+\infty}_{\sqrt{X_0}} a\left(\sqrt{2N}r\right) \ dr + \frac{\gamma N}{\vnorm{\eta_{\varepsilon}}_{L^{\infty}(\R_+)}} \int^{+\infty}_{\sqrt{X_0}} \psi_{\varepsilon}\left(\sqrt{2N}r\right) \ dr \geq \sqrt{V_0},
\end{align}
then the solution of system \eqref{eq:cuckersmale_R} with initial datum $(x^0,v^0)$ tends to consensus.
\end{theorem}

Let us see how we can apply Theorem \ref{th:HaHaKimExtended} to obtain an estimate of the consensus region for the local average \eqref{eq:approxmeanR}.
%The most interesting example of such a family is the one which has introduced this section: w
We consider $\Omega = [0, +\infty]$, the sequence of functions $(\chi_{[0,R]})_{R \in \Omega}$ and $\eta_R$ as in \eqref{eq:normalizingR} (notice that, as before, we have $\Omega_{C\!S} = \{0\}$ and $\Omega_{U} = \{\infty\}$). Since it holds $\vnorm{\eta_R}_{L^{\infty}(\R_+)} \leq N$, if $R$ is sufficiently large to satisfy $\sqrt{2NX_0} \leq R$, condition \eqref{eq:HaHaKimLarge} is satisfied as soon as
\begin{align*}
\int^{+\infty}_{\sqrt{X_0}} a\left(\sqrt{2N}r\right) \ dr + \gamma \left(\frac{R}{\sqrt{2N}} - \sqrt{X_0}\right) \geq \sqrt{V_0},
\end{align*}
by means of a trivial integration. If, instead, $R$ is so small that $\sqrt{2NX_0} > R$ holds, condition \eqref{eq:HaHaKimLarge} is satisfied as soon as
\begin{align*}
\int^{+\infty}_{\sqrt{X_0}} a\left(\sqrt{2N}r\right) \ dr \geq \sqrt{V_0},
\end{align*}
recovering Theorem \ref{thm:hhk}. As can be seen, we have enlarged the original consensus region provided by Theorem \ref{thm:hhk} by a term whose size is linearly increasing in $R$. This implies that, in the case $R = +\infty$, the consensus region coincides with the entire space $\R^{dN}\times\R^{dN}$, hence the system converges to consensus regardless of the initial datum.

\subsubsection{Empirical estimation of the enlarged consensus region}\label{sim:local}
We present a series of numerical tests aiming at estimating empirically the enlarged consensus region given by \eqref{eq:HaHaKimLarge} %illustrating the results of the last section, %. First, we describe the setting upon which the initial configurations of agents are determined; we follow 
following similar ideas as those presented in \cite{cafotove10}. We consider a system of $N$ agents in dimension $d=2$ with a randomly generated initial configuration of positions and velocities \[(x^0,v^0)\in[-1,1]^{2N}\times[-1,1]^{2N}\,,\]
interacting by means of the kernel \eqref{eq:cuckerkernel} with $H = 1$, $\sigma = 1$ and $\beta=1$.
We recall that relevant quantities for the analysis of our results are given by (here we stress the dependance on $x$ and $v$)
\begin{align*}
X[x](t)\define\frac{1}{2N^2}\sum_{i,j=1}^N\|x_i(t)-x_j(t)\|^2\quad\text{and}\quad V[v](t)\define\frac{1}{2N^2}\sum_{i,j=1}^N\|v_i(t)-v_j(t)\|^2\,.
\end{align*}
Notice that, once a random initial configuration has been generated, it is possible to rescale it to a desired $(X_0,V_0)$ parametric pair, by means of
\[
(x,v)=\left(\sqrt{\frac{X_0}{X[\tilde{x}]}}\tilde{x}, \sqrt{\frac{V_0}{V[\tilde{v}]}} \tilde{v}\right)\,,
\]
such that $(X[x],V[v])=(X_0,V_0)$. As simulations of the trajectories have been generated by prescribing a value for the pair $(X_0,V_0)$, which is used to rescale randomly generated initial conditions, there are slight variations on the initial positions and velocities in every model run, which can affect the final consensus direction. However our results are stated in terms of $X,V$, and independently of the specific initial configuration. For simulation purposes  the system is integrated in time with the specific feedback control by means of a Runge-Kutta 4th-order scheme.

%\subsection{Local feedback control}\label{sim:local}
%The last numerical case study presented concerns the dynamical system considered in Section \ref{sec:localmeanR}, where the feedback is computed according to the local average \eqref{eq:truncated_perturbation}.

%From a theoretical perspective, this result is presented in Theorem \ref{th:HaHaKimExtended}, which describes a consensus region for a feedback control based on local averages. %This theorem recovers on its asymptotics previous results in \cite{HaHaKim}, and \cite{caponigro2015sparse}, related to consensus regions for uncontrolled and fully controlled systems under total information feedback, respectively.
As it was shown in Example \ref{ex:notsharp}, %(and as it has been reported in the literature \cite{cafotove10})
that estimates for consensus regions such as the one provided by Theorem \ref{thm:hhk}, are not sharp in many situations. In this direction, we proceed to contrast the theoretical consensus estimates with the numerical evidence. For this purpose, for a fixed number of agents, we span a large set of possible initial configurations determined by different values of $(X_0,V_0)$. For every pair $(X_0,V_0)$ we randomly generate a set of 20 initial conditions, and we simulate for a sufficiently large time frame. We measure consensus according to a threshold established on the final value of $V$; we consider that consensus has been achieved if the final value of $V$ is lower or equal to $1e-5$. We proceed by computing empirical probabilities of consensus for every point of our state space $(X_0,V_0)$;  results in this direction are presented in Figures \ref{fig:4} and \ref{fig:5}. We first consider the simplified case of 2 agents; according to Example \ref{ex:notsharp}, for this particular case, the consensus region estimate provided by Theorem \ref{thm:hhk} is sharp, as illustrated by the results presented in Figure \ref{fig:4}. Furthermore, it is also the case for Theorem \ref{th:HaHaKimExtended}; for $R>0$, the predicted consensus region coincides with the numerically observed ones.

\begin{figure}[!ht]
\centering
%\resizebox{\textwidth}{!}{
%\begin{tabular}{cc}
%\epsfig{file=e20.eps,width=0.49\linewidth,clip=}
%\epsfig{file=t20.eps,width=0.49\linewidth,clip=}\\
%\epsfig{file=e22.eps,width=0.49\linewidth,clip=}
%\epsfig{file=t22.eps,width=0.49\linewidth,clip=}
\includegraphics[width = 0.49\textwidth]{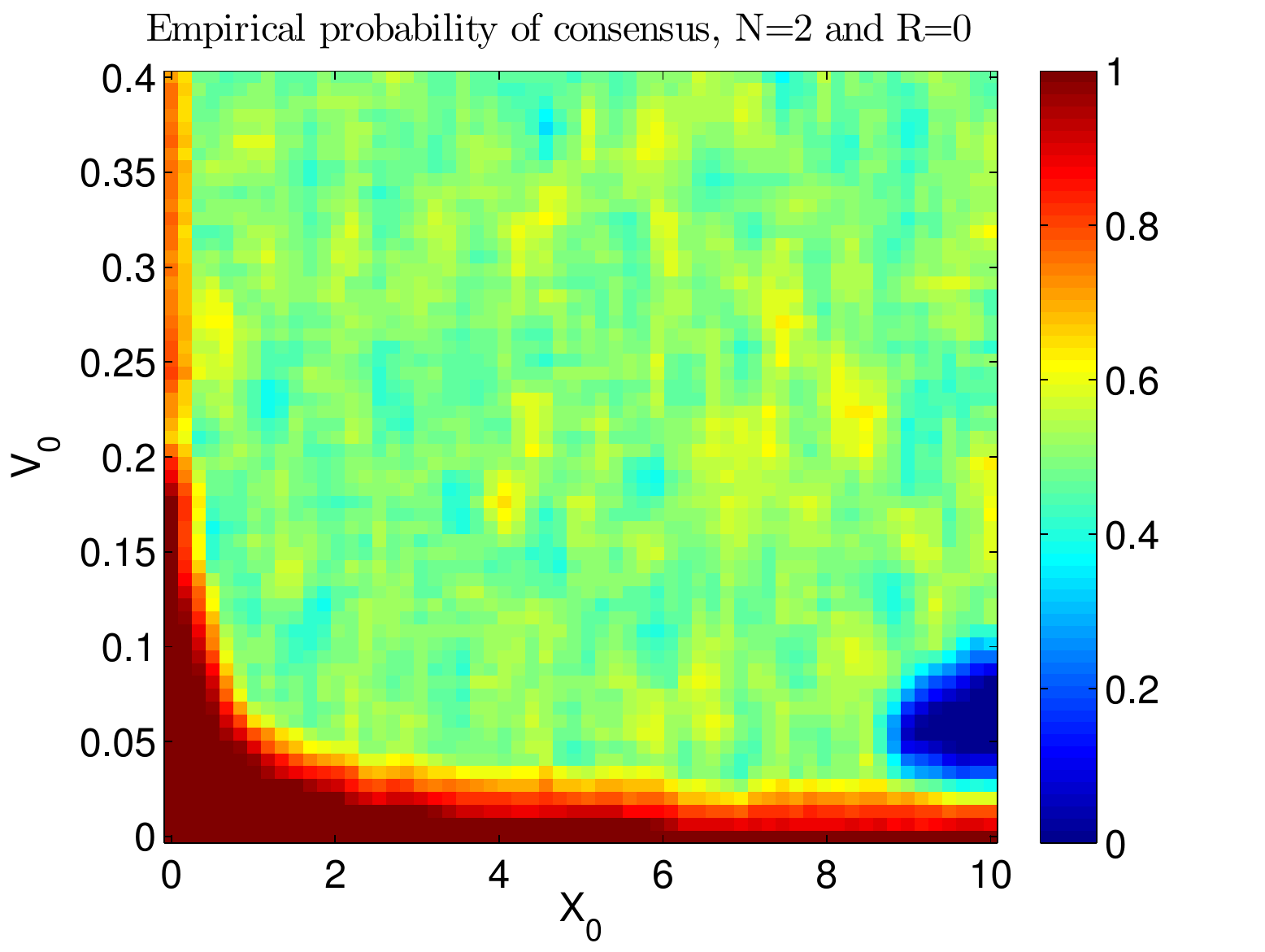}
\includegraphics[width = 0.49\textwidth]{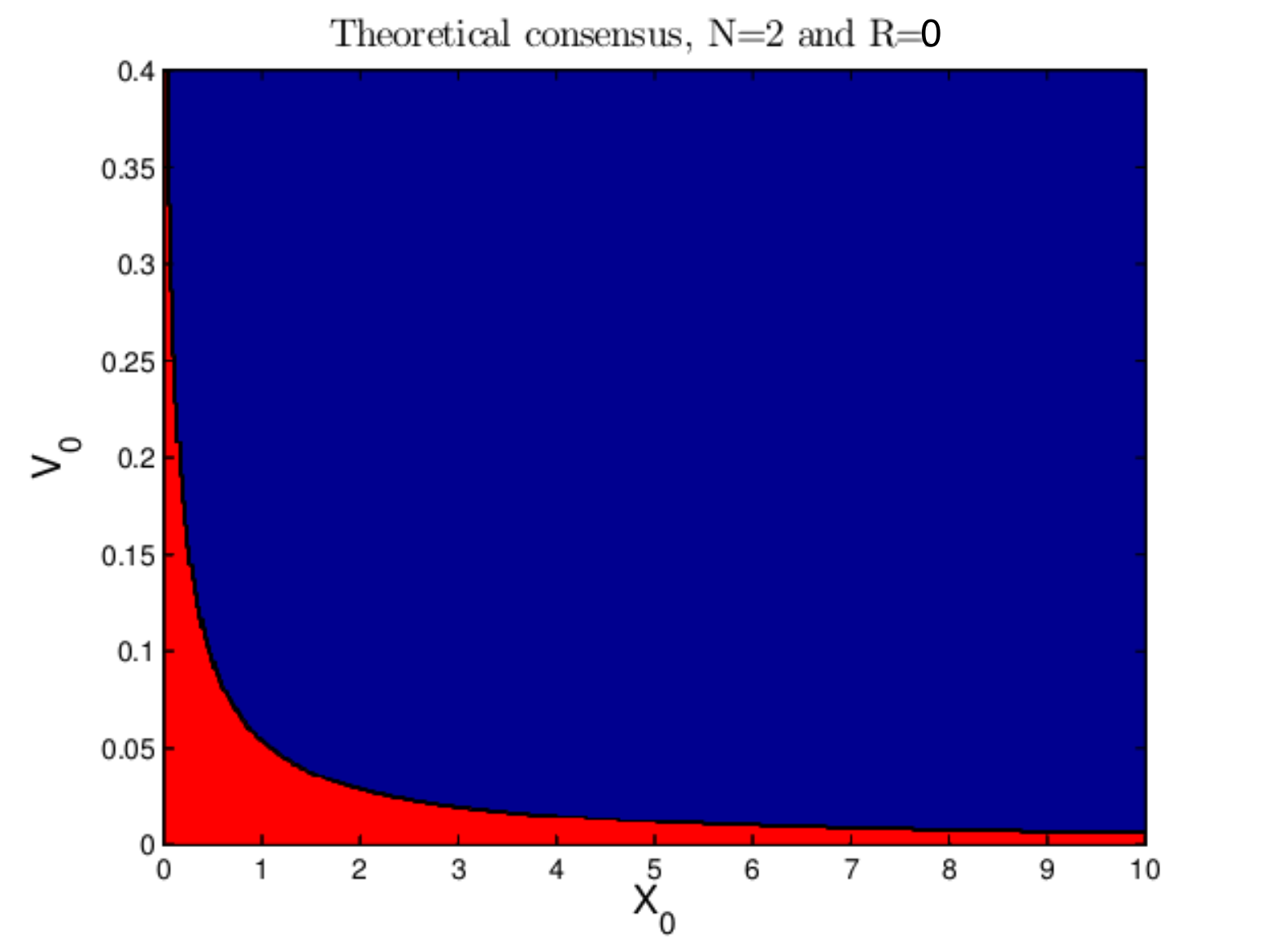}\\
\includegraphics[width = 0.49\textwidth]{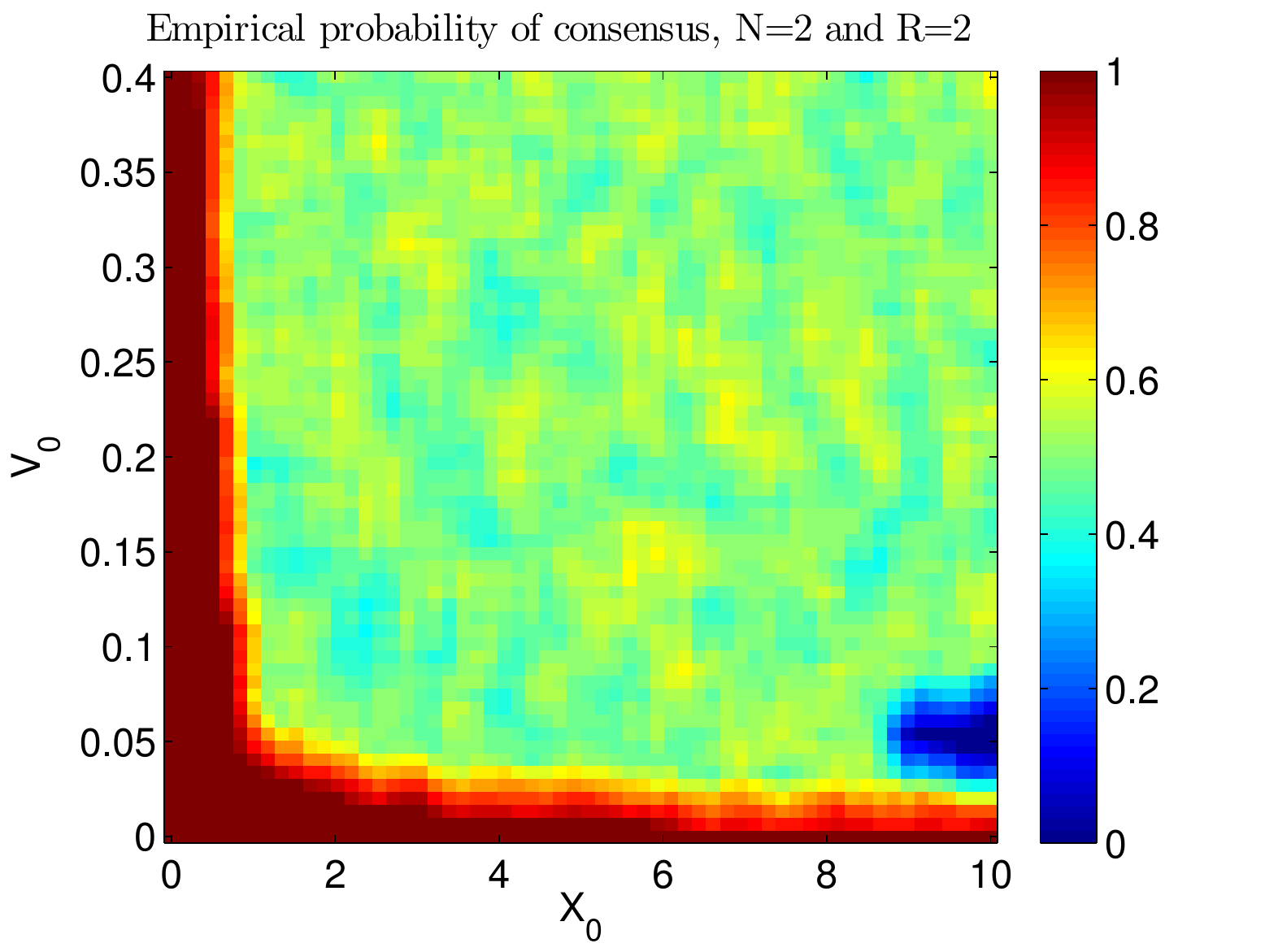}
\includegraphics[width = 0.49\textwidth]{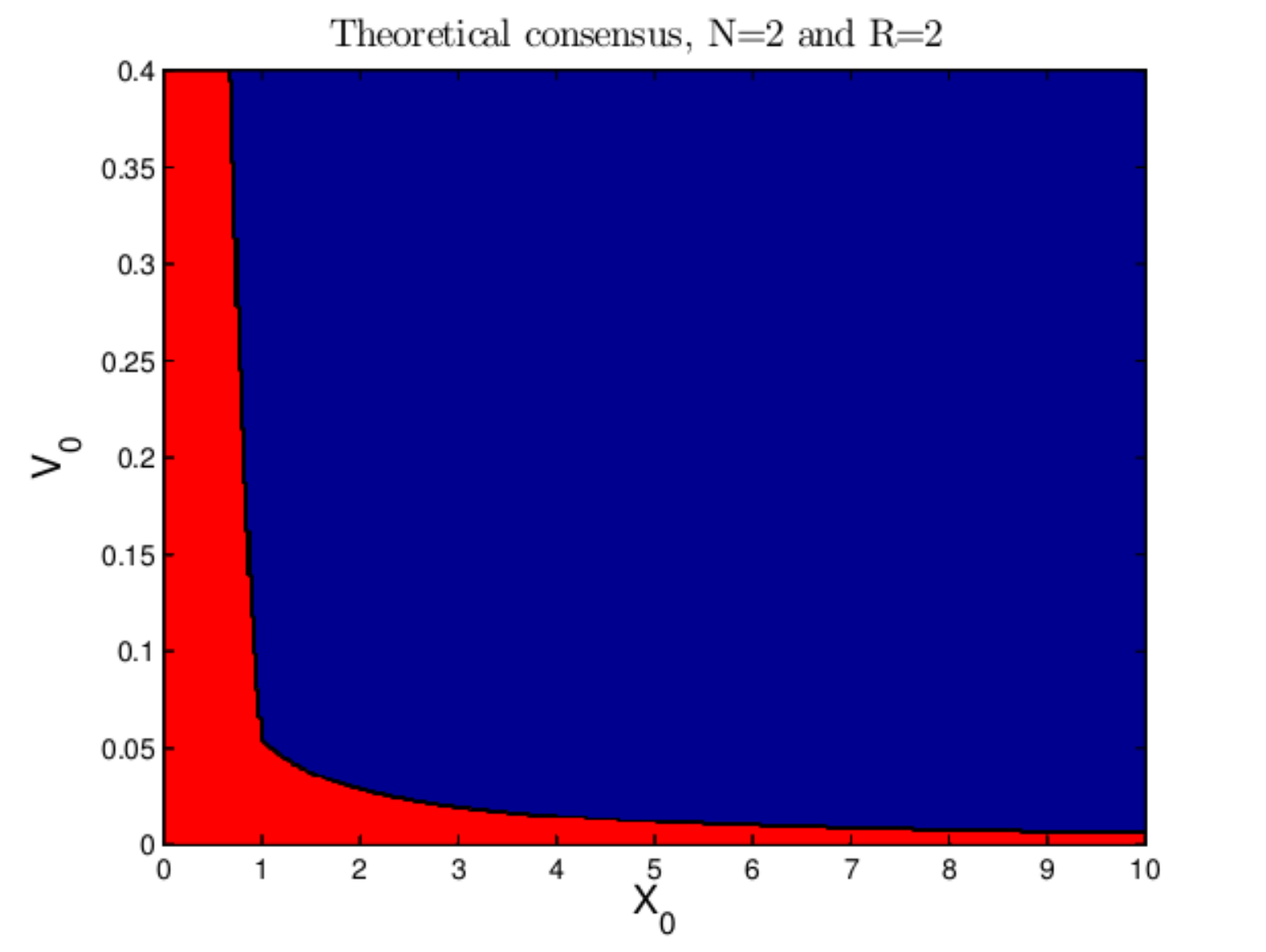}
%\end{tabular}}
\caption{Local feedback control. Empirical consensus regions and theoretical estimates for two-agent systems.}
\label{fig:4}
\end{figure}

Figure \ref{fig:5} illustrates the case when a larger number of agents is considered. In a similar way as for Theorem \ref{thm:hhk}, the consensus region estimate is conservative if compared with the region where numerical experiments exhibit convergent behavior. Nevertheless, Theorem \ref{th:HaHaKimExtended} is consistent in the sense that the theoretical consensus region increases gradually as $R$ grows, eventually covering any initial configuration, which is the case of the total information feedback control, as presented in \cite[Proposition 2]{caponigro2015sparse}. The numerical experiments also confirm this phenomena, as shown in Figure \ref{fig:6}, where contour lines showing the $80\%$ probability of consensus for different radii locate farther from the origin as $R$ increases.

\begin{figure}[!ht]
\centering
%\resizebox{\textwidth}{!}{
%\begin{tabular}{cc}
%\epsfig{file=r020.eps,width=0.49\linewidth,clip=}
%\epsfig{file=t020.eps,width=0.49\linewidth,clip=} \\
%\epsfig{file=r120.eps,width=0.49\linewidth,clip=}
%\epsfig{file=t120.eps,width=0.49\linewidth,clip=} \\
%\epsfig{file=r220.eps,width=0.49\linewidth,clip=}
%\epsfig{file=t220.eps,width=0.49\linewidth,clip=}
\includegraphics[width = 0.49\textwidth]{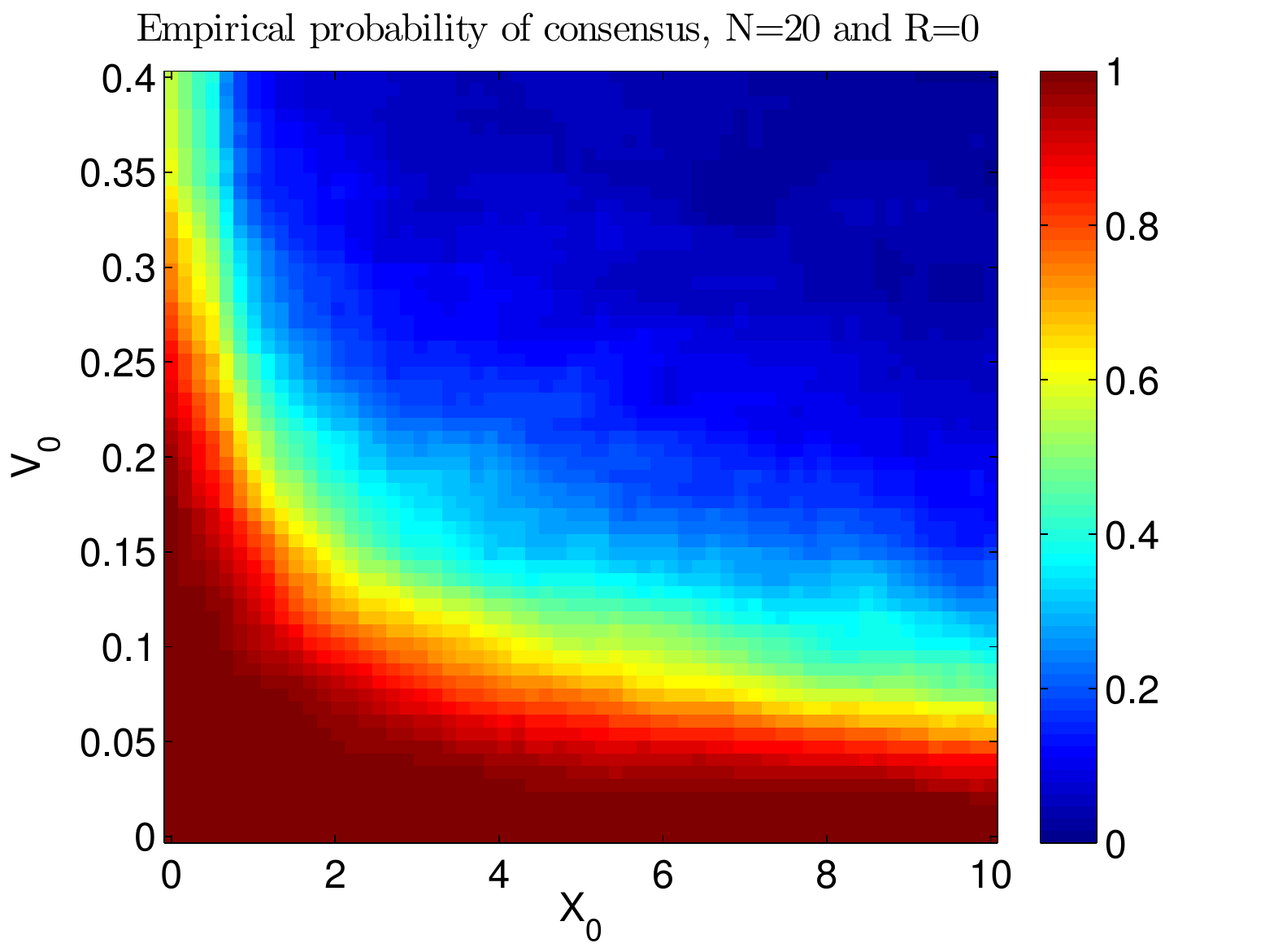}
\includegraphics[width = 0.49\textwidth]{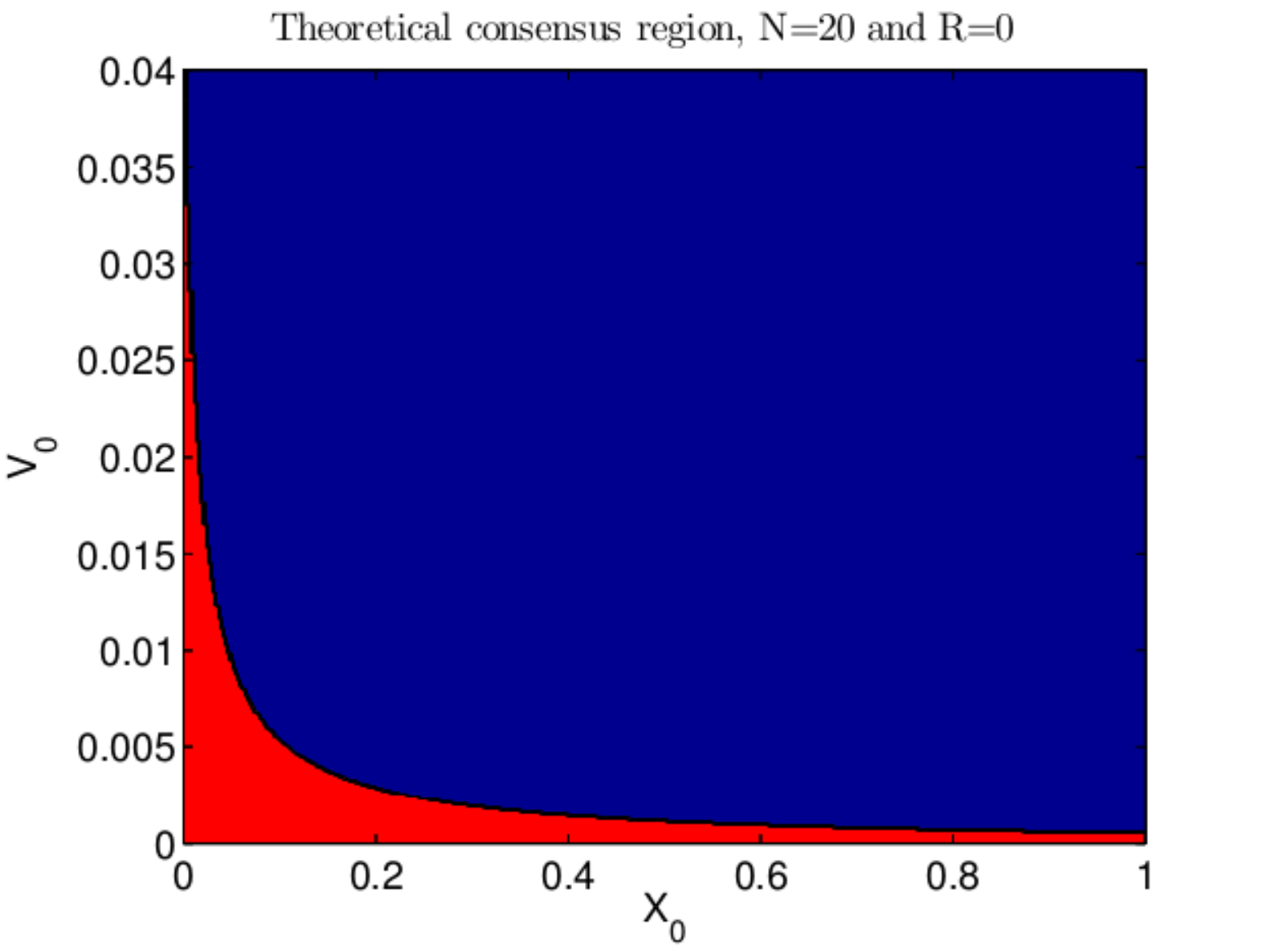}\\
\includegraphics[width = 0.49\textwidth]{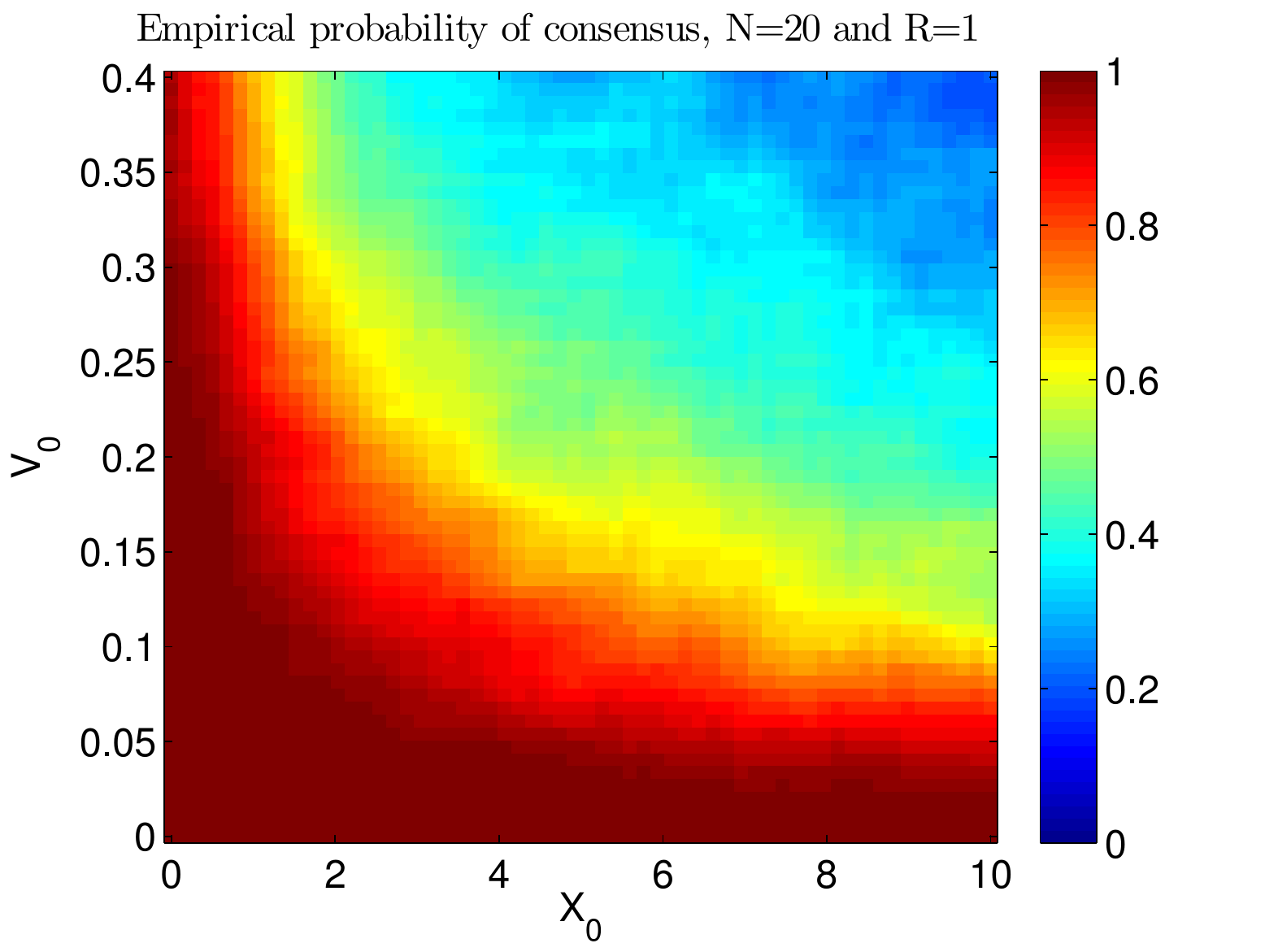}
\includegraphics[width = 0.49\textwidth]{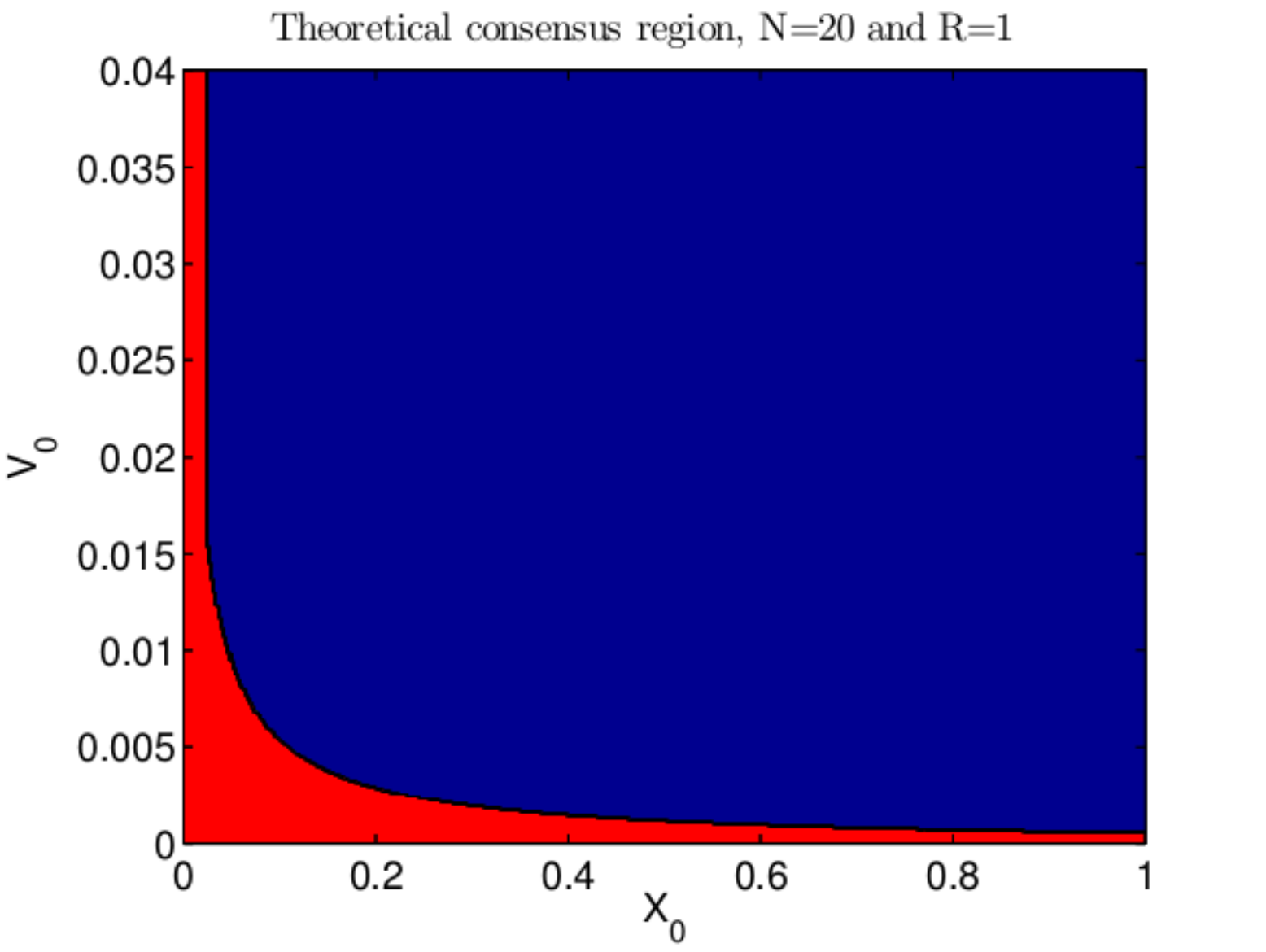}\\
\includegraphics[width = 0.49\textwidth]{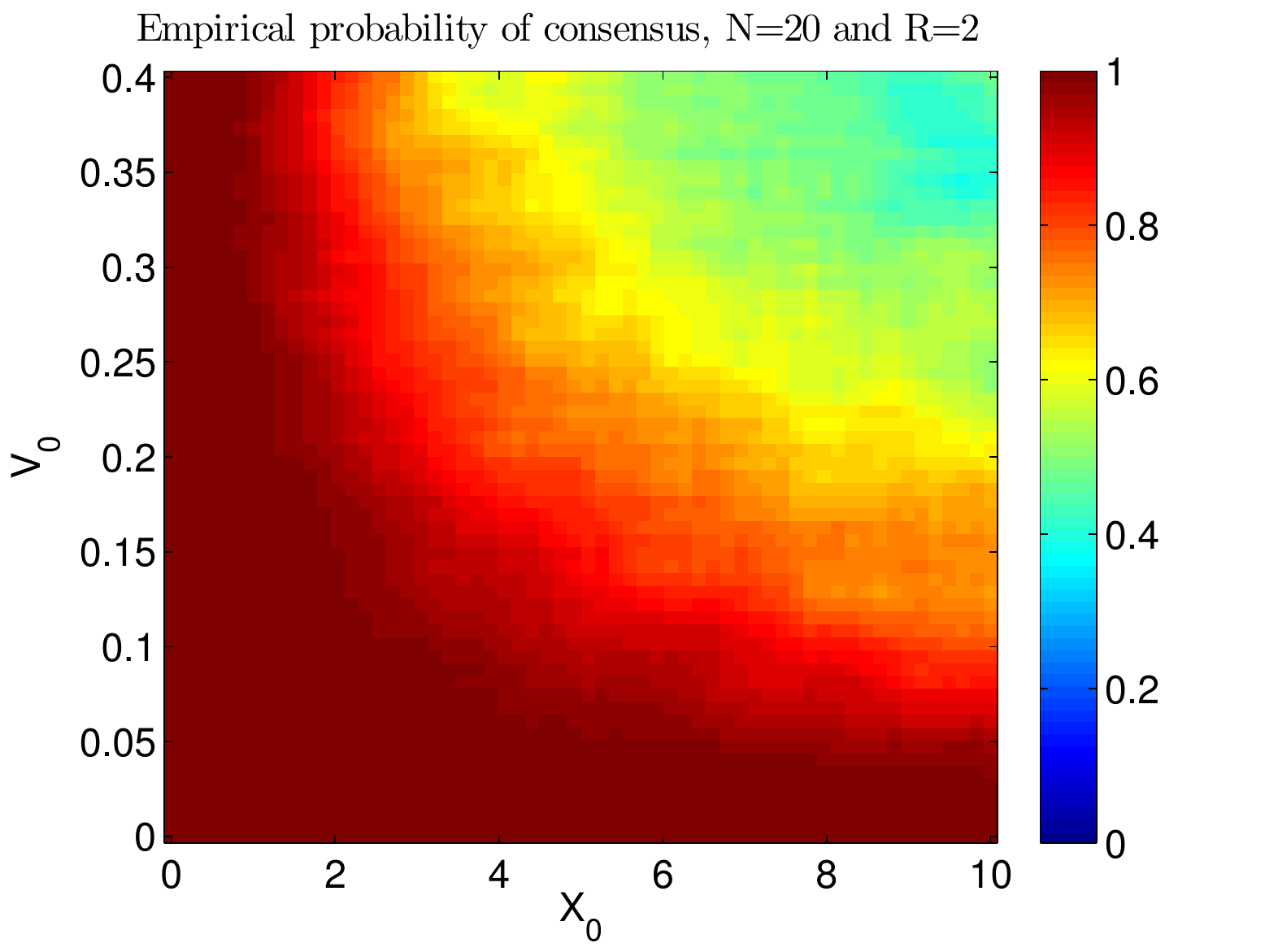}
\includegraphics[width = 0.49\textwidth]{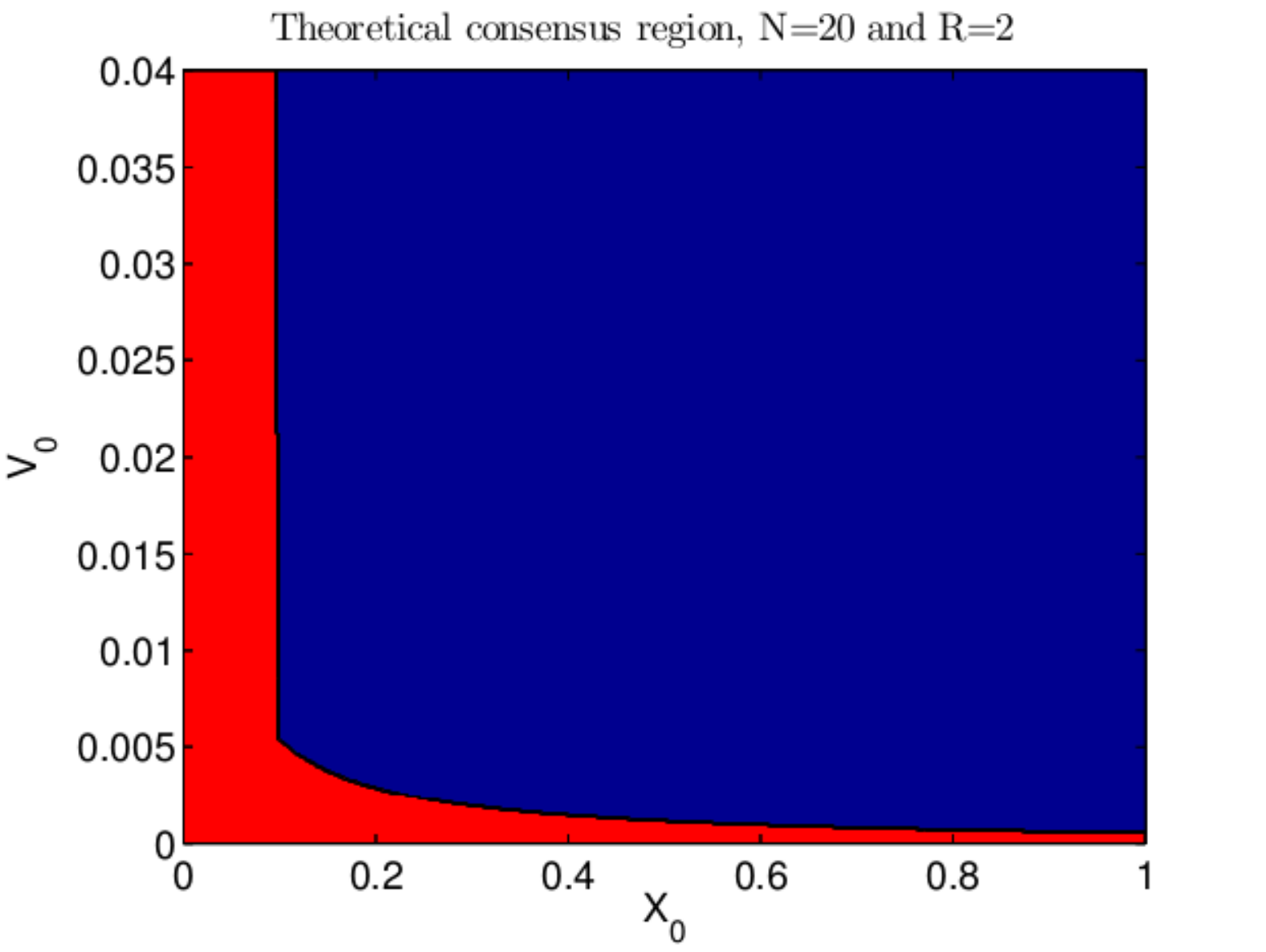}
%\end{tabular}}
\caption{Local feedback control. Empirical consensus regions and theoretical estimates for $N=20$ agents and different control radii $R$.}
\label{fig:5}
\end{figure}

\begin{figure}[!ht]
\centering
\includegraphics[width = 0.5\textwidth]{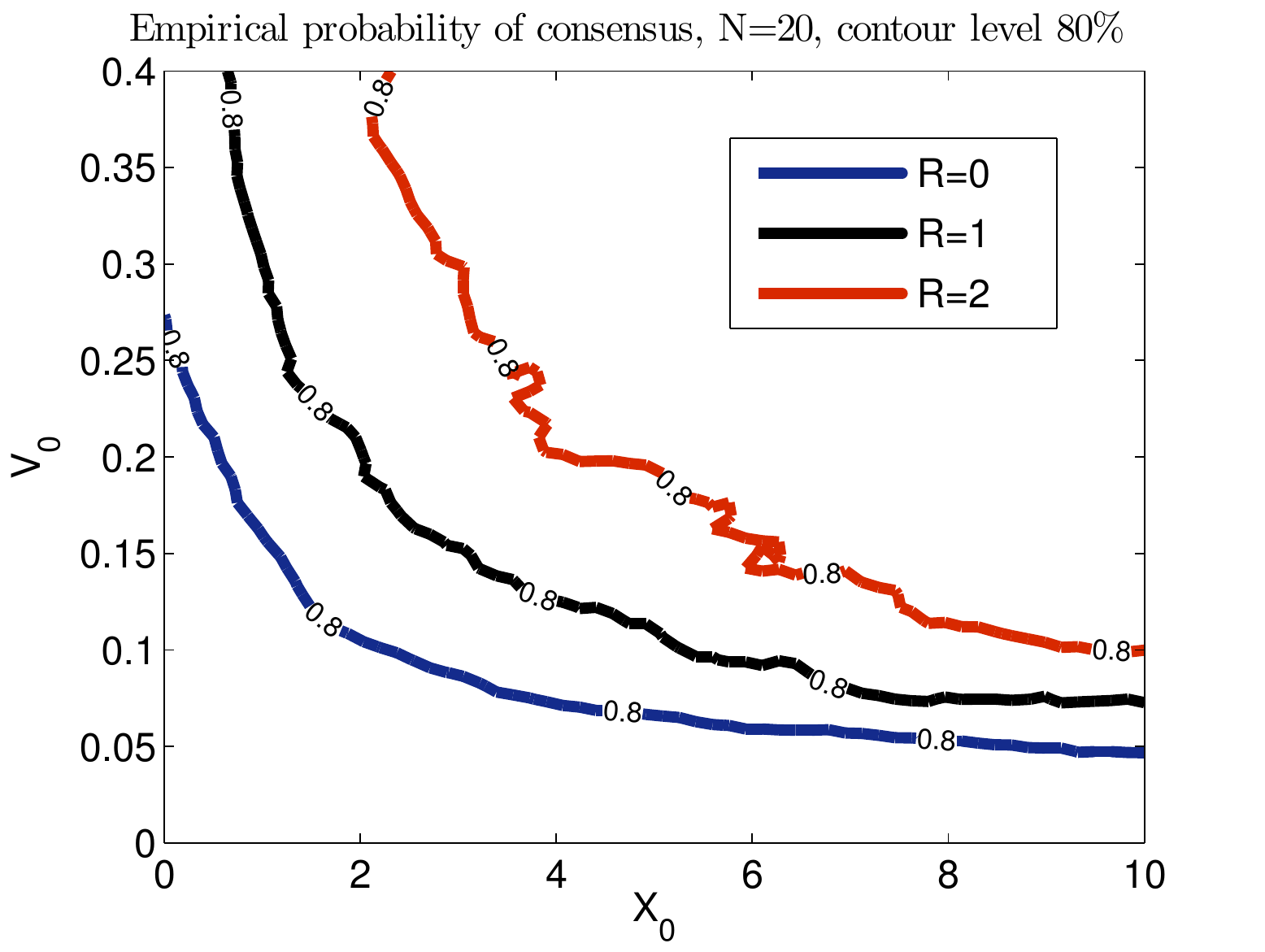}
\caption{Local feedback control. Empirical contour lines for the 80$\%$ probability of consensus with different control radii.}
\label{fig:6}
\end{figure}

%\graphicspath{{./cuckerdong/}}

%\chapter{Sparse stabilization of dynamical systems}\label{ch:sparsecontrol}
%\chapter{Sparse control of dynamical systems}\label{ch:sparsecontrol}

%In Chapter \ref{ch:decentralized} we have seen the limitations of decentralized controls in guaranteeing consensus emergence. To overcome them, in this chapter we shall introduce an external policy maker overseeing the evolution of the dynamics with the task to help the agents coordinate. %and in charge of the control of the system.
%The \textit{centralized} control strategy we will design is parsimonious, optimal in terms of maximizing the decay rate of a certain Lyapunov functional related to the pattern we want to enforce, and numerically easy to implement.

\section{Sparse control of the Cucker-Smale model}\label{sec:cuckersparse}

We have seen throughout the previous sections how difficult it is to ensure unconditional convergence to consensus for alignment models. In particular, in Section \ref{sec:localmeanR} we have proven that the addition of a local feedback does not always help: Theorem \ref{th:HaHaKimExtended} shows that we can guarantee unconditional convergence to consensus with respect to the initial datum for dynamical systems of the form
\begin{align*}
\left\{
\begin{aligned}
\begin{split}
\dot{x}_{i}(t) & = v_{i}(t), \\
\dot{v}_{i}(t) & = \frac{1}{N} \sum_{j = 1}^N a\left(\vnorm{x_i(t) - x_j(t)}\right)\left(v_{j}(t)-v_{i}(t)\right) + \gamma\left(\frac{1}{|\Lambda_R(t,i)|} \sum_{j \in \Lambda_R(t,i)} v_j(t) - v_i(t)\right).
\end{split}
\end{aligned}
\right.
\end{align*}
only in the case $R = +\infty$, for which the identity
$$\frac{1}{|\Lambda_R(t,i)|} \sum_{j \in \Lambda_R(t,i)} v_j(t) = \overline{v}(t)$$
holds. This means that either the agents have perfect information of the state of the entire system (so that the local mean $\overline{v}_i$ is equal to the true mean $\overline{v}$) or, as the numerical simulations in Section \ref{sim:local} show, there are situations where the agents are not able to converge to consensus. As already pointed out in Section \ref{sec:first_results}, this is a very strong requirement to ask for, and not many real-life scenarios are able to support it. Consider, for instance, the case of an assembly of people trying to reach an unanimous decision, like the European Union Council: since the extra term can be interpreted as an additional desire of each agent to agree with people whose goal is near to his, the requirement $R = +\infty$ corresponds to asking that all the individual goals are close, i.e., all agents pursue the same end. A truly imaginative world indeed! We are thus facing an inherent, severe limitation of the decentralized approach.

\subsection{Centralized feedback interventions}

To overcome this apparent dead-end, let us write $u_i(t) = \gamma(\overline{v}(t) - v_i(t))$, i.e.,
\begin{align}\label{eq:cuckercontr}
\left\{
\begin{aligned}
\begin{split}
\dot{x}_{i}(t) & = v_{i}(t), \\
\dot{v}_{i}(t) & = \frac{1}{N} \sum_{j = 1}^N a\left(\vnorm{x_i(t) - x_j(t)}\right)\left(v_{j}(t)-v_{i}(t)\right) + u_i(t).
\end{split}
\end{aligned}
\right.
\end{align}
Instead of interpreting $u_i$ as a \textit{decentralized} force, let us consider it as an \textit{external} force from an outside source acting on the system to help it to coordinate. This new approach sheds a completely different light on the problem: %This is a completely different interpretation of the problem:
with respect to the example considered before, is like introducing a moderator heading the discussion, who can make pressure on the participants to the council facilitating the consensus process. Adding an external figure implementing intervention policies broadens further the expressive power of the problem: indeed, since we are in principle no more tied to specific interventions of the form $u_i(t) = \gamma(\overline{v}(t) - v_i(t))$, this setting enables us to ask ourselves the following question
\medskip
\begin{center}
(Q) \textit{ given a set of constraints, which control $u$ is the best to reach a specific goal?}
\end{center}
\medskip

In this section, we shall study a specific instance of this very general issue in the case of system \eqref{eq:cuckercontr}. In our setting, the constraints shall be
\begin{enumerate}[label=$(\roman*)$]
\item the control is of \textit{feedback-type}, i.e., computed instantaneously as a function of the state variables, following a \textit{locally optimal} criterion;%we are not able to forecast the evolution of the system (for computational reasons or for lack of knowledge of the interaction mechanism, for instance). This means that, since we do not know how the system may react in the future to our intervention, we focus our attention on \textit{locally optimal} control strategies.% techniques and results from Optimal Control are not available. %We shall see how to overcome this inadequacy in Chapter \ref{ch:learning};
\item there is a maximal amount of resources $M > 0$ that the central policy maker can spend at any given time for the intervention;
\item the control should act on the least amount of agents possible at any time.
\end{enumerate}
For the time being, our goal is again alignment, hence we seek for a control $u$ for which the associated solution to system \eqref{eq:cuckercontr} tends to consensus in the sense of Definition \ref{def:consensus}. We have seen in Proposition \ref{p:equivconsensus} that an effective criterion for consensus emergence is the minimization of the Lyapunov functional $V$: if we are able to prove that our control strategy is able to drive $V$ below the threshold level given by Theorem \ref{thm:hhk}, we have automatically consensus emergence (see Figure \ref{consreg}). The maximization of the decay rate of $V$ is a locally optimal criterion, % formulated on an instant-by-instant basis,
and hence compatible with point ($i$). %; a globally optimal one, available when predictions on the possible evolution of the system are possible, shall be discussed in Chapter \ref{ch:mfpmp}.

\begin{figure}[!htb]
\centering
\includegraphics[scale=0.8]{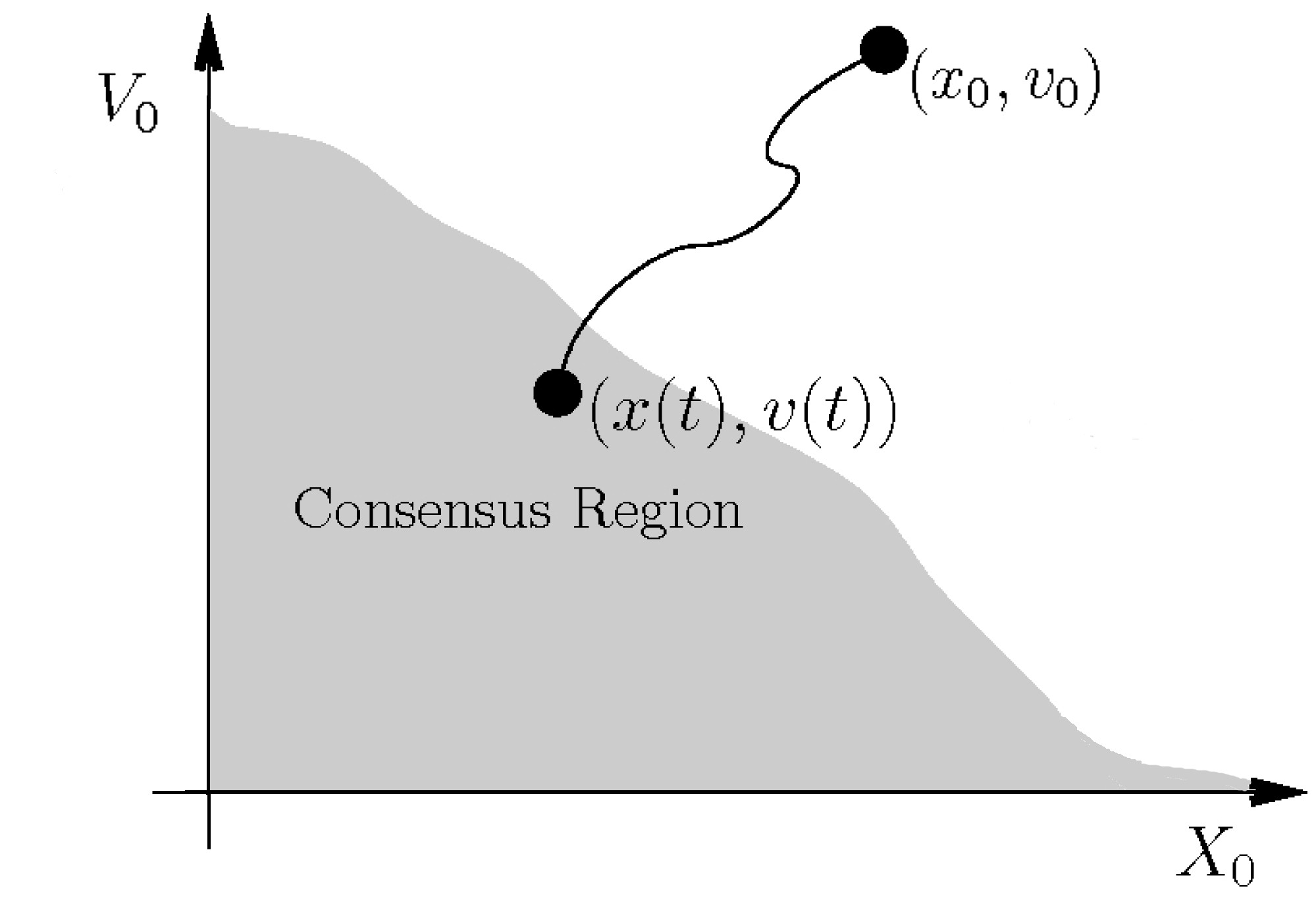}
\caption{Steering the system to a point fulfilling the conditions of Theorem \ref{thm:hhk}.} 
\label{consreg}
\end{figure}

The following preliminary estimate shows the effect of a control on $V$.
%As a preliminary result, we can estimate the effect of a control on the growth of $V$.
\begin{lemma}\label{le:wichtig}
For any measurable function $u:\R_+\rightarrow\R^{dN}$ it holds
\begin{align*}%\label{eq:lyapunovdecay}
\frac{d}{dt}V(t)\leq 2B(u(t),v(t)).
\end{align*}
\end{lemma}
\begin{proof}
Using the representation of system \eqref{eq:cuckersmale} in Laplacian form \eqref{systemmatrix}, and the fact that $L(x(t))$ is positive definite, we get
$$\frac{d}{dt}V(t) = \frac{d}{dt}B(v(t),v(t)) = -2B(L(x(t))v(t),v(t)) + 2B(u(t),v(t)) \leq 2B(u(t),v(t)).$$
This concludes the proof.
\end{proof}

The constraint on the maximal amount of available resources $M$ given by point ($ii$) leads to the following definition of \textit{admissible controls}.

\begin{definition}[Admissible controls]\label{def:admcontrols}
A measurable function $u = (u_1,\ldots,u_N):\R_+ \funarrow \R^{dN}$ is an \textit{admissible control} if it satisfies
\begin{align}\label{eq:controlbounded}
\sum^N_{i = 1} \|u_i(t)\| \leq M \quad \text{ for every } t \geq 0.
\end{align}
\end{definition}

As an immediate corollary of Lemma \ref{le:wichtig} we can show 
%The following result shows
that the problem of finding admissible controls steering the system to consensus is well-posed.

\begin{corollary}[Total control, {\cite[Proposition 2]{caponigro2015sparse}}]\label{pr:inutile}
Fix $M >0$, an initial condition $(x^0,v^0) \in \R^{dN}\times\R^{dN}$, and $0 < \alpha \leq M/(N\sqrt{V_0})$. Then, the feedback control defined pointwise in time as
\begin{align}\label{eq:totalcontrol}
u(t) = -\alpha v^{\perp}(t) \quad \text{ for every } t \geq 0,
\end{align}
is admissible and the solution associated to $u$ %. Moreover, system \eqref{eq:cuckercontr} admits a unique solution associated to  $u$ which
tends to consensus.
\end{corollary}
\begin{proof}
Let $(x,v):\R_+\funarrow\R^{dN}\times\R^{dN}$ be a solution of system \eqref{eq:cuckercontr} with $u$ as in the statement. %Since with this choice of $u$ we retrieve system \eqref{eq:cuckersmale_uniform}, such solution exists and is unique by \cite[Chapter 1, Theorems 1, 2, 4]{filipov}. Moreover,
Lemma \ref{le:wichtig} implies that
$$\frac{d}{dt}V(t) \leq 2B(u(t),v(t)) = - 2\alpha B(v^{\perp}(t),v(t)) = -2\alpha V(t);$$
Therefore, an application of Gronwall's Lemma yields $V(t) \leq e^{-2 \alpha t}V(0)$, so $V(t)$ tends to $0$ exponentially fast as $t\rightarrow+\infty$. In particular, $X(t)$ keeps bounded and the trajectory reaches the consensus region in finite time. Lastly, it follows that %since $V$ is decreasing by Lemma \ref{lem:bigv_growth}, it holds
\begin{align*}
\sum^N_{i = 1} \|u_i(t)\| \leq \sqrt{N}\sqrt{\sum^N_{i = 1}\|u_i(t)\|^2} = \alpha \sqrt{N}\sqrt{\sum^N_{i = 1}\|v^{\perp}_i(t)\|^2} = \alpha N \sqrt{V(t)} \leq \alpha N \sqrt{V_0} \leq M,
\end{align*}
which implies the admissibility of the control.
\end{proof}

Corollary \ref{pr:inutile}, although very simple, is somehow remarkable: not only it shows that we can steer to consensus the system from any initial condition, but that the strength of the control $M > 0$ can be arbitrarily small.
However, this result has perhaps only theoretical validity, because the stabilizing control $u = -\alpha v^{\perp}$ needs to act instantaneously on all the agents, thus requires the external policy maker to interact at every instant with all the agents in order to steer the system to consensus, a procedure that requires a large amount of instantaneous communications, whence the name of \textit{total control}. 
%However this is not enough. Even if controls can be successful in enforcing alignment, they may be unfeasible in practice. Undesirable controls are, for example, those that force the central coordinator to interact at every instant with \textit{almost all} the agents in the system, like the control of Corollary \ref{pr:inutile}, since keeping such a large communication network always active can be extremely expensive and highly dispersive.
This motivates point ($iii$) and is the reason why we look for interventions that target the fewest number of agents at any given time. However, this leads us into the difficult combinatorial problem of the selection of the best few control components to be activated. How can we solve it?

The problem resembles very much the one in information theory of finding the best possible sparse representation of data in form of vector coefficients with respect to an adapted dictionary for the sake of their compression, see \cite[Chapter 1]{mallat2008wavelet}. In our case, the relationship between control choices and result will be usually highly nonlinear, especially for several known dynamical systems modeling social dynamics: were this relationship more simply linear instead, then a rather well-established theory would predict how many degrees of freedom are minimally necessary to achieve the expected outcome. Moreover, depending on certain spectral properties of the linear model, the theory allows also for efficient algorithms to compute the relevant degrees of freedom, relaxing the associated combinatorial problem. This theory is known in mathematical signal processing and information theory under the name of \textit{compressed sensing}, see the seminal work \cite{tao,donoho} and the review chapter \cite{fora10}. The major contribution of these papers was to realize that one can combine the power of convex optimization, in particular $\ell_1$-norm minimization, and spectral properties of random linear models in order to achieve optimal results on the ability of $\ell_1$-norm minimization of recovering robustly linearly constrained sparsest solutions. Borrowing a leaf from compressed sensing, we model sparse stabilization and control strategies by penalizing the class of vector-valued controls $u = (u_1, \ldots, u_N) \in \R^{dN}$ by means of the mixed $\ell^N_1-\ell^d_2$-norm
%To translate these words into a mathematical statement we borrow a leaf from \textit{compressed sensing} (see as a reference \cite{tao,donoho,fora10}), a branch of mathematical signal processing which provides optimal results on the ability of %the combination of convex optimization and of the spectral properties of random linear models is capable of proving that
%the $\ell_1$-norm minimization of robustly recovering \textit{sparse} solutions, i.e., with very few non-zero entries. One can obtain control strategies $u = (u_1, \ldots, u_N) \in \R^{dN}$ acting on very few agents by requiring the following mixed $\ell^N_1-\ell^d_2$-norm of $u$
\begin{align*}
\sum^N_{i = 1} \|u_i\|_{\ell^d_2}
\end{align*}
%to be as small as possible: the $\ell_1$-norm constraint on the number of components relative to the agents will select controls that at every instant interact with very few of them. 
The above mixed norm has been already used, for instance, in \cite{elra10} to optimally sparsify multivariate vectors in compressed sensing problems, or in \cite{fora08} as a \textit{joint sparsity} constraint. The use of $\ell_1$-norms to penalize controls was first introduced in the seminal paper \cite{crlo65} to model linear fuel consumption, while lately the use of $L^1$ minimization in optimal control problems with partial differential equation has become very popular, for instance in the modeling of optimal placing of sensors \cite{caclku12,clku12,hestwa12,st09,wawa11}.

\subsection{Sparse feedback controls}

%needs to be informed of the status of the entire system.
We wonder whether we can stabilize the system by means of interventions that are more parsimonious than the total control, since they are more realistically modeling actual government actions. From %the instructive proof above
Lemma \ref{le:wichtig}, we learn that a good strategy to steer the system to consensus is actually the minimization of $B(u(t),v(t))$ with respect to $u$, for all $t$.
%As already noted, Corollary \ref{pr:inutile} has merely a theoretical value: the feedback control $u = -\alpha v^{\perp}$ is not convenient for practical purposes, since it requires the external policy maker to interact at every instant with all the agents in order to steer the system to consensus, a procedure that requires a large amount of instantaneous communications.
For this reason, we choose controls according to a specific variational principle leading to a componentwise sparse stabilizing feedback law.

\begin{definition}\label{def:cuckervariational}
For every $M>0$ and every $(x,v)\in\R^{dN}\times\R^{dN}$, let $U(x,v)$ be the set of solutions of the variational problem
\begin{align}\label{eq:variationalprinciple}
\min_{u\in\R^{dN}} \left( B(u,v) + \gamma(B(x,x))\frac{1}{N}\sum^N_{i = 1}\|u_i\| \right) \quad \text{ subject to } \sum^N_{i = 1} \|u_i\| \leq M,
\end{align}
where the \textit{threshold functional} $\gamma$ is defined as
\begin{align*}%\label{eq:thresholdgamma}
\gamma(X) \define \int^{\infty}_{\sqrt{X}}a(\sqrt{2N}r)dr.
\end{align*}
\end{definition}

Notice that the variational principle \eqref{eq:variationalprinciple} is balancing the minimization of $B(u, v)$, which we mentioned above as relevant to promote convergence to consensus, and the $\ell_1$-norm term $\sum^N_{i = 1}\|u_i\|$, expected to promote sparsity.

Each value of $\gamma(B(x,x))$ yields a partition of $\R^{dN} \times \R^{dN}$ into four disjoint sets:
\begin{description}
	\item[$\mathcal{P}_1$] $\define \{ (x,v) \in \R^{dN} \times \R^{dN} : \max_{1 \leq i \leq N} \vnorm{v^{\perp}_i} < \gamma(B(x,x))^2 \}$,
	\item[$\mathcal{P}_2$] $\define \{ (x,v) \in \R^{dN} \times \R^{dN} : \max_{1 \leq i \leq N} \vnorm{v^{\perp}_i} = \gamma(B(x,x))^2 \text{ and } \exists k \geq 1 \text{ and } i_1, \ldots, i_k \break\in \{1, \ldots N\}$ $\text{ such that } \vnorm{v^{\perp}_{i_1}} = \ldots = \vnorm{v^{\perp}_{i_k}} \text{ and } \vnorm{v^{\perp}_{i_1}} > \vnorm{v^{\perp}_j} \text{ for every } j \not \in \{i_1, \ldots, i_k\}\}$,
	\item[$\mathcal{P}_3$] $\define \{ (x,v) \in \R^{dN} \times \R^{dN} : \max_{1 \leq i \leq N} \vnorm{v^{\perp}_i} > \gamma(B(x,x))^2 \text{ and } \exists ! i \in \{1, \ldots N\} \text{ such} \\ \text{ that } \vnorm{v^{\perp}_i} > \vnorm{v^{\perp}_j} \text{ for every } j \not = i \}$,
	\item[$\mathcal{P}_4$] $\define \{ (x,v) \in \R^{dN} \times \R^{dN} : \max_{1 \leq i \leq N} \vnorm{v^{\perp}_i} > \gamma(B(x,x))^2 \text{ and } \exists k > 1 \text{ and } i_1, \ldots, i_k\break \in \{1, \ldots N\}$ $\text{ such that } \vnorm{v^{\perp}_{i_1}} = \ldots = \vnorm{v^{\perp}_{i_k}} \text{ and } \vnorm{v^{\perp}_{i_1}} > \vnorm{v^{\perp}_j} \text{ for every } j \not \in \{i_1, \ldots, i_k\} \}$,
\end{description}

Moreover, since we are minimizing $B(u,v) = B(u,v^{\perp})$, it is easy to see that, for every $(x, v) \in \R^{dN}\times\R^{dN}$ and every element $u(x,v)=(u_1(x,v), \dots, u_N(x,v))^T\in U(x, v)$ there exist nonnegative real numbers $\eps_i \geq 0$ such that, for every $i = 1, \ldots, N$, it holds
\begin{align}\label{eq:formcontrolcucker}
u_i(x,v) =\begin{cases}
	\displaystyle -\eps_i \frac{v^{\perp}_i}{\|v^{\perp}_i\|} & \quad \text{  if } \|v^{\perp}_i\| \not = 0, \\
	0 & \quad \text{  if } \|v^{\perp}_i\| = 0,
	\end{cases}
\end{align}
where $0 \leq \sum^N_{i = 1} \eps_i \leq M$. The values of the $\eps_i$'s can be determined on the basis of which partition $(x,v)$ belongs to:
\begin{itemize}
\item if $(x,v) \in \mathcal{P}_1$ then $\eps_i = 0$ for every $i = 1, \ldots, N$;
\item if $(x,v) \in \mathcal{P}_2$ then indicating with $i_1, \ldots, i_k$ the indexes such that $\vnorm{v^{\perp}_{i_1}} = \ldots = \vnorm{v^{\perp}_{i_k}} = \gamma(B(x,x))$ and $\vnorm{v^{\perp}_{i_1}} > \vnorm{v^{\perp}_j}$ for every  $j \not \in \{i_1, \ldots, i_k\}$, we have $\eps_j = 0$ for every  $j \not \in \{i_1, \ldots, i_k\}$;
\item if $(x,v) \in \mathcal{P}_3$ then, indicating with $i$ the only index such that $\vnorm{v^{\perp}_i} > \vnorm{v^{\perp}_j}$ for every  $j \not = i$, we have $\eps_i = M$ and $\eps_j = 0$ for every  $j \not = i$;
\item if $(x,v) \in \mathcal{P}_4$ then, indicating with $i_1, \ldots, i_k$ the indexes such that $\vnorm{v^{\perp}_{i_1}} = \ldots = \vnorm{v^{\perp}_{i_k}}$ and $\vnorm{v^{\perp}_{i_1}} > \vnorm{v^{\perp}_j}$ for every  $j \not \in \{i_1, \ldots, i_k\}$, we have $\eps_j = 0$ for every  $j \not \in \{i_1, \ldots, i_k\}$ and $\sum^k_{\ell = 1} \eps_{i_{\ell}} = M$.
\end{itemize}

Notice that any control $u(x,v) \in U(x,v)$ acts as an additional force which pulls agents towards having the same mean consensus parameter. The imposition of the $\ell^N_1-\ell^d_2$-norm constraint has the function of enforcing \textit{sparsity}: from the observation above clearly follows that
\begin{align*}
U|_{\mathcal{P}_1} = \{0\} \quad \text{and} \quad U|_{\mathcal{P}_3} = \{(0,\ldots,0,-M v^{\perp}_i / \|v^{\perp}_i\|,0, \ldots,0)^T\},
\end{align*}
for some unique $i \in \{1, \ldots, N\}$, i.e., the restrictions of $U$ to $\mathcal{P}_1$ and to $\mathcal{P}_3$ are single-valued. However, even if not all controls belonging to $U$ are sparse, there exist selections with maximal sparsity.

\begin{definition}[{\cite[Definition 4]{caponigro2015sparse}}]\label{cuckersparsecontrol}
We select the \textit{sparse feedback control} $u(x,v) \in U(x,v)$ according to the following criterion:
\begin{itemize}
\item if $\max_{1 \leq i \leq N}\vnorm{v^{\perp}_i} \leq \gamma(B(x,x))^2$, then $u(x,v) = 0$;
\item if $\max_{1 \leq i \leq N}\vnorm{v^{\perp}_i} > \gamma(B(x,x))^2$, denote with $\maxindex(x,v) \in \{1,\ldots,N\}$ the smallest index such that
\begin{align*}
\vnorm{v^{\perp}_{\maxindex(x,v)}} = \max_{1 \leq i \leq N}\vnorm{v^{\perp}_i}.
\end{align*}
Then
\begin{align*}
u_j(x,v) \define\begin{cases}
\displaystyle -M \frac{v^{\perp}_{\maxindex(x,v)}}{\|v^{\perp}_{\maxindex(x,v)}\|} & \quad \text{ if } j = \maxindex(x,v), \\
 0 & \quad \text{ otherwise.}
\end{cases}
\end{align*}
\end{itemize}
\end{definition}

The geometrical interpretation of why the sparse feedback control is a solution of \eqref{eq:variationalprinciple} is given by the graphics in Figure \ref{fig:sprc} below, representing the scalar situation.

\begin{figure}[!h]
	\centering
	\includegraphics[width=0.47\linewidth]{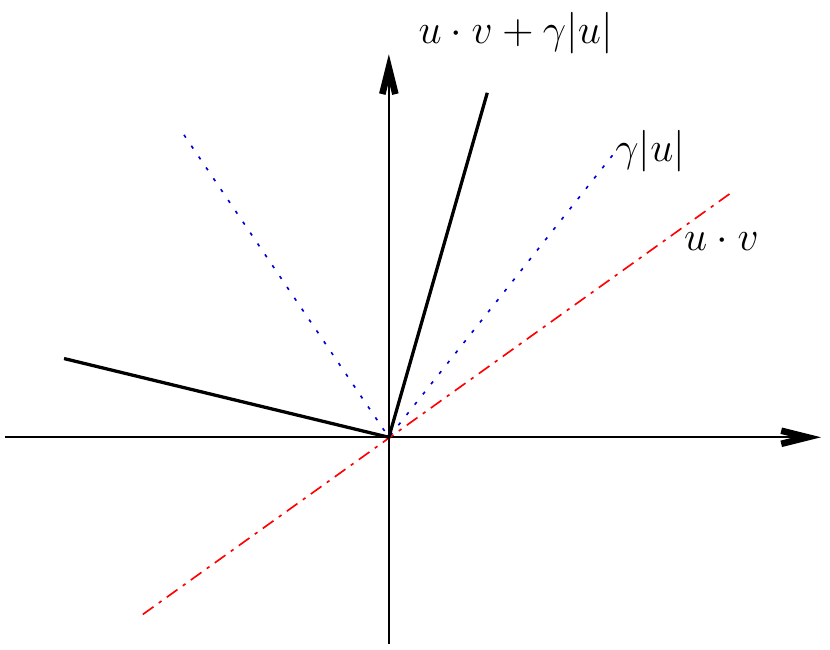}
	\includegraphics[width=0.50\linewidth]{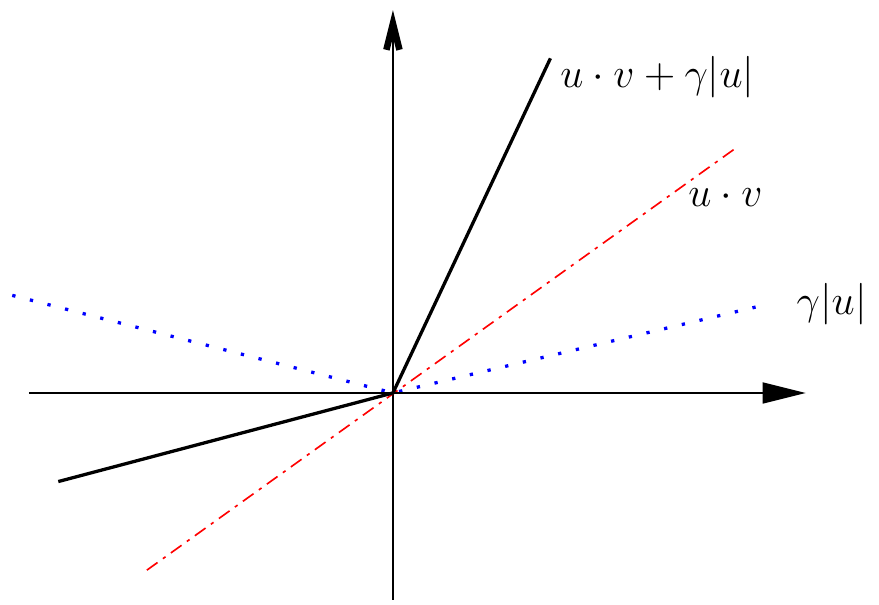}
	\caption{Geometrical interpretation of the solution of \eqref{eq:variationalprinciple} in the scalar case. On the left: for $|v| \leq \gamma$ the minimal solution $u \in [-M,M]$ is zero. On the right: for $ |v| > \gamma $ the minimal solution $u \in [-M,M]$ is for $|u|=M$.}
	\label{fig:sprc}
\end{figure}

The following result shows that the above feedback control strategy is capable of steering the system to the consensus region in finite time.

\begin{theorem}[{\cite[Theorem 3]{caponigro2015sparse}}]\label{greedycontr}
For every initial condition $(x_{0},v_{0}) \in \R^{dN} \times \R^{dN}$ and $M>0$, there exist $T >0$ and a \textit{piecewise constant} selection of the sparse feedback selection of Definition \ref{cuckersparsecontrol} %control $u:[0,T] \to \RN$, with $\sum_{i=1}^{N}\|u_{i}(t)\| \leq M$ for every $t\in [0,T]$
such that the associated absolutely continuous solution reaches the consensus region at the time $T$.
%More precisely, we can choose adaptively the control law explicitly as {one of the solutions} of the variational problem 
%\begin{equation}\label{varprinc}
%\min B (v, u) + \frac{\gamma(x)}{N} \sum_{i=1}^{N}\|u_{i}\|  \quad \mbox{ subject to } \sum_{i=1}^{N}\|u_{i}\| \leq M\,,
%\end{equation}
%where $\gamma(x) = \sqrt{N} \int_{\sqrt{NB(x,x)}}^{\infty} a (\sqrt{2}r) dr, \quad \mbox{the threshold in Theorem \ref{hahakimthm}}$.
%{This choice of the control makes $V(t) = B(v(t),v(t))$
%vanishing in finite time, in particular there exists $T$ such that {$B(v(t),v(t)) \leq \gamma(x)^{2}$}, $t \geq T$.}
\end{theorem}

This result is truly remarkable, since it holds again independently of the initial conditions and of the strength $M > 0$ of the control. %Moreover, it indicates the controllability of consensus systems simply by acting on one agent at each time and for nontrivial intervals of time.
%Remarkably, the above feedback control instantaneously maximizes the decay rate of the Lyapunov functional $V(t)$ in the set $U(x(t),v(t))$.
Furthermore, the sparse feedback control is optimal for consensus problems with respect to any other control strategy in $U(x(t),v(t))$ which spreads control over multiple agents, as the following result shows.

\begin{proposition}[{\cite[Proposition 3]{caponigro2015sparse}}]\label{prop:sparseoptimal}
	The sparse feedback control of Definition \ref{cuckersparsecontrol} is for every $t \geq 0$ an instantaneous minimizer of
	\begin{align*}
		\mathcal{D}(t, u) \define \frac{d}{dt}V(t)
	\end{align*}
	over all possible feedback controls in $U(x(t),v(t))$. %Moreover, the functionals $X$ and $V$ of the solution associated to it satisfies for every $t \geq 0$
	%\begin{align}\label{eq:strongdecay}
	%	\sqrt{V(t)} \leq \sqrt{V(0)} - \frac{M}{N}t \quad \text{ and } \quad
%X(t) \leq 2X(0) + \frac{N^2}{2M^2}V(0)^2.
%	\end{align}
\end{proposition}

%Unfortunately, the existence of solutions %the issue of the well-posedness
%of system \eqref{eq:cuckercontr} with the control $u$ as in Definition \ref{cuckersparsecontrol} is still an open problem. We shall discuss the reasons in the next section.

A direct consequence of Proposition \ref{prop:sparseoptimal} is that, for Cucker-Smale systems, a feedback stabilization is most effective if all the attention of the controller is focused on the agent farthest away from consensus. %very few (actually in this case only one) agents at each time.
This also means that, despite the fact that the external policy maker may have few resources at disposal and can allocate them at each time only on very few key players in the system, it is always possible to effectively stabilize the dynamics to return to energy levels where the system tends autonomously to consensus. This result is perhaps surprising if confronted with the more intuitive strategy of controlling more, or even all, agents at the same time. This let us answer to the question (Q) raised at the beginning of this section as follows:
\medskip
\begin{center}
(A)\textit{ under the constraints $(i)-(iii)$, sparse is better.}
\end{center}

\subsection{Numerical implementation of the sparse control strategy}

We now compare the performances of the sparse feedback control with the self-organizing power of an uncontrolled Cucker-Smale system and the efficacy of the total control strategy \eqref{eq:totalcontrol}. In Figure \ref{fig:sparsecontr1}--left it is shown a simulation of a Cucker-Smale system with $\beta = 1$ without control (in black), with the total control (in blue), and with the sparse feedback control (in red). While the uncontrolled scenario seems far from converging towards a consensus state, both the total control and the sparse control strategies successfully align the agents in very short time. The greater effectiveness of the sparse feedback control can be witnessed in Figure \ref{fig:sparsecontr1}--right, where it is shown the decay of the Lyapunov functional $V$ in the three different cases: the sparse control is more efficient in bringing $V$ to $0$, as Proposition \ref{prop:sparseoptimal} predicts.

\begin{figure}[!h]
	\centering
	\includegraphics[width=0.49\linewidth]{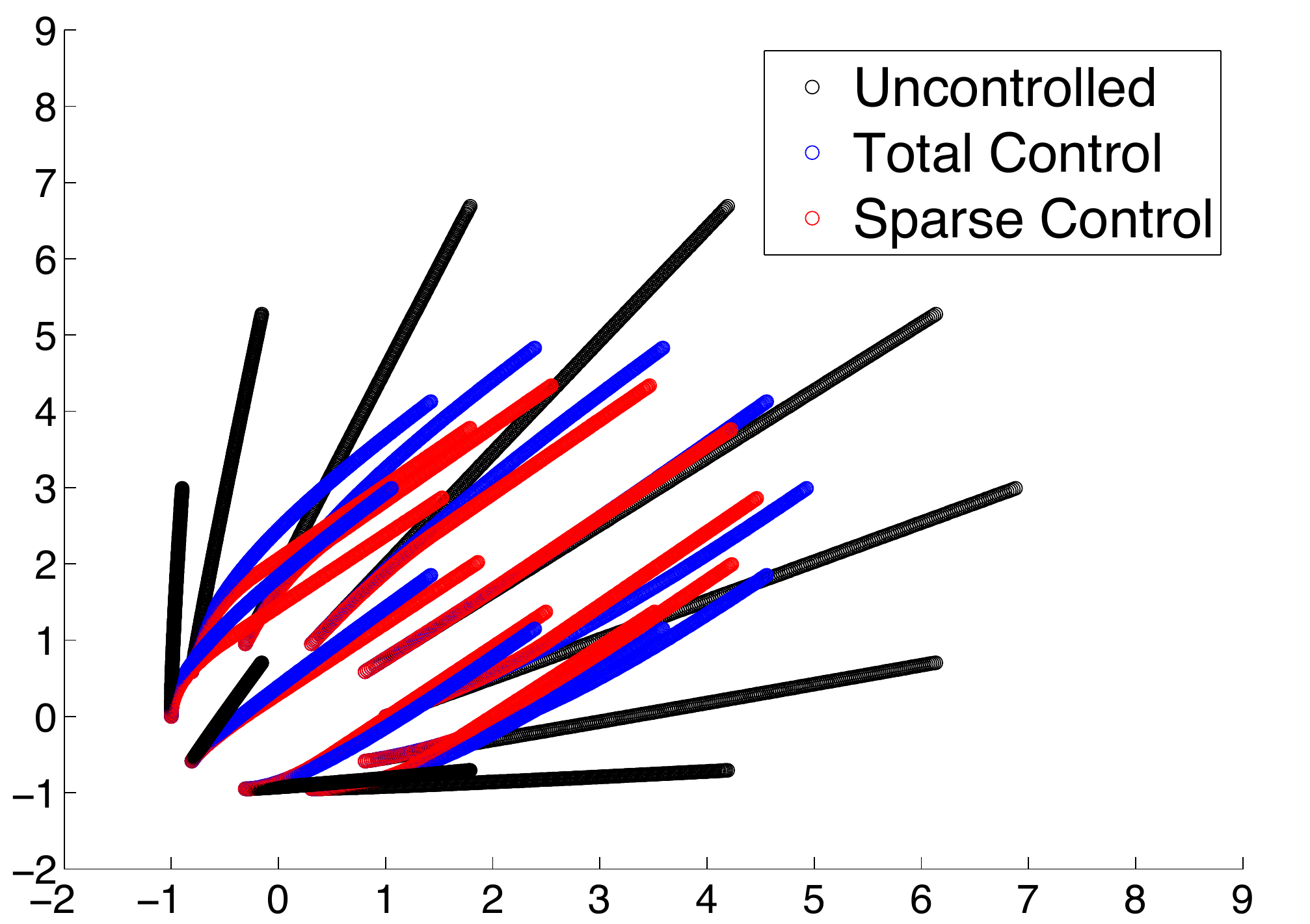}
	\includegraphics[width=0.49\linewidth]{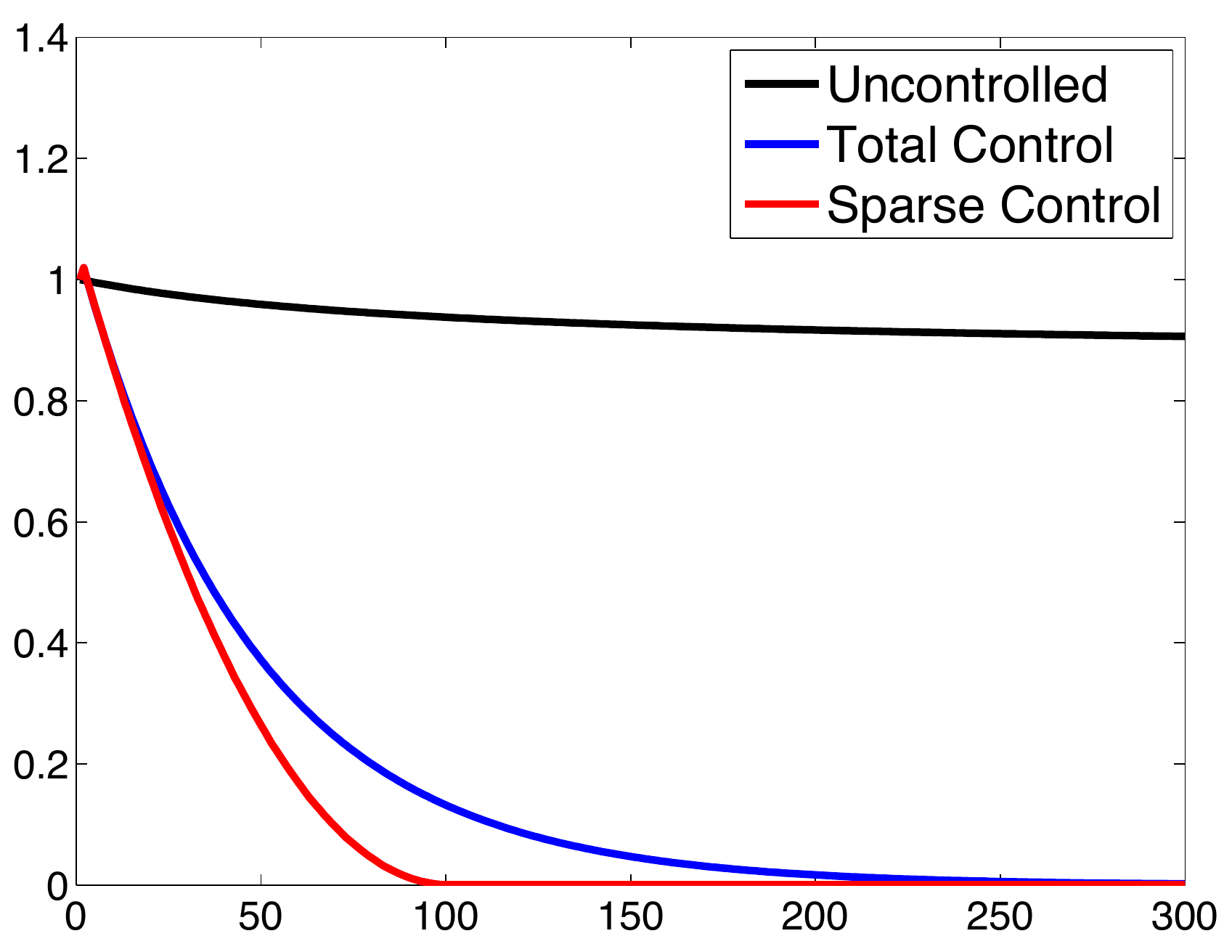}
	\caption{Comparison between sparse, total and no control. On the left: space evolution without control (in black), with the total control (in blue), and with the sparse control strategy (in red). On the right: the respective behavior of the functional $V$.}
	\label{fig:sparsecontr1}
\end{figure}

The situation where the sparse control strategy works at its bests is when the velocities of the agents are almost homogeneous, except for few outliers which are very distant from the mean velocity. As extensively discussed in \cite{bonginijunge2014sparse}, in such situations the total control is suboptimal because it also acts on agents which do not need any intervention, while the sparse control strategy is locally optimal because it focuses all its strength on the small group of outliers. Such scenario is portrayed in Figure \ref{fig:sparsecontr2}: starting from the same initial datum of Figure \ref{fig:sparsecontr1}, we modify the velocity of one agent so that it decisively deviates from the mean velocity. This time, the difference in the outcome of the two control strategies is much more visible.
More generally, an empirical detector of configurations where it is convenient to use the sparse feedback control is the so-called \textit{asymmetry measure}, proposed in \cite[Section 3.6.5]{bonginithesis}.

\begin{figure}[!h]
	\centering
	\includegraphics[width=0.49\linewidth]{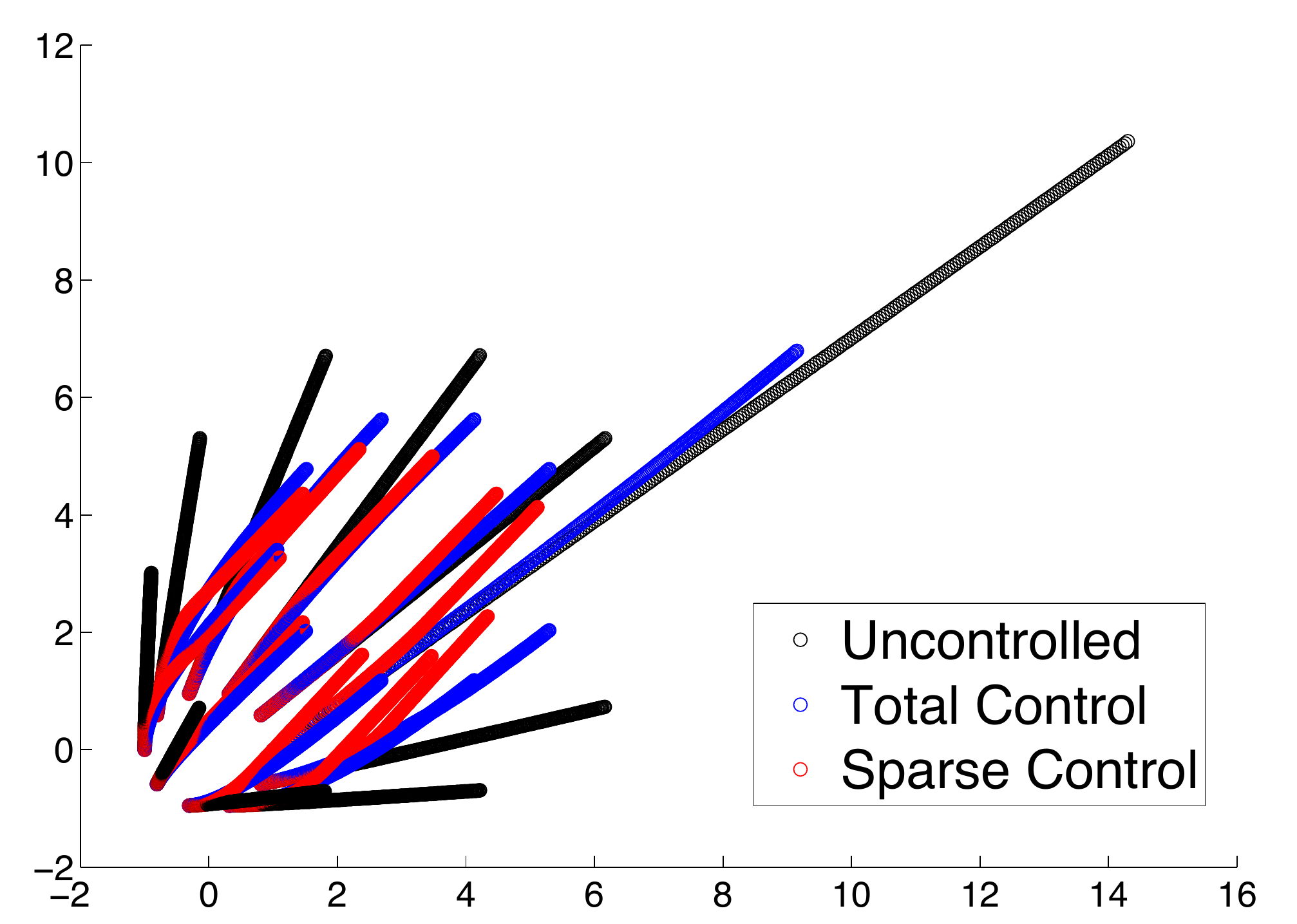}
	\includegraphics[width=0.49\linewidth]{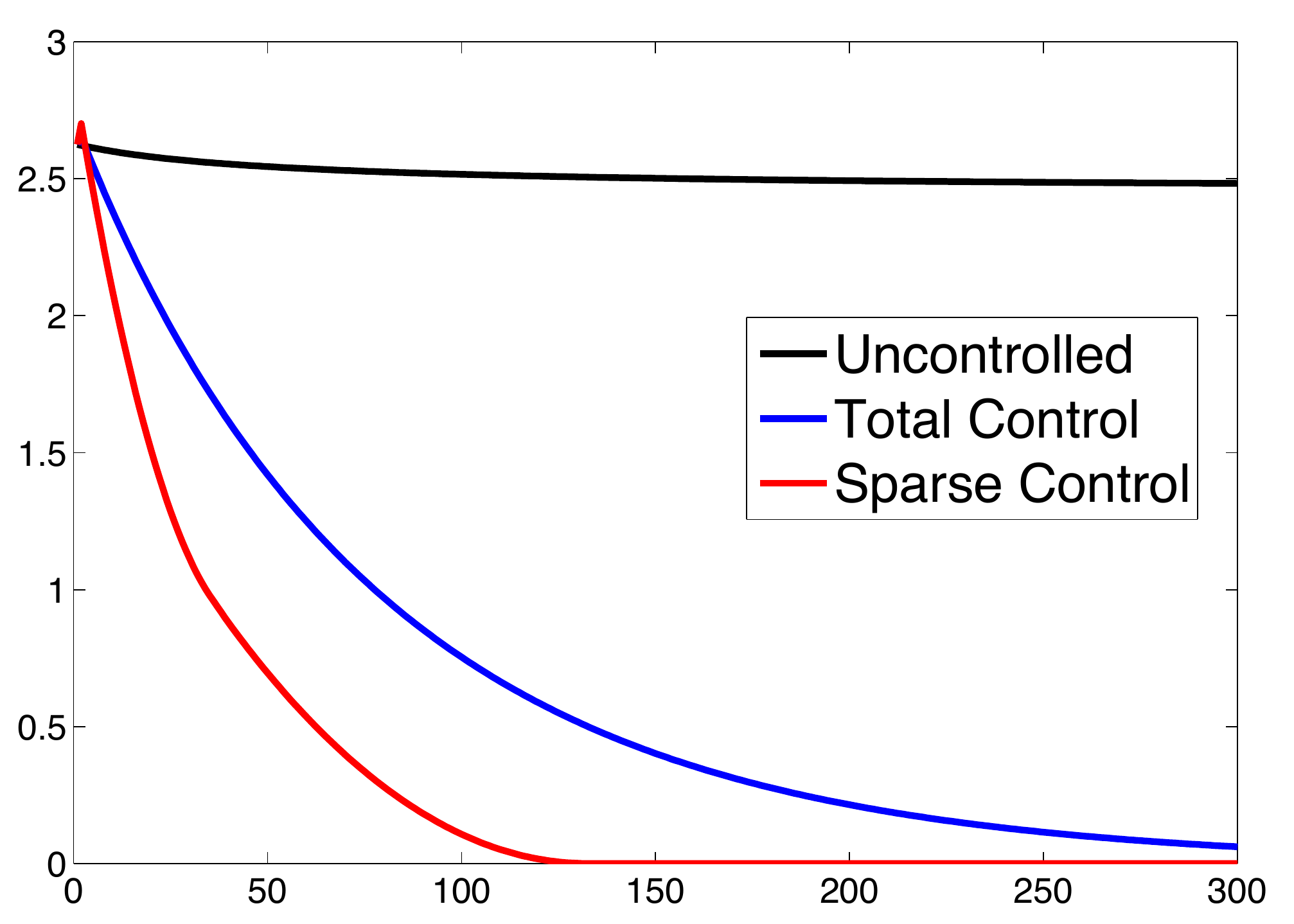}
	\caption{Configuration with one outlier. On the left: space evolution without control (in black), with the total control (in blue), and with the sparse control strategy (in red). On the right: the respective behavior of the functional $V$.}
	\label{fig:sparsecontr2}
\end{figure}

\section{The Cucker-Dong model}\label{sec:cuckerdongintro}

We now show how the sparse feedback control strategy previously introduced has far more reaching potential, as it can address also situations which do not match the structure \eqref{systemmatrix}, like the Cucker and Dong model of cohesion and avoidance introduced in \cite{cucker14}, which is given by the following system of differential equations
\begin{align}
\left\{\begin{aligned}
\begin{split} \label{eq:cuckerdong}
\dot{x}_{i}(t) & = v_{i}(t), \\
\dot{v}_{i}(t) & = -b_{i}(t)v_{i}(t) + \sum_{j = 1}^N a\left(\vnorm{x_{i}(t)-x_{j}(t)}^{2}\right)\left(x_{j}(t)-x_{i}(t)\right) +\\
&\quad+ \sum_{j \not = i}^N f\left(\vnorm{x_{i}(t)-x_{j}(t)}^{2}\right)\left(x_{i}(t)-x_{j}(t)\right),
\end{split}\quad i=1,\ldots,N,
\end{aligned}
\right.
\end{align}
%This model describes the dynamics of $N$ agents with main states $(x_{1}, \ldots, x_{N})$ and consensus parameters $(v_{1},\ldots,v_{N})$, which can be considered space and velocity variables in the context of group motion dynamics.
The evolution is governed by an \textit{attraction} force, modeled by a function $a:\mathbb{R}_+ \funarrow \R_+$, which is, for some fixed constant $H > 0$ and $\beta \geq 0$, of the form
\begin{align*}
a(r) = \frac{H}{(1 + r)^{\beta}},
\end{align*}
(notice that here we have $r$ in place of $r^2$, since we write $a(\vnorm{x_{i}-x_{j}}^{2})$ in place of $a\left(\vnorm{x_{i}-x_{j}}\right)$, hence $a$ has the same form as \eqref{eq:cuckerkernel}), though in general any Lipschitz-continuous, nonincreasing function with maximum in $a(0)$ suffices. This force is counteracted by a \textit{repulsion} given by a locally Lipschitz continuous or $\mathcal{C}^1$, nonincreasing function $f:(0,+\infty)\funarrow\R_+$. We request that
\begin{align*} %\label{eq:f_finite}
& \int^{+\infty}_\delta f(r) \ dr < +\infty, \quad \text{ for every } \delta > 0.
\end{align*}
%\begin{align} \label{eq:f_unbounded}
%& \int^{+\infty}_{0} f(r) \ dr = +\infty.
%\end{align}
A typical example of such a function is $f(r) = r^{-p}$ for every $p > 1$. %The reason of these requests will be clear after Remark \ref{rem:fRequest}.
The uniformly continuous, bounded functions $b_{i}:\R_+\funarrow[0,\Lambda]$, $i=1,\dots, N$, for a given $\Lambda \geq 0$, are interpreted as a friction which helps the system to stay confined.

It is easily seen how the above model can be rewritten as
\begin{align*}
\left\{
\begin{aligned}
\dot{x}(t) & = v(t), \\
\dot{v}(t) & = -L(x(t)) x(t) - v(t) b(t),
\end{aligned}
\right.
\end{align*}
where for any $x \in \R^{dN}$ the function $L(x) \define L^{a}(x) - L^f(x)$ is the difference between the Laplacians of the two matrices $(a(\vnorm{x_{i}-x_{j}}^{2}))_{i,j=1}^N$ and $(f(\vnorm{x_{i}-x_{j}}^{2}))_{i,j=1}^N$, respectively,
and we have set $vb\define (v_i b_i)_{i=1}^N$ for any $b = (b_1, \ldots, b_N)$. Notice that, differently from \eqref{systemmatrix}, now the Laplacians are acting on the variable $x$ and not anymore on $v$, mixing the dynamics of the two components of the state: as a consequence, the Cucker-Dong model is a non-dissipative system with singular repulsive interaction forces. Similar models considering attraction, repulsion  and other effects, such as alignment or self-drive, appear in the recent literature
and they seem effectively describing realistic situations of conditional pattern formation, see, e.g., some of the most related contributions \cite{MR2507454,ChuDorMarBerCha07,CuckerDong11,d2006self}.

\begin{figure}[!htb]
\centering
\includegraphics[scale=.5]{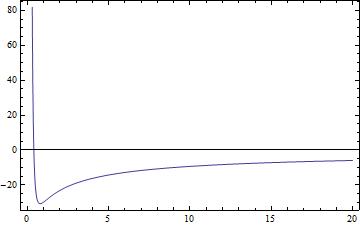}
\caption{Sum of the attraction and repulsion forces $h(r) = f(r) - a(r)$ as a function of the distance $r>0$. The parameters here are $H=50$, $\beta=0.7$, and $p=4$.}
\label{fig:repatt}
\end{figure}

At  first glance it may seem perhaps a bit cumbersome to consider a rather arbitrary splitting of the force into two terms governed by the functions $a$ and $f$ instead
of considering more naturally a unique function $h(r) \define f(r) - a(r)$ of the distance $r>0$, as depicted in Figure \ref{fig:repatt}. However, as we shall clarify in short,
the interplay of the polynomial decay of the function $h$ to infinity and its singularity at $0$ is fundamental in order to be able to characterize the
confinement and collision avoidance of the dynamics, and such a splitting, emphasizing the individual role of these two properties, will turn out to
be useful in our statements. As a matter of fact, several forces in nature do have similar behavior, for instance the van der Waals forces are
governed by Lennard-Jones potentials for which $h(r)= \sigma_f/r^{13} - \sigma_a/r^{7}$, for suitable positive constants $\sigma_f$ and $\sigma_a$.
%For such forces the polynomial decay corresponds to the parameter $\beta=7$, hence a rather fast decay %, which is actually supercritical ($\beta>1$) and it
%that makes the confinement of the system conditional to the initial total energy level (see Theorem \ref{th:cuckerdong} below).

\subsection{Pattern formation for the Cucker-Dong model}

To quantify the behavior of the system we introduce a quantity called the \emph{total energy} which includes the kinetic and potential energies; for all $(x, v) \in \R^{dN}\times\R^{dN}$ we define
\begin{align} \label{eq:energyfunction}
E(x, v) \define \sum_{i=1}^N \vnorm{v_{i}}^{2} + \frac{1}{2}\sum_{i < j}^N \int_{0}^{\vnorm{x_{i}-x_{j}}^{2}} a(r)dr + \frac{1}{2}\sum_{i < j}^N \int_{\vnorm{x_{i}-x_{j}}^{2}}^{+\infty}f(r)dr.
\end{align}
If $(x(t), v(t))$ is a point of a trajectory of system \eqref{eq:cuckerdong}, we set $E(t)\define E(x(t), v(t))$.

%\begin{remark}
%	Every term appearing in the definition of $E(x,v)$ is nonnegative (and well-defined by \eqref{eq:f_finite}), and hence we are allowed to bound from above every term appearing in the expression of $E(x,v)$ (as, for instance, $\vnorm{v_{i}}^{2}$ or $\frac{1}{2} \int_{\vnorm{x_{i}-x_{j}}^{2}}^{+\infty}f(r)dr$) by $E(x,v)$ itself.
%\end{remark}

The total energy %$E(t)= E(x(t), v(t))$
is a Lyapunov functional for system \eqref{eq:cuckerdong} and, provided we are in presence of no friction at all (i.e., $\Lambda = 0$), it is a conserved quantity. %Since the proof of this result, already derived in \cite{cucker14}, will be helpful later on, we will report it in full details.

\begin{proposition}[{\cite[Equation (3.1)]{cucker14}}] \label{pr:der_en_unc}
For every $t \geq 0$, we have
\begin{align*}
\frac{d}{dt}E(t) & = -2 \sum^N_{i = 1}{b_i(t)\|v_i(t)\|^2}.
\end{align*}
Hence, if $\Lambda = 0$ then $\frac{d}{dt}E \equiv 0$.
\end{proposition}
%\begin{proof}
%Let us compute
%\begin{align}
%\frac{d}{dt}E(t) & = \frac{d}{dt}\sum^N_{i = 1}{\|v_i(t)\|^2}  + \sum^N_{i,j = 1}{a(\|x_i(t) - x_j(t)\|^2)(x_i(t)- x_j(t))\cdot(v_i(t) - v_j(t))}  -\nonumber\\
%&\quad - \sum^N_{i,j = 1}{f(\|x_i(t) - x_j(t)\|^2)(x_i(t) - x_j(t))\cdot(v_i(t) - v_j(t))}.\label{eq:derivativeenergyequal}
%\end{align}
%The first term of the sum above is
%	\begin{align}
%\frac{d}{dt}\sum^N_{i = 1}{\|v_i(t)\|^2} & = 2 \sum^N_{i = 1}{\dot{v}_i(t) \cdot v_i(t) }\nonumber\\
%& = -2 \sum^N_{i = 1}{b_i(t)\|v_i(t)\|^2} - \nonumber\\
%&\quad - \sum^N_{i = 1}\sum^N_{j = 1}{a(\|x_i(t) - x_j(t)\|^2) (x_i(t) - x_j(t))\cdot( v_i(t)- v_j(t))} +\nonumber\\
%&\quad + \sum^N_{i = 1}\sum^N_{j = 1}{f(\|x_i(t) - x_j(t)\|^2)(x_i(t)- x_j(t))\cdot( v_i(t)- v_j(t))}, \label{eq:velocitysquared}
%\end{align}
%which, plugged into \eqref{eq:derivativeenergyequal}, yields
%\begin{align*}
%\frac{d}{dt}E(t) & = -2 \sum^N_{i = 1}{b_i(t)\|v_i(t)\|^2}.
%\end{align*}
%So, if $\Lambda = 0$ then $b_i \equiv 0$ for every $i = 1, \ldots N$, and thus $\frac{d}{dt}E \equiv 0$, as stated.
%\end{proof}

If the attraction force at far distance is very strong (for $\beta \leq 1$), despite an initial high level of kinetic energy and of repulsion potential energy, perhaps due to a space compression of the group of particles, the dynamics
is guaranteed to keep confined and collision avoiding in space at all times. If the attraction force is instead weak at far distance, i.e., $\beta>1$, then confinement and collision avoidance turn out to be properties of
the dynamics only conditionally to {\it initial} low levels of kinetic energy and repulsion potential energy, meaning that the particles should not be initially too fast and too close to each other. This latter condition is formulated in terms of a total energy
critical  threshold
\begin{align*} %\label{eq:threshold}
	\vartheta\define\frac{N-1}{2}\int_{0}^{+\infty}a(r)dr.
\end{align*}
This fundamental dichotomy of the dynamics has been characterized in the following result.

\begin{theorem}[{\cite[Theorem 2.1]{cucker14}}] \label{th:cuckerdong}
%Consider a population of $N$ agents modeled by system \eqref{eq:cuckerdong} with initial datum
Consider an initial datum $(x^0,v^0)\in\R^{dN}\times\R^{dN}$ satisfying $\|x_{i}^0 - x_{j}^0\|^2 > 0$ for all $i \not = j$ and
\begin{align*} %\label{eq:energybounded}
E(0) \define E(x^0,v^0)< \frac{1}{2} \int^{+\infty}_{0} f(r)dr.
\end{align*}
Then there exists a unique solution $(x(\cdot), v(\cdot))$ of system \eqref{eq:cuckerdong} with initial condition $(x^0, v^0)$. Moreover, if one of the two following hypotheses holds:
\begin{enumerate}
	\item $\beta \leq 1$,
	\item $\beta > 1$ and  $E(0) < \vartheta$,
\end{enumerate}
	then the population is cohesive and collision-avoiding, i.e., there exist two constants $B_0,b_0 > 0$ such that, for all $t \geq 0$
	\begin{align} \label{eq:spacecontainment}
	b_0 \leq \vnorm{x_i(t) - x_j(t)} \leq B_0, \quad \text{ for all } 1\leq i \not = j \leq N.
	\end{align}
\end{theorem}

%\begin{remark} \label{rem:fRequest}
%	In the case $N = 1$ we are dealing with a single agent with initial position $x(0)$ and initial speed $v(0)$, so by definition $E(0) = \vnorm{v(0)}^2$. Hence condition \eqref{eq:energybounded} becomes
%	\begin{align*}
%	\vnorm{v(0)} < \frac{1}{2} \int^{+\infty}_{0} f(r)dr.
%	\end{align*}
%	Since a single particle satisfies trivially condition \eqref{eq:spacecontainment} of Theorem \ref{th:cuckerdong}, no matter how its initial velocity is, and it is governed by a system like \eqref{eq:cuckerdong} \emph{for every} $\beta \geq 0$, then in order Theorem \ref{th:cuckerdong} to be valid for single agents we have to require that $f$ satisfies condition \eqref{eq:f_unbounded}, that is
%\begin{align*}
%	\int^{+\infty}_0 f(r) \ dr = \infty.
%\end{align*}
%\end{remark}

Motivated by Theorem \ref{th:cuckerdong}, we will call \emph{consensus region} the set
	\begin{align*}
		C \define \set{ w \in \mathbb{R}}{ w \leq \vartheta }
	\end{align*}
	We will say that the system \eqref{eq:cuckerdong} is \emph{in the consensus region at time $t$} if $E(t) \in C$. It is an obvious corollary of Theorem \ref{th:cuckerdong} the fact that if system \eqref{eq:cuckerdong} is in the consensus region at time $T$, for some $T \geq 0$, then condition \eqref{eq:spacecontainment} is fulfilled for every $t \geq T$.
	
\begin{remark}
Theorem \ref{th:cuckerdong} is the Cucker-Dong counterpart of Theorem \ref{thm:hhk}. Indeed, for the choice of $a$ as in \eqref{eq:cuckerkernel}, Theorem \ref{thm:hhk} implies that
\begin{enumerate}[label=$(\roman*)$]
	\item if $\beta \leq 1/2$ then $a \not \in L^1(\R_+)$, therefore consensus is achieved regardless of the initial conditions;
	\item if $\beta > 1/2$ then  $a \in L^1(\R_+)$, and consensus is guaranteed only if \eqref{eq:HaHaKim} is satisfied.
\end{enumerate}
\end{remark}

\begin{remark}
Let us stress again the fact that the word \emph{consensus} must be intended here as a stable \emph{cohesion and collision-avoiding} dynamics, in the spirit of the conclusion of Theorem \ref{th:cuckerdong}. This is in contrast with the meaning of the word \emph{consensus} in Definition \ref{def:consensus}, which describes a situation where all the agents move according to the same velocity vector. We point out that this definition of \emph{consensus} does not imply this particular feature, but it is rather intended to make a parallel between Theorem \ref{th:cuckerdong} and Theorem \ref{thm:hhk}, as already done by the authors in \cite[Remark 1]{cucker14}.
\end{remark}

As for the model \eqref{eq:cuckersmale} we could construct non-consensus events if one violates the sufficient condition \eqref{eq:HaHaKim}, also for the model \eqref{eq:cuckerdong}  and in violation of the threshold $E(0) <\vartheta$, one can exhibit non-cohesion events.

\begin{example}[Non-cohesion events \cite{cucker14}] 
Consider $N=2$, $d=2$, $\beta > 1$, $f \equiv 0$, $b_i \equiv 0$, and $x(t) = x_1(t) - x_2(t)$, $v(t) = v_1(t) - v_2(t)$ relative position and velocity of two agents on the line. Then we may rewrite the system as
\begin{equation}\label{counterexcd}
\left\{
\begin{split}
\dot x &= v\\
\dot v &= -\frac{x}{(1+x^{2})^\beta}.
\end{split}
\right.
\end{equation}
For the sake of compactness, we introduce the quantity
$$
\Psi(x) \define \frac{1}{(\beta-1)(1+ x^2)^{\beta-1}} \quad \text{ for every } x \in \R^2.
$$
We now prove that, if we are given the initial conditions $x(0)=x_{0}>0$ and $v(0)=v_{0} > 0$ satisfying 
$v(0)^2 \geq \Psi(x(0)),$%\frac{1}{(\beta-1)(1+ x(0)^2)^{\beta-1}},$
then $x(t) \to +\infty$ for $t \to +\infty$. Indeed, by direct integration in \eqref{counterexcd} one obtains
$v(t)^2 = \Psi(x(t)) + v(0)^2 - \Psi(x(0)),$ %\frac{1}{(\beta-1)(1+ x(t)^2)^{\beta-1}}+ \underbrace{\left[ v(0)^2- \frac{1}{(\beta-1)(1+ x(0)^2)^{\beta-1}}\right]}_{\define\Psi_0 \geq 0},$
and it follows that $v(t)>0$ for all $t \geq 0$. This implies that $x(\cdot)$ is increasing: had this function an upper bound $x_*$, then we would have $\dot x(t) = v(t) \geq (\Psi(x_*) + v(0)^2 - \Psi(x(0)))^{1/2}$, which in turn implies %( \frac{1}{(\beta-1)(1+ x_*^2)^{\beta-1}}+ \Psi_0)^{1/2}$ and
$x(t) \to +\infty$ for $t\to +\infty$, a contradiction.%, and $x(\cdot)$ is unbounded.
\end{example}

%Again, one can pose the question of whether, given $E(0)>\vartheta$ a \textit{sparse} control can bring in finite time $T$ the energy under the threshold $E(T) < \vartheta$.
%We showed in \cite{MR3195343} that in the latter situation of lost self-organization, one can nevertheless steer the system \eqref{eq:procontr2} to return to stable energy levels by sparse feedback controls. 
%\\

\section{Sparse control of the Cucker-Dong model}

Notice the similarity of the present situation and that of Section \ref{sec:cuckersparse}: in both cases we have a system whose desired pattern can be enforced by decreasing a certain Lyapunov functional under the action of a sparse intervention. Given a positive constant $M$ modeling the limited resources given to the external policy maker to influence instantaneously the dynamics, it is very natural to define the set of \textit{admissible controls} precisely as in Definition \ref{def:admcontrols}: a control $u:\R_+ \funarrow \R^{dN}$ is admissibile if it is a measurable functions which satisfies the $\ell^N_1 - \ell^d_2$-norm constraint \eqref{eq:controlbounded} for every $t \geq 0$. 
%\begin{eqnarray} \label{eq:controlbounded}
%\sum^N_{i = 1}{\|u_i(t)\|} \leq M.
%\end{eqnarray}
Hence, the \textit{controlled} Cucker-Dong model  is given by
\begin{align}
\left\{\begin{aligned}
\begin{split} \label{eq:cuckerdongcontrol}
\dot{x}_{i}(t) & = v_{i}(t), \\
\dot{v}_{i}(t) & = -b_{i}(t)v_{i}(t)\! + \!\sum_{j = 1}^N a\left(\vnorm{x_{i}(t) \! -\! x_{j}(t)}^{2}\right)\left(x_{j}(t)\!-\!x_{i}(t)\right)+ \\
&\quad + \sum_{\stackrel{j = 1}{i \not = j}}^N f\left(\vnorm{x_{i}(t)\!-\!x_{j}(t)}^{2}\right)\left(x_{i}(t)\!-\!x_{j}(t)\right) \!+\! u_i(t),
\end{split}\quad i = 1,\ldots,N,
\end{aligned}\right.
\end{align}
where $u$ is admissible. %For this system we define again the total energy function $E$ and the threshold $\vartheta$ as in \eqref{eq:energyfunction} and \eqref{eq:threshold}, respectively.

The control should be exerted until $E(T) < \vartheta$ at some finite time $T$, and then it should be turned off, similarly to the sparse selection of Definition \ref{cuckersparsecontrol}. Since we start from $E(0) > \vartheta$, then it is necessary that our control forces the total energy to decrease, for instance by ensuring $\frac{d}{dt}E < 0$. The following technical result helps us to identify the form of admissible controls satisfying this property. %For the time being, we assume that a Filippov solution of \eqref{eq:cuckerdongcontrol} exists, later on we shall prove it.

\begin{lemma} \label{le:derivativeenergyleq}
Suppose there exists a solution of the system \eqref{eq:cuckerdongcontrol}. % and let $E$ be the total energy function associated to it.
Then%, for every $t \geq 0$ we have
\begin{align} \label{eq:derivativeenergyleq}
\frac{d}{dt}E(t) & = -2 \sum^N_{i = 1}{b_i(t)\|v_i(t)\|^2} + 2 \sum^N_{i = 1}{u_i(t)\cdot v_i(t)} \quad \text{ for every } t \geq 0.
\end{align}
\end{lemma}
%\begin{proof}
%The identity \eqref{eq:derivativeenergyequal} in the proof of Proposition \ref{pr:der_en_unc} is still valid for a solution to \eqref{eq:cuckerdongcontrol}, while \eqref{eq:velocitysquared} changes as follows
%	\begin{align*}
%\frac{d}{dt}\sum^N_{i = 1}{\|v_i(t)\|^2} & = -2 \sum^N_{i = 1}{b_i(t)\|v_i(t)\|^2} -\\
%&\quad - \sum^N_{i = 1}\sum^N_{j = 1}{a(\|x_i(t) - x_j(t)\|^2)(x_i(t) - x_j(t))\cdot(v_i(t) - v_j(t))} + \\
%&\quad + \sum^N_{i = 1}\sum^N_{j = 1}{f(\|x_i(t) - x_j(t)\|^2)(x_i(t) - x_j(t))\cdot(v_i(t) - v_j(t))} +\\
%&\quad+2 \sum^N_{i = 1}{u_i(t)\cdot v_i(t)},
%\end{align*}
%because of the control term. Inserting this expression into \eqref{eq:derivativeenergyequal} we get \eqref{eq:derivativeenergyleq}.
%\end{proof}

\subsection{Extending the sparse control strategy}

From expression \eqref{eq:derivativeenergyleq}, it is clear that the best way our control can act on $E$ in order to push it below the threshold is not acting on the mutual distances between agents, but according to the velocities $v$. Hence, we focus on the following family of controls, closely resembling \eqref{eq:formcontrolcucker} .

\begin{definition} \label{def:dongcontrol}
Let $(x,v) \in \R^{dN}\times\R^{dN}$ and $0 \leq \eps \leq M/E(0)$. We define the {\it sparse feedback control} $u(x,v) = (u_1(x,v), \ldots, u_N(x,v))^T \in \R^{dN}$ associated to $(x,v)$ as
\begin{align*}
u_i(x,v) \define
\begin{cases}
	\displaystyle - \eps E(x,v) \frac{v_{\maxindex(x,v)}}{\|v_{\maxindex(x,v)}\|} & \quad \mbox{  if } i = \maxindex(x,v), \\
	0 & \quad \mbox{ otherwise.}
	\end{cases}
\end{align*}
where $\maxindex(x,v)$ is the minimum index such that
\begin{align*}
\|v_{\maxindex(x,v)}\| =\max_{1\leq j \leq N} \|v_j\|.
\end{align*}
\end{definition}

Whenever the point $(x,v)$ is a point of a curve $(x,v): \R_+ \rightarrow \R^{dN}\times\R^{dN}$, i.e. $(x,v) = (x(t), v(t))$ for some $t \geq 0$, we will replace everywhere $u(x,v)=u(x(t),v(t))$ and $\maxindex(x,v) = \maxindex(x(t),v(t))$ with $u(t)$ and $\maxindex(t)$, respectively.

\begin{remark}
Definition \ref{def:dongcontrol} makes sense if $\|v_{\maxindex(t)}(t)\| \not = 0$ for at least almost every $t \geq 0$.  Notice that, if the latter condition were not holding, then $v_{i}(t) = 0$ for all $i = 1, \ldots, N$ and for all $t \geq 0$, hence $\dot{v}_{i}(t) = 0$ for all $i = 1, \ldots, N$ and for all $t \geq 0$, hence the configuration of the system would be in a steady state and no control would be needed.
\end{remark}

The parameter $\eps$ will help us to tune the control in order to ensure the convergence to the consensus region. %Moreover, whenever it is clear from the context, we will omit the time (or the space) dependency of $\maxindex$.
Indeed, notice that if we were able to prove that $\vnorm{\overline{v}(t)} \geq \eta$ holds for every $t \geq 0$ for some $\eta > 0$, then it would follow that
\begin{align*}
	\frac{d}{dt}E(t) & \leq 2 \left(-\eps E(t) \frac{v_{\maxindex(t)}(t)}{\vnorm{v_{\maxindex(t)}(t)}}\right)\cdot v_{\maxindex(t)}(t) = - 2\eps E(t) \vnorm{v_{\maxindex(t)}(t)}  \leq - 2 \eps \eta E(t),
\end{align*}
from which we obtain the estimate $E(t) \leq E(0) e^{-2\eps \eta t}$ for every $t \geq 0$. Therefore it follows that $E$ is decreasing: this in turn implies that, whenever $\eps \leq M/E(0)$ holds, we have
\begin{align*}
\sum^N_{i = 1}\|u_i(t)\|=\eps E(t) \leq \frac{M}{E(0)}E(t) \leq M,
\end{align*}
whence the validity of the constraint \eqref{eq:controlbounded}. Therefore, the control of Definition \ref{def:dongcontrol} is admissible.

By exploiting several nontrivial a priori estimates for stability (collected in \cite[Section 3.2]{bofo13}), which were not necessary for system \eqref{systemmatrix} due to its dissipative nature, we obtain the following result, which resembles closely Theorem \ref{greedycontr}.

\begin{theorem}[{\cite[Theorem 4.1 and Proposition 4.2]{bofo13}}] \label{th:mainresult}
Fix $M > 0$. Let $(x^0, v^0) \in \R^{dN}\times\R^{dN}$ be such that the following hold:
\begin{description}
	%\item[\quad \quad $(a)$] $0 < \|x^0_{i} - x^0_{j}\| \leq C_0 \sqrt{E(0)} \mbox{ for all }1 \leq i \not = j \leq N$, for some $C_0 > 0$,
	\item[\quad \quad $(a)$] $\|\overline{v}(0)\| \geq \eta > 0$;
	\item[\quad \quad $(b)$] for $$c \define \exp \left(-\frac{2\sqrt{3}}{9}\frac{M\|\overline{v}(0)\|^3}{E(0) \sqrt{E(0)} \left(\Lambda \sqrt{E(0)} + \frac{M}{N}\right)} \right)$$ it holds $c\vartheta > E(0) > \vartheta$.%the threshold $\vartheta$ satisfies $$\vartheta > E(0) \exp \left(-\frac{2\sqrt{3}}{9}\frac{M\|\overline{v}(0)\|^3}{E(0) \sqrt{E(0)} \left(\Lambda \sqrt{E(0)} + \frac{M}{N}\right)} \right).$$
\end{description}
%\begin{align*}
%	q & \define \frac{2 M\|\overline{v}(0)\|^3}{3 E(0) \sqrt{E(0)} \left(\Lambda \sqrt{E(0)} + \frac{M}{N}\right) \ln\left(\frac{E(0)}{\vartheta}\right)}
%\end{align*}
Then there exist constants $T > 0$ and $\Gamma = \Gamma(x^0,v^0,\vartheta,\eta,c) > 0$, %$\tau_0 > 0$, $q > 0$, $\omega > 0$ and $\eps > 0$ such that for every $\tau \leq \tau_0$
%\begin{align} \label{eq:eps_small}
%	\frac{1}{2 T^{\star}(q)\left(\frac{1}{q}\vnorm{\overline{v}(0)} - \mu(\tau, q) \right)} \ln\left(\frac{E(0)}{\vartheta}\right) \leq \varepsilon \leq \frac{M}{E(0)}
%\end{align}
%holds, where for the sake of brevity we have set
%\begin{align} \label{eq:defintion_mu}
%\mu(\tau, q) \define M \Lambda \tau^2 \!+\! \left(M \!+\! \Lambda \sqrt{E(0)} \!+\! (a(0) \!+\! f(\omega^2))\left((T^{\star}(q) \!+\! \tau) D(\tau) \!+ \!C_0\right)N\sqrt{E(0)}\right)\tau.
%\end{align}
%Moreover, the sampling solution of system \eqref{eq:cuckerdongcontrol} associated with the control $u$ of Definition \ref{def:dongcontrol} with control parameter $\eps$ satisfying \eqref{eq:eps_small}, the sampling time $\tau \leq \tau_0$ and initial datum $(x^0, v^0)$ reaches the consensus region before time $T^{\star}(q)$.
and a \textit{piecewise constant} selection of the sparse feedback control of Definition \ref{def:dongcontrol} %control $u:[0,T] \to \RN$, with $\sum_{i=1}^{N}\|u_{i}(t)\| \leq M$ for every $t\in [0,T]$
such that
\begin{description}
	\item[\quad \quad $(1.)$] $\|\overline{v}(t)\| \geq \eta$ for every $t \leq T$;
	\item[\quad \quad $(2.)$] whenever $\Gamma \leq \eps \leq M/E(0)$ holds, the associated absolutely continuous solution reaches the consensus region before time $T$.
\end{description}
\end{theorem}

We remark that, while the stabilization of Cucker-Smale systems by means of sparse feedback controls is unconditional with respect to the initial conditions (see Theorem \ref{greedycontr}), for the Cucker-Dong model our analysis guarantees stabilization only within certain total energy levels, which is suggesting that also stabilization can be conditional. However, the numerical experiments reported in Section \ref{sec:dongnumerics} suggest that it is possible to exceed such an upper energy barrier in many cases, even if there are pathological situations for which there is no hope to steer the agents towards a cohesive configuration.

\subsection{Optimality of the sparse feedback control}

We now pass to show that the sparse feedback control of Definition \ref{def:dongcontrol} is a minimizer of a variational criterion similar to \eqref{eq:variationalprinciple}. To this end, notice that each value of $\eta \geq 0$ appearing in Theorem \ref{th:mainresult} yields a partition of $\R^{dN}\times\R^{dN}$ into four disjoint sets:
\begin{description}
	\item[$\mathcal{P}_1$] $\define \{ (x,v) \in \R^{dN}\times\R^{dN} : \max_{1 \leq i \leq N} \vnorm{v_i} < \eta \}$,
	\item[$\mathcal{P}_2$] $\define \{ (x,v) \in \R^{dN}\times\R^{dN} : \max_{1 \leq i \leq N} \vnorm{v_i} = \eta \text{ and } \exists k \geq 1 \text{ and } i_1, \ldots, i_k \break\in \{1, \ldots N\}$ $\text{ such that } \vnorm{v_{i_1}} = \ldots = \vnorm{v_{i_k}} \text{ and } \vnorm{v_{i_1}} > \vnorm{v_j} \text{ for every } j \not \in \{i_1, \ldots, i_k\}\}$,
	\item[$\mathcal{P}_3$] $\define \{ (x,v) \in \R^{dN}\times\R^{dN} : \max_{1 \leq i \leq N} \vnorm{v_i} > \eta \text{ and } \exists ! i \in \{1, \ldots N\} \text{ such} \\ \text{ that } \vnorm{v_i} > \vnorm{v_j} \text{ for every } j \not = i \}$,
	\item[$\mathcal{P}_4$] $\define \{ (x,v) \in \R^{dN}\times\R^{dN} : \max_{1 \leq i \leq N} \vnorm{v_i} > \eta \text{ and } \exists k > 1 \text{ and } i_1, \ldots, i_k\break \in \{1, \ldots N\}$ $\text{ such that } \vnorm{v_{i_1}} = \ldots = \vnorm{v_{i_k}} \text{ and } \vnorm{v_{i_1}} > \vnorm{v_j} \text{ for every } j \not \in \{i_1, \ldots, i_k\} \}$,
\end{description}

The above partition naturally leads to the following class of feedback controls.

\begin{definition} \label{def:convex_control}
For every $(x, v) \in \R^{dN}\times\R^{dN}$ we denote with $U(x, v) \subseteq \R^{dN}$ the set of all vectors $u(x,v)=(u_1(x,v), \dots, u_N(x,v))^T \in \R^{dN}$, whose vector entries are of the form
\begin{align*}
u_i(x,v) = \begin{cases}
	\displaystyle -\eps_i E(x,v) \frac{v_i}{\|v_i\|} &\quad \mbox{  if } \|v_i\| \not = 0, \\
	0 &\quad \mbox{  if } \|v_i\| = 0,
	\end{cases}
\end{align*}
where the coefficients $\eps_i \geq 0$ satisfy
\begin{align*}
\sum^N_{i = 1} \eps_i \leq \frac{M}{E(0)},
\end{align*}
and
\begin{itemize}
\item if $(x,v) \in \mathcal{P}_1$ then $\eps_i = 0$ for every $i = 1, \ldots, N$;
\item if $(x,v) \in \mathcal{P}_2$ then indicating with $i_1, \ldots, i_k$ the indexes such that $\vnorm{v_{i_1}} = \ldots = \vnorm{v_{i_k}} = \eta$ and $\vnorm{v_{i_1}} > \vnorm{v_j}$ for every  $j \not \in \{i_1, \ldots, i_k\}$, we have $\eps_j = 0$ for every  $j \not \in \{i_1, \ldots, i_k\}$;
\item if $(x,v) \in \mathcal{P}_3$ then, indicating with $i$ the only index such that $\vnorm{v_i} > \vnorm{v_j}$ for every  $j \not = i$, we have $\eps_i = M/E(0)$ and $\eps_j = 0$ for every  $j \not = i$;
\item if $(x,v) \in \mathcal{P}_4$ then, indicating with $i_1, \ldots, i_k$ the indexes such that $\vnorm{v_{i_1}} = \ldots = \vnorm{v_{i_k}}$ and $\vnorm{v_{i_1}} > \vnorm{v_j}$ for every  $j \not \in \{i_1, \ldots, i_k\}$, we have $\eps_j = 0$ for every  $j \not \in \{i_1, \ldots, i_k\}$ and $\sum^k_{\ell = 1} \eps_{i_{\ell}} = M/E(0)$.
\end{itemize}
\end{definition}

\begin{remark} \label{rem:u_enlarged}
	Under the hypotheses of Theorem \ref{th:mainresult}, the control $u(t)$ introduced in Definition \ref{def:dongcontrol} belongs to $U(x(t),v(t))$ whenever $t \leq T$, since it is guaranteed that $\max_{i \leq i \leq N} \vnorm{v_i(t)} \geq \eta$ for every $t \leq T$.
\end{remark}

The set $U(x, v)$ is closed and convex, and, moreover, has the following very elegant alternative variational interpretation, reminiscent of Definition \ref{def:cuckervariational}.

\begin{proposition}{\cite[Propositions 5.2 and 5.4]{bofo13}} \label{prop:alternativeU}
For every $(x, v) \in \R^{dN}\times\R^{dN}$ and for every $M \geq 0$, set
\begin{align*}
m(x,v) \define M \frac{E(x,v)}{E(0)} \quad \text{ and } \quad K(x,v) \define \set{u \in \R^{dN}}{\sum^N_{i = 1} \vnorm{u_i} \leq m(x,v)}.
\end{align*}
Let $\mathcal{J}: \R^{dN} \funarrow \mathbb{R}$ be the functional defined by
\begin{align*}%\label{functomin}
\mathcal{J}(u,v) \define v \cdot u + \eta \sum^N_{i = 1} \vnorm{u_i}.
\end{align*}
Then
\begin{align*}
U(x,v) & = \argmin_{u \in K(x,v)} \mathcal{J}(u,v).
\end{align*}
%Moreover, the set-valued function $U:\R^{dN}\times\R^{dN} \funarrow \mathscr{P}(\R^{dN})$ which associates to every point $(x, v)$ the set $U(x,v)$ is upper hemicontinuous.
\end{proposition}

The next result is the Cucker-Dong counterpart of Proposition \ref{prop:sparseoptimal}: the sparse feedback control minimizes the decay rate of the functional $E$ among the controls introduced in Definition \ref{def:convex_control}.

\begin{theorem} \label{th:optimality}
	The feedback control of Definition \ref{def:dongcontrol} is an instantaneous minimizer of
	\begin{align*}
		\mathcal{D}(t, u) \define \frac{d}{dt}E(t)
	\end{align*}
	over all possible feedback controls $u \in U(x(t), v(t))$.
\end{theorem}

Similarly to what we have seen in the case of the Cucker-Smale system, the previous result shows that the most effective control strategy that the external policy maker can enact is to allocate all the resources at its disposal only on very few key agents in the system, in order to keep the dynamics bounded and collision avoiding. One of the most relevant differences with respect to Theorem \ref{greedycontr}, though, is that for the Cucker-Smale model the stabilization can be achieved unconditionally, i.e., independently
of the initial conditions $(x^0,v^0)$. For the Cucker-Dong model, instead, a similar sparse control strategy yields only a conditional results, i.e., we obtain stabilization conditionally to an initial energy level satisfying
$\vartheta < E(0) < c \vartheta,$
%for some constant $c>1$,
as stated in condition ($b$) of Theorem \ref{th:mainresult}. Our numerical experiments, which follow below, suggest that it is possible to exceed such an upper energy barrier, but it is unclear whether this is just a matter of fortunate
choices of good initial conditions or we can actually have a broader stabilization range than the one analytically derived above.

\subsection{Numerical validation of the sparse control strategy}\label{sec:dongnumerics}

In this section we will report the results of significant numerical simulations on Cucker-Dong systems in dimension $d = 2$ with and without the use of the sparse control strategy outlined in Definition \ref{def:dongcontrol}. Throughout the section, we will keep fixed the number of agents ($N = 8$), the friction applied ($\Lambda = 0$, i.e., frictionless) and the form of the repulsive function ($f(r) = r^{-p}$). We restrict only to $N = 8$ simply for an easier visualization of the results. This means that we will vary the shape of the function $a$ (i.e., we will act on $\beta$), the slope of the repulsion function (changing the value of $p$) and the maximum amount of strength of the sparse control (the parameter $M$). The parameter $\eps$ is always set equal to $M/E(0)$.

\subsubsection{The effect of sparse controls on the system}

Figure \ref{fig:uncontr1_posvel} displays the spatial evolution and speeds of the agents of a Cucker-Dong system with $\beta = 1.1$ and $p=2$:

\begin{figure}[!htb]
\centering
\includegraphics[width=\linewidth]{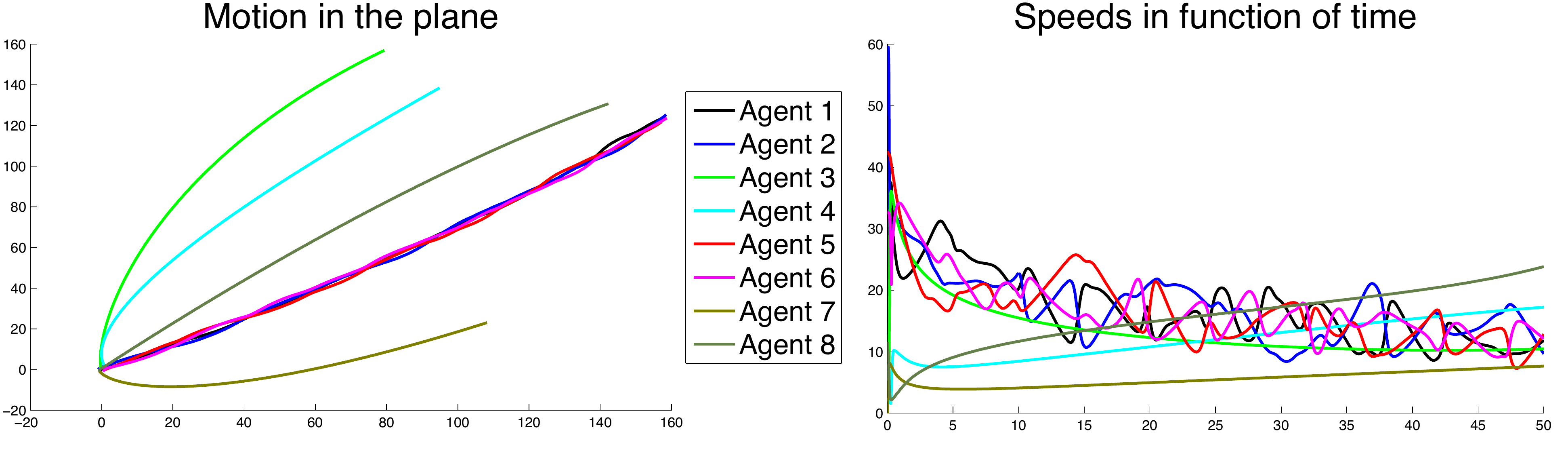}
\caption{Space evolution and speeds of the uncontrolled system.}
\label{fig:uncontr1_posvel}
\end{figure}

Though we can not infer the divergence of the system from this finite-time simulation, the portrayed situation seems far from going towards a flocking behavior. The only agents which seem to flock are Agent 1, Agent 2, Agent 5 and Agent 6 (resp. black, blue, red and magenta trajectories), as it is also visible by the corresponding speed graph, in which the speed of each agent is adjusted to the one of the other agents.

\begin{figure}[!htb]
\centering
\includegraphics[width=.67\linewidth]{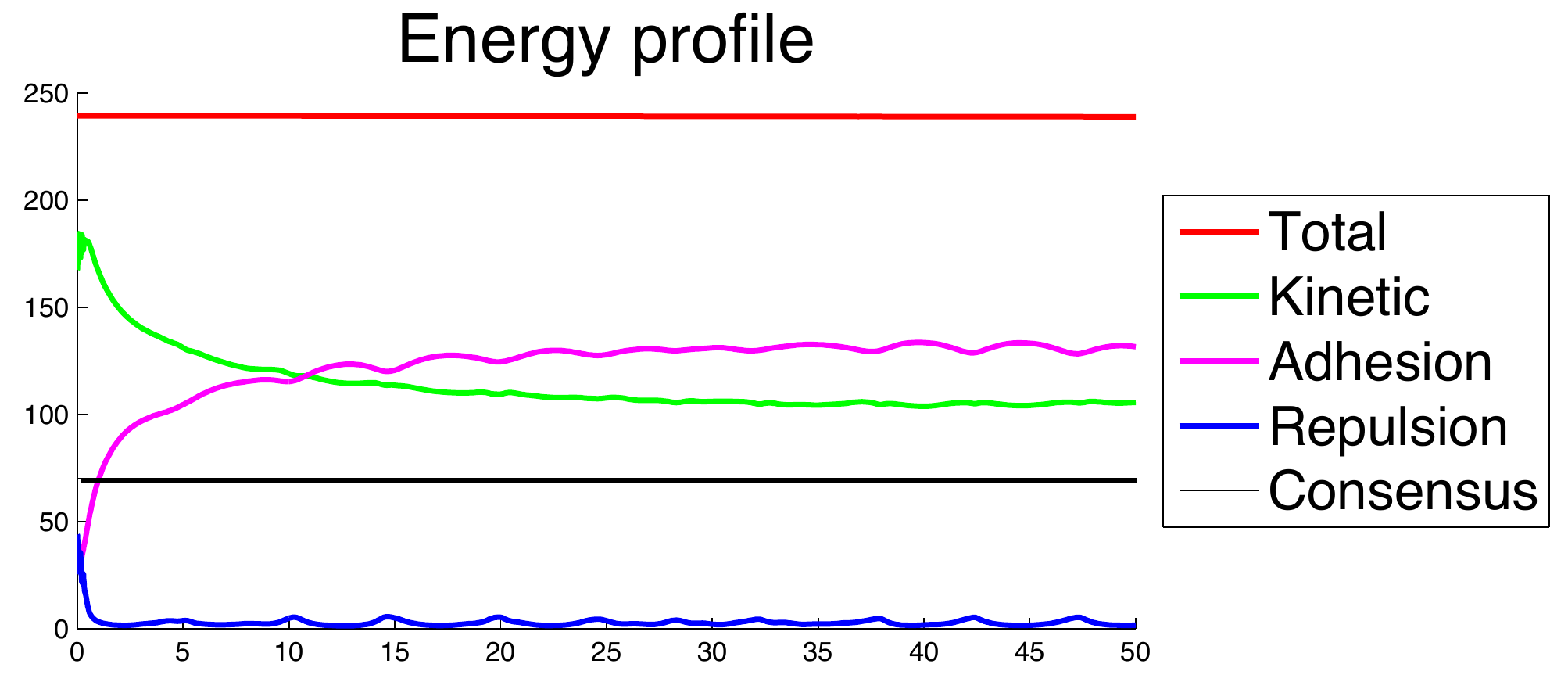}
\caption{Energy profile of the uncontrolled system.}
\label{fig:uncontr1_energy}
\end{figure}

Figure \ref{fig:uncontr1_energy} shows that the total energy $E$ (the red line) is constant and far away from the consensus threshold $\vartheta$ (black line). The increase in the distances between particles is reflected in an increase in the adhesion potential energy (the one due to $a$, see \eqref{eq:energyfunction}) and in a decrease in the repulsive one (due to $f$).

If instead we apply our sparse control strategy with $M = 35$ on the same system with the same initial conditions, the situation gets immediately far better from a consensus point of view, as Figure \ref{fig:contr1} witnesses.

\begin{figure}[!h]
\centering
\includegraphics[width=\linewidth]{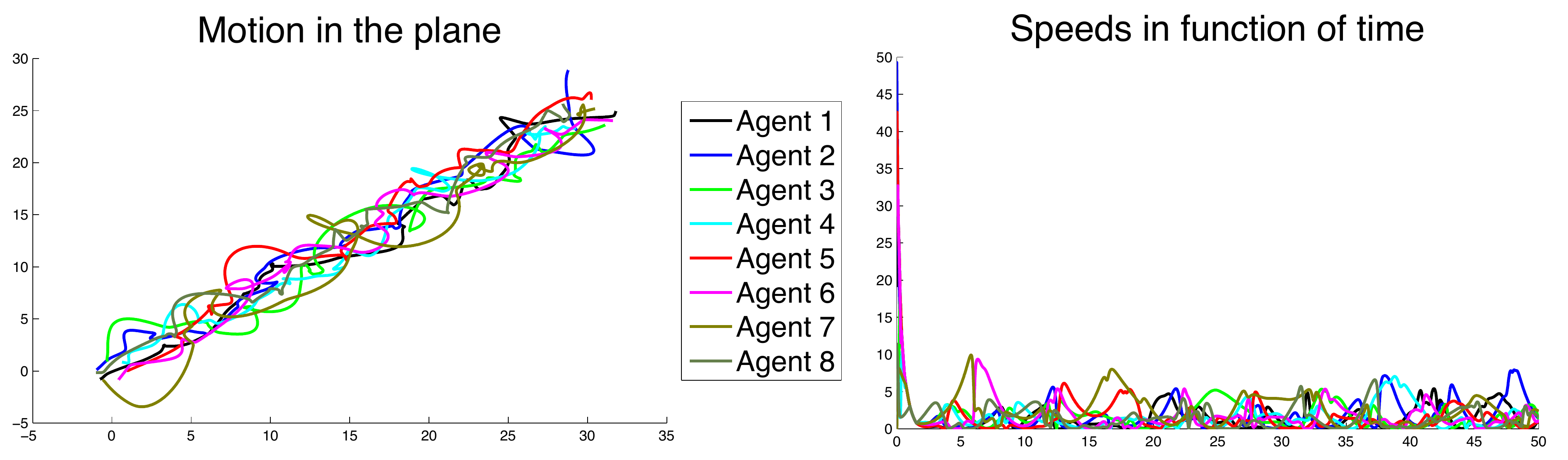}
\caption{Space evolution and speeds of the controlled system.}
\label{fig:contr1}
\end{figure}

The spatial evolution graph shows a braid movement which resembles a pattern near to flocking as it is commonly interpreted. The action of our control is evident from the energy profile of the system, portrayed in Figure \ref{fig:contr1_energy}, where the total energy is driven below the threshold in a very short time. The fall of the total energy is mainly due to its kinetic part (the green line), which is the only one directly affected by our control strategy. The sharp decrease of the kinetic energy is also witnessed in the graph showing the modulus of the speeds, where, after a quick, strong brake at the beginning, they stabilize at a very low level.

\begin{figure}[!htb]
\centering
\includegraphics[width=.67\linewidth]{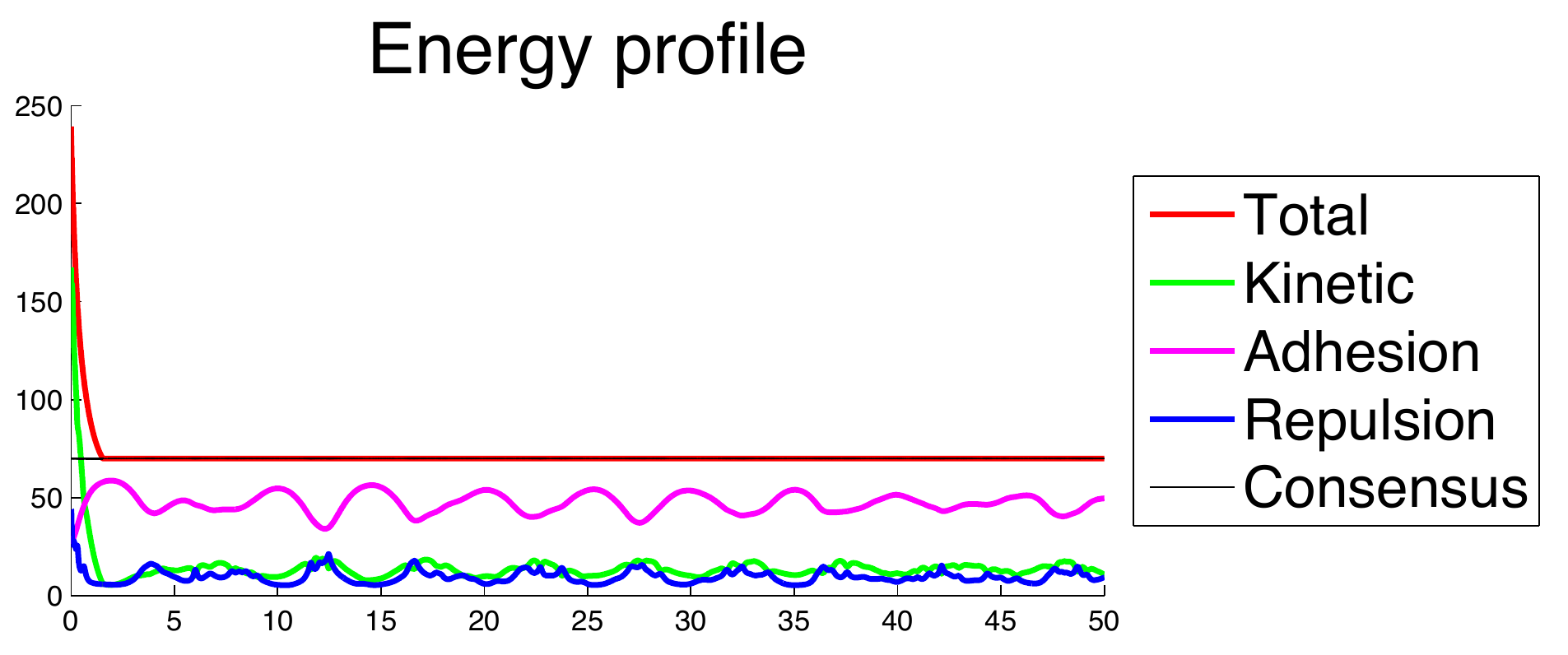}
\caption{Energy profile of the controlled system.}
\label{fig:contr1_energy}
\end{figure}

%\subsection{The case of $\beta = 1.02$ and $p = 1.1$}
\subsubsection{Tuning the parameter $M$}

The second case study takes into account a system with a weaker communication rate than before ($\beta = 1.02$) and with a different form of the repulsive function ($p = 1.1$), and we apply on it our control strategy with several values for $M$.

\begin{figure}[!htb]
\centering
\includegraphics[width=\linewidth]{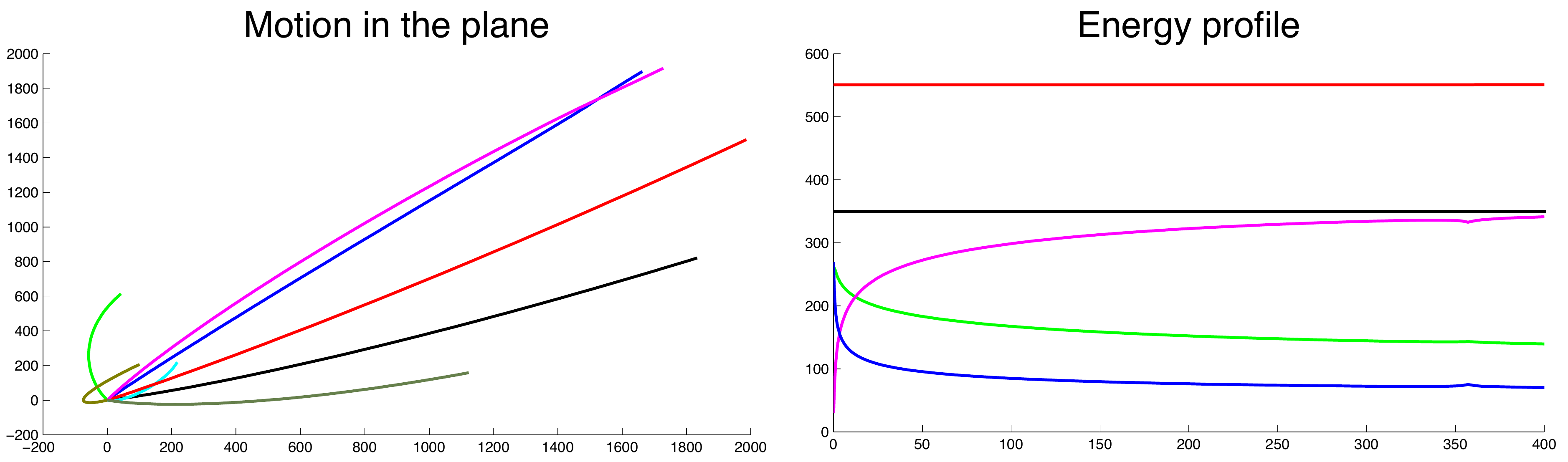}\\
\includegraphics[width=\linewidth]{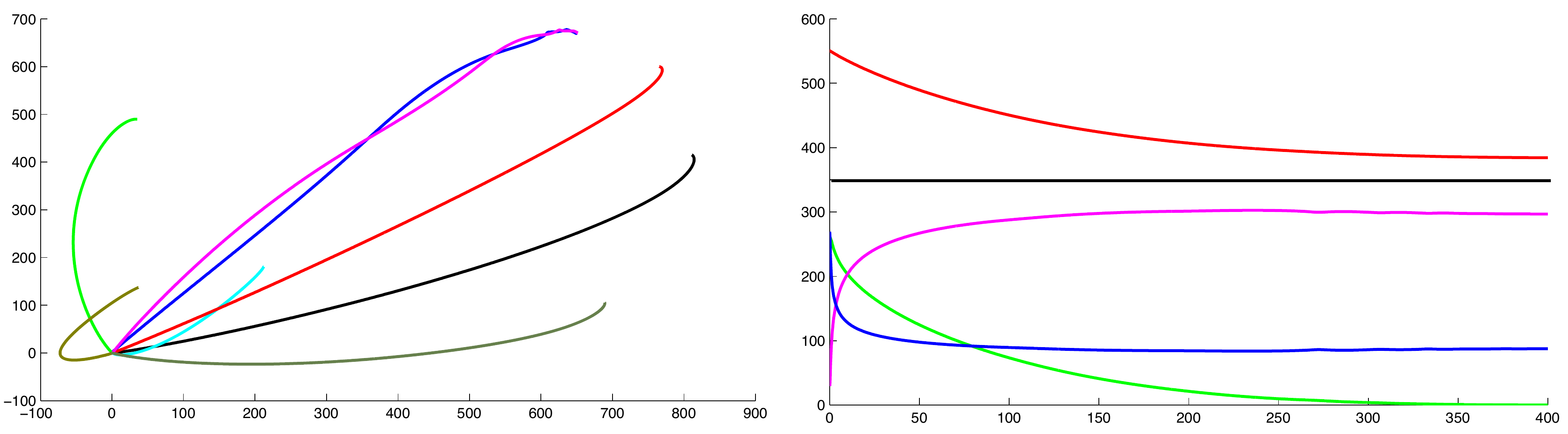}\\
\includegraphics[width=\linewidth]{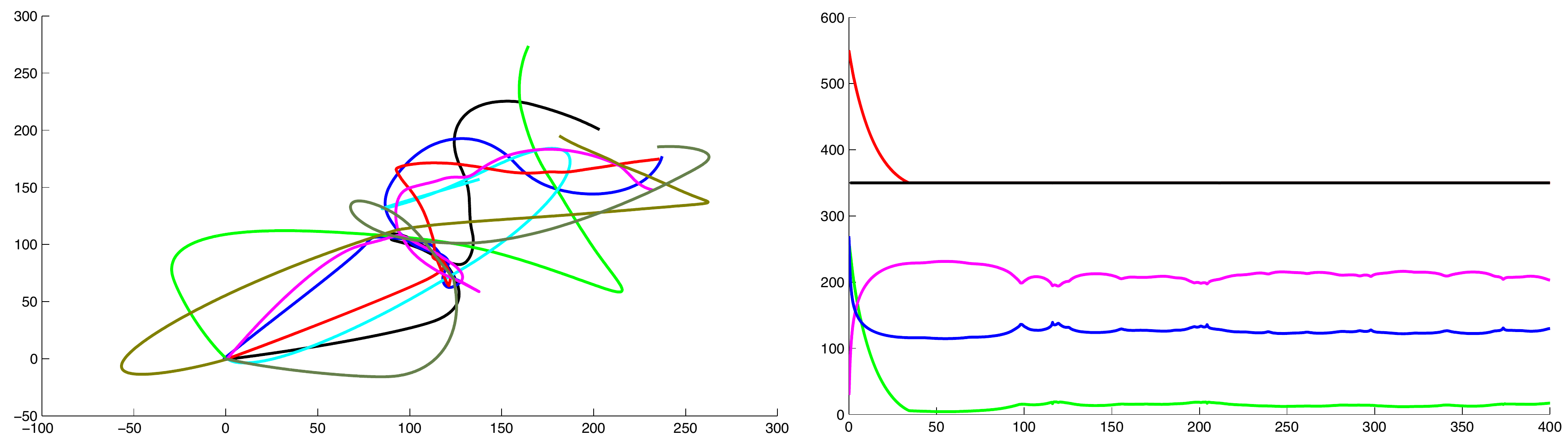}\\
\includegraphics[width=\linewidth]{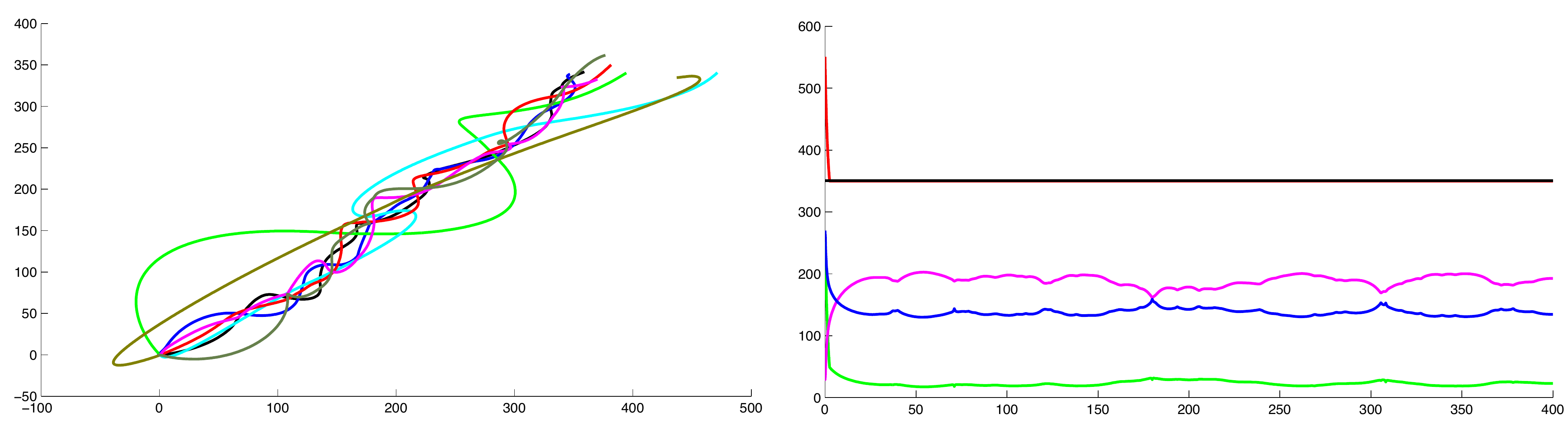}
\caption{Spatial evolutions (left) and relative energy profiles (right). From top to bottom: $M = 0, M = 0.1, M = 1, M = 10$. The colors in the left column stand for: total energy (red), consensus region (black), adhesion energy (magenta), repulsion energy (blue), kinetic energy (green).}
\label{fig:contr2pos}
\end{figure}
The top-left corner of Figure \ref{fig:contr2pos} is the uncontrolled system: it seems legitimate to suppose that it is very unlikely that the system will converge to consensus, especially looking at its energy profile graph (top-right corner of Figure \ref{fig:contr2pos}), which shows an increase in the adhesion potential energy, phenomenon associated to an increase in the distance between particles, as already pointed out. In the second line we see the spatial evolution graph of the same system but with the sparse control strategy acting with parameter $M = 0.1$, where the agents are starting to converge to consensus, as is also evident in their energy profile.
The two bottom lines of Figure \ref{fig:contr2pos} display the action of controls with $M = 1$ and $M = 10$, respectively. It is clear how the situation goes better as $M$ increases, which is due to the fact that the threshold is reached in shorter time (see the relative energy profile).

The right column of Figure \ref{fig:contr2pos} also clearly confirms the behavior of the decay rate of the energy as a function of $M$, as predicted by our analysis: $E(t)$ decreases as $e^{-kMt}$, for a certain constant $k>0$.

It is interesting to notice that convergence to the consensus region occurs even if the hypothesis $(b)$ of Theorem \ref{th:mainresult} is not met, i.e., $\vartheta$ is very far away from $E(0)$, as it is likely to be a sub-optimal sufficient condition. Indeed, in all the case studies above
\begin{align*}
	c = \exp \left(-\frac{2\sqrt{3}}{9}\frac{M \vnorm{\overline{v}(0)}^3}{E(0) \sqrt{E(0)} \left(\Lambda \sqrt{E(0)} + \frac{M}{N}\right)} \right) \approx 1,
\end{align*}
but, nonetheless, we were able to steer the system to consensus in finite time.

\subsubsection{A counterexample to unconditional sparse controllability}\label{sec:counter}

The last numerical experiment we report shows that in certain pathological situations the sparse control strategy can fail to steer a Cucker-Dong systems to consensus.

We consider $N = 2$ agents in dimension $d = 2$ and choose the interaction parameters as $H = 1$, $\beta = 2$, $p = 1.1$, $\Lambda = 0$, and $M = 1$. In this situation, the force balance $f - a$ is completely in favor of the repulsive force, as Figure \ref{fig:wrongbalance} shows: this means that, regardless of the mutual positions of the agents, they shall always be repelled from each other.

\begin{figure}[!h]
\centering
\includegraphics[width=.57\linewidth]{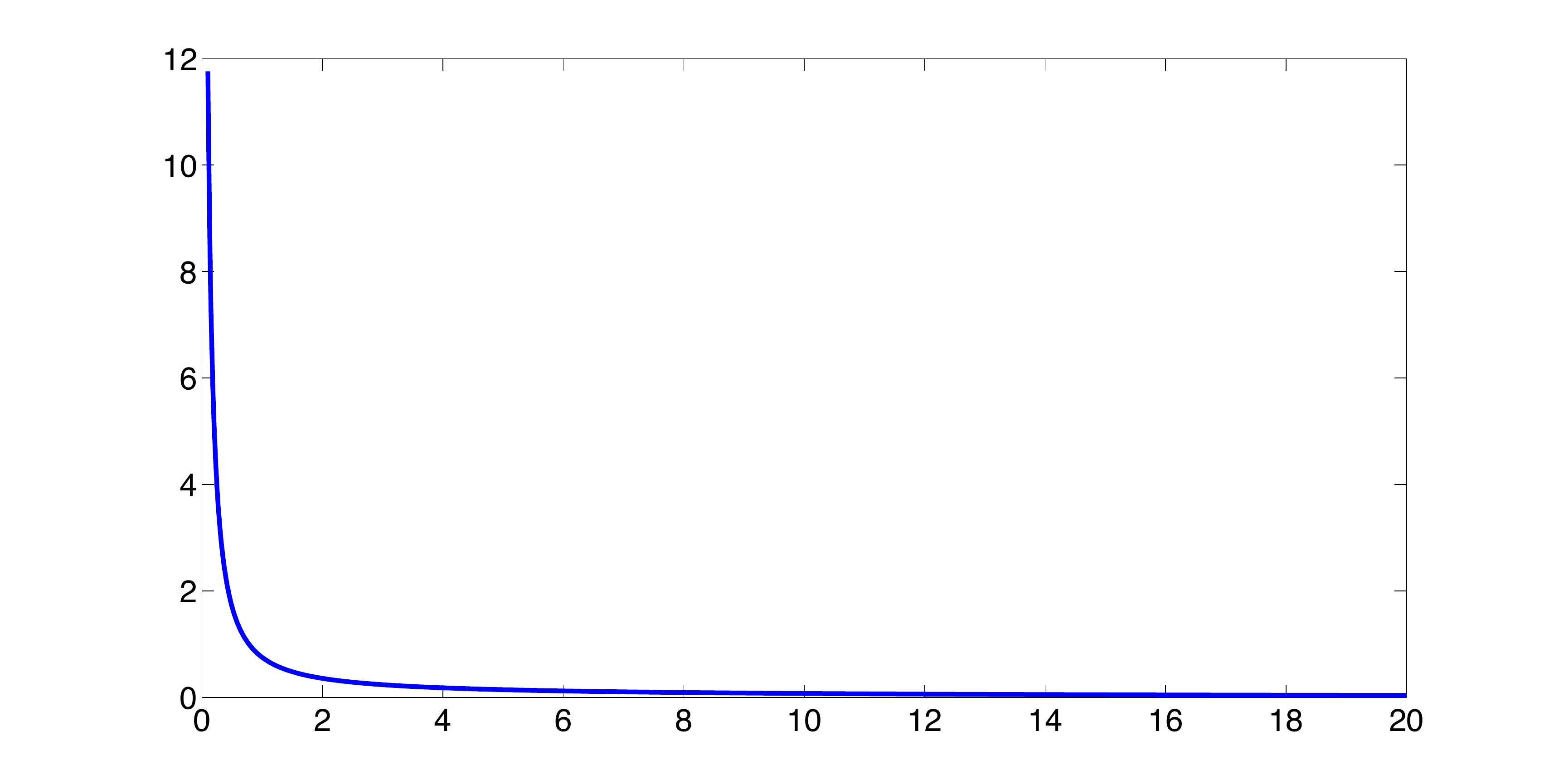}
\caption{Sum of the attraction and repulsion forces $h(r) = f(r)-a(r)$ as a function of the distance $r>0$ in the case study of Section \ref{sec:counter}}
\label{fig:wrongbalance}
\end{figure}

If we exert the sparse control strategy, the only result that we obtain is to freeze the agents where they are. Indeed, Figure \ref{fig:counter} shows that the agents' speeds are rapidly reduced to values close to 0 as an effect of the control (also visible in the energy profile from the trajectory of the kinetic energy), but the total energy stays far away from the consensus region (the black line). The picture makes very clear that the sparse feedback control does not affect the potential energy of the system, as the sum of the adhesion and repulsion energies stays constant in time. %(the adhesion energy line is almost totally juxtaposed to the consensus line, as the theory prescribes in the case of $N = 2$ agents).

\begin{figure}[!h]
\centering
\includegraphics[width=\linewidth]{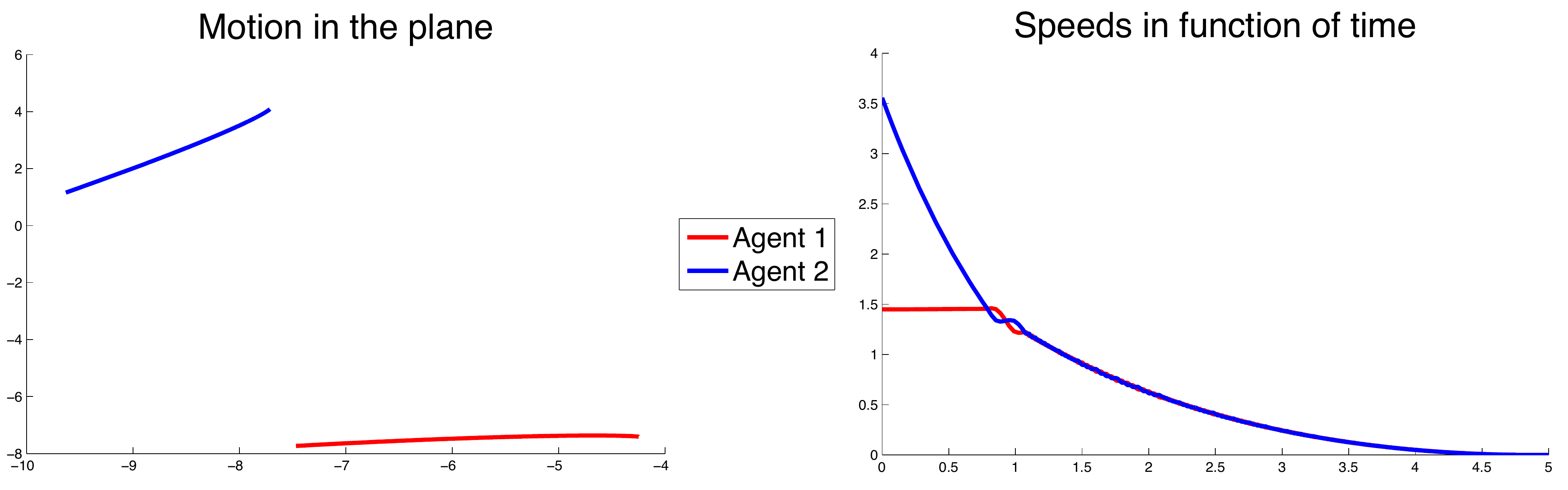}\\
\includegraphics[width=.63\linewidth]{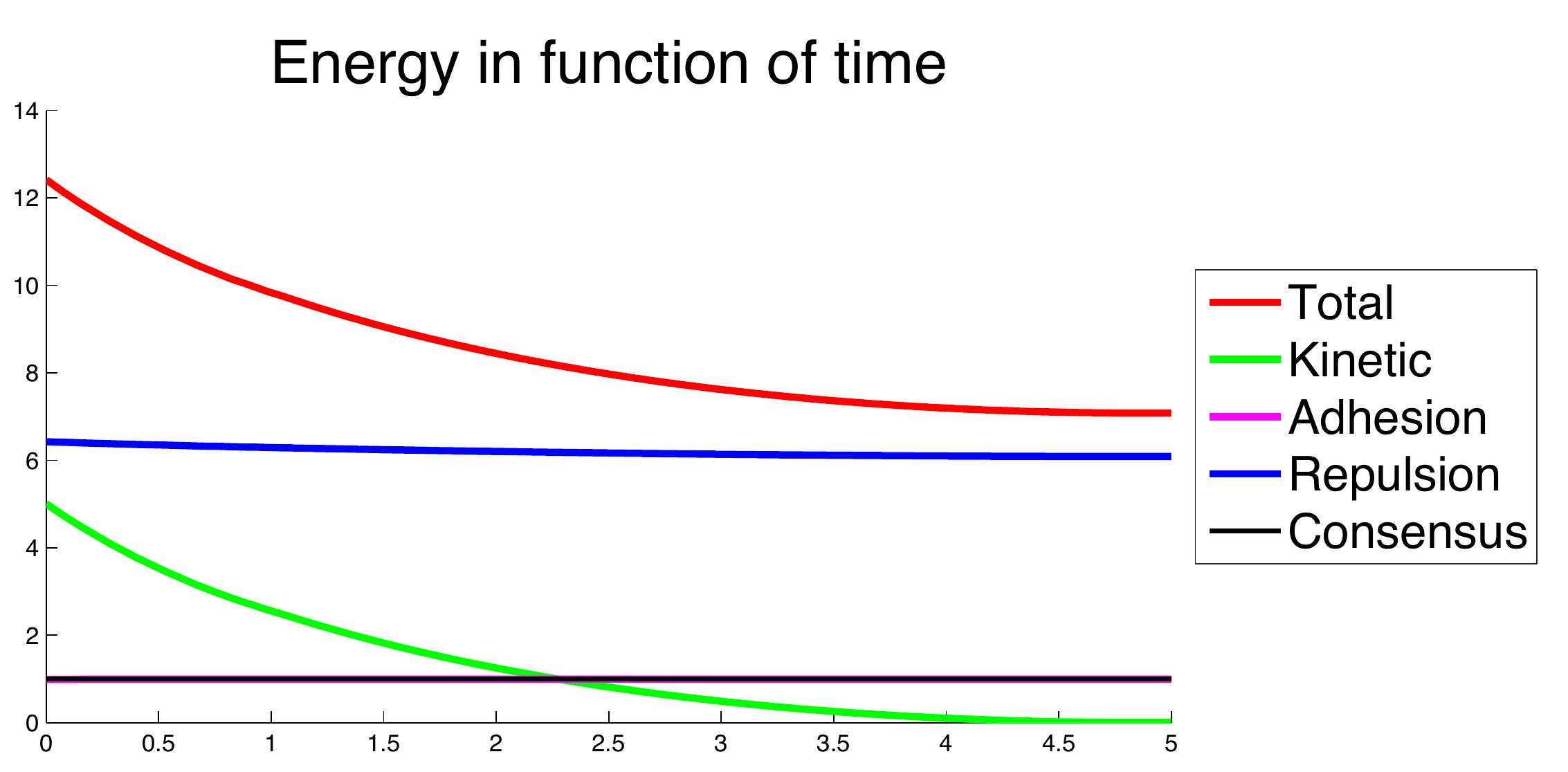}
\caption{Space evolution, speeds and energy profile of the system considered in Section \ref{sec:counter}.}
\label{fig:counter}
\end{figure}

Furthermore, notice that, as soon as we shut down the control, the two agents will start to move again in opposite directions (very slowly, since the energy of a system without control stays constant), hence not only the total energy remains above the threshold, but also the system is not in consensus.

However, it must be observed that the control strategy fails in this situation due to the peculiar nature of the system. As a matter of fact, being the force balance strictly repulsive, the trajectories of any solution will never remain cohesive. This leaves open the question whether there exist ``non-pathological'' instances of the Cucker-Dong model (in the sense that their solutions are not doomed to diverge regardless of the initial condition) for which the sparse control strategy does not work.

\begin{acknowledgement}
The authors acknowledge the support of the ERC-Starting Grant ``High-Dimensional Sparse Optimal Control'' (HDSPCONTR - 306274).
\end{acknowledgement}

\bibliographystyle{abbrv}	% (uses file "plain.bst")
\bibliography{Literature}

\end{document}